\documentclass[a4paper,10pt]{amsart}

\usepackage[utf8]{inputenc} 
\usepackage[T1]{fontenc}
\usepackage{mathtools,amssymb,amsthm} 
\usepackage{dsfont}
\usepackage[all]{xy}
\usepackage{xparse} 
\usepackage{microtype}
\usepackage{mathrsfs}
\usepackage{varioref}
\usepackage{longtable}
\usepackage{hyperref}

\input{macros.tex}

\title[Unipotent $\ell$-blocks for simply-connected $p$-adic groups]{Unipotent $\ell$-blocks for simply-connected $p$-adic groups}
\author{Thomas Lanard}
\email{thomas.lanard@univie.ac.at}

\begin{document}

\begin{abstract}
    Let $\kk$ be a non-archimedean local field and $G$ the $\kk$-points of a connected simply-connected reductive group over $\kk$. In this paper, we study the unipotent $\lprime$-blocks of $G$, for $\lprime \neq p$. To that end, we introduce the notion of $\dun$-series for finite reductive groups. These series form a partition of the irreducible representations and are defined using Harish-Chandra theory and $d$-Harish-Chandra theory. The $\lprime$-blocks are then constructed using these $\dun$-series, with $d$ the order of $q$ modulo $\lprime$, and consistent systems of idempotents on the Bruhat-Tits building of $G$. We also describe the stable $\lprime$-block decomposition of the depth zero category of an unramified classical group.

\end{abstract}

\maketitle

\section*{Introduction}

Let $\kk$ be a non-archimedean local field and $\res$ its residue field. Let $q$ be the cardinal of $\res$ and $p$ its characteristic. Let $\gpalg{G}$ be a connected reductive group over $\kk$ and denote by $G:=\gpalg{\G}[\kk]$ the $\kk$-points of $\gpalg{G}$.

\bigskip

Let $\rep[\mathbb{C}]{G}$ be the category of smooth representations of $G$ with complex coefficients. One way to study this category is to decompose it in a minimal product of subcategories, called blocks, and describe them. This problem was solved by Bernstein in \cite{bern} who describes the blocks with inertial classes of cuspidal support.

Congruences between automorphic forms were used to solve remarkable problems of arithmetic-geometry. Hence, it becomes natural to study the smooth representations of $p$-adic groups with coefficients in $\Zl$, for $\lprime$ a prime number different from $p$. In the same way, we would like to have a decomposition of their category $\rep[\Zl]{G}$ into $\lprime$-blocks. However, we do not have a result like the Bernstein decomposition, for the $\lprime$-blocks. A decomposition of $\rep[\Fl]{\gl{n}(\kk)}$ into blocks was proved by Vignéras in \cite{vigneras}(see also the work of Sécherre and Stevens \cite{SecherreStevens} for inner forms of $\gl{n}(\kk)$). After that, Helm reached in \cite{helm} a decomposition into $\lprime$-blocks of $\rep[\Zl]{\gl{n}(\kk)}$. He describes these $\lprime$-blocks with the notion of mod $\lprime$ inertial supercuspidal support. Apart from $\gl{n}$ and its inner forms, we don't know much about the $\lprime$-blocks.

The decomposition of Bernstein and Vignéras-Helm both use the ``unicity of the supercuspidal support'', which is true for $\gl{n}$ and in the complex case, but not in general. Therefore, a new strategy to study the $\lprime$-blocks is needed. A new method, using consistent systems of idempotents on the Bruhat-Tits building, was used in \cite{dat_equivalences_2014} to construct in depth zero the $\lprime$-blocks for $\gl{n}$. Then, this was used in \cite{lanard} and \cite{lanard2} to obtain decompositions of the depth zero category over $\Zl$, for a group which is split over an unramified extension of $\kk$. These decompositions are constructed using Deligne-Lusztig theory. They present a lot of interesting properties and links with the local Langlands correspondence, but they are not blocks in general, just unions of blocks.

\bigskip

In this paper, we deal with two problems, the study of the unipotent $\lprime$-blocks and the stable $\lprime$-blocks for unramified classical groups.

\bigskip

Let us start by the unipotent $\lprime$-blocks. Let $\repun{\Ql}{G}$ be the subcategory of unipotent representations. Using \cite{lanard2} (with the system of conjugacy classes composed of the trivial representation for every polysimplex), we also get a $\lprime$-unipotent category over $\Zl$ : $\repun{\Zl}{G}$. The unipotent $\lprime$-blocks are the $\lprime$-blocks of $\repun{\Zl}{G}$.

In \cite{lanard2}, the idempotents are constructed using Deligne-Lusztig theory. A first difficulty for an $\lprime$-block decomposition is that Deligne-Lusztig theory does not produce primitives idempotents. Moreover, replacing naively Deligne-Lusztig idempotents by primitive central ones won't produce consistent systems of idempotents for the $p$-adic group. This is why we introduce for $\gpfini{G}$, a finite reductive group over $\res$, the notion of a $\dun$-series. A $\dun$-series will be a minimal set of irreducible characters with the property that it is a union of Harish-Chandra series (in order to get $p$-adic blocks) and that the idempotent associated has integer coefficients (to get a decomposition over $\Zl$).

\bigskip

Let $(\gpfinialg{G}, \fr)$ be a connected reductive group over $\res$. The $\lprime$-blocks of $\gpfinialg{G}^{\fr}$ are then described using $d$-cuspidal pairs, see \cite{bmm} and \cite{CabanesEnguehardBlocks}. For an integer $d$, a $d$-split Levi subgroup is the centralizer of a $\fr$ stable torus $\gpfinialg{G}$, such that the cardinal of $\gpfinialg{T}^{\fr}$ is a power of $\cycl{d}(q)$, where $\cycl{d}$ is the $d$-th cyclotomic polynomial. The usual Harish-Chandra induction and restriction is then replaced by the Deligne-Lusztig induction and restriction from these $d$-split Levi subgroups. An irreducible character $\chi$ is said to be $d$-cuspidal if and only if $\rdl \chi = 0$ for every proper $d$-split Levi subgroup $\gpfinialg{L}$ and every parabolic $\gpfinialg{P}$ admitting $\gpfinialg{L}$ as Levi subgroup. Let $d$ be the order of $q$ modulo $\lprime$. Then we get a bijection (with some restrictions on $\lprime$) between conjugacy classes of pairs $(\gpfinialg{M},\chi)$, consisting of a $d$-split Levi $\gpfinialg{M}$ and a $d$-cuspidal character of $\gpfinialg{M}^{\fr}$, and $\lprime$-blocks of $\gpfinialg{G}^{\fr}$.

We define a $\dun$-set to be a subset of $\Irr(\gpfinialg{G}^{\fr})$ which is both a union of Harish-Chandra series and of $d$-series (that is a set of characters having the same $d$-cuspidal support). A $\dun$-series is then a $\dun$-set with no proper non-empty $\dun$-subset. In Theorem \ref{theresumdun}, we completely compute the unipotent $\dun$-series of $\gpfinialg{G}^{\fr}$.

\bigskip

Let $\bt$ be the semi-simple Bruhat-Tits building associated to $G$. For $\sigma \in \bt$, we denote by $\quotred{G}{\sigma}$ the reductive quotient of $G$ at $\sigma$, which is a connected reductive group over $\res$. Let $\TG$ be the set of $G$-conjugacy classes of pairs $(\sigma,\pi)$, where $\sigma \in \bt$ and $\pi$ is an irreducible cuspidal representation of $\quotred{G}{\sigma}$. The work of Morris in \cite{morris} shows that to an element $\mathfrak{t} \in \TG$ we can associate $\rep[\Ql][\mathfrak{t}]{G}$, a union of blocks of depth zero. We define an equivalence relation $\sim$ on $\TG$ (see Section \ref{secidempotenthc} for more details) such that $\rep[\Ql][\mathfrak{t}]{G}=\rep[\Ql][\mathfrak{t}']{G}$ if and only if $\mathfrak{t} \sim \mathfrak{t}'$. Denote by $[\mathfrak{t}]$ the equivalence class of $\mathfrak{t}$. Hence, we get a decomposition of the depth zero category
\[ \rep[\Ql][0]{G}=\prod_{[\mathfrak{t}] \in \TG/{\sim}} \rep[\Ql][[\mathfrak{t}]]{G}.\]
Moreover, when $G$ is semisimple and simply-connected, the categories $\rep[\Ql][[\mathfrak{t}]]{G}$ are blocks.

We also denote by $\TGu$ the subset of $\TG$ of pairs $(\sigma,\pi)$ with $\pi$ unipotent, and $\TGl$ the subset of $\TG$ of pairs $(\sigma,\pi)$ with $\pi$ in a Deligne-Lusztig series associated with a semi-simple conjugacy class in $\quotred*{G}{\sigma}$ of order a power of $\lprime$. The equivalence relation $\sim$ is trivial on $\TGu$ (see remark \ref{remSimUnTriv}). We have $\repun{\Ql}{G} = \prod_{\mathfrak{t} \in \TGu} \rep[\Ql][[\mathfrak{t}]]{G}$ and $\repun{\Zl}{G} \cap \rep[\Ql]{G} = \prod_{[\mathfrak{t}] \in \TGl/{\sim}} \rep[\Ql][[\mathfrak{t}]]{G}$ (see the remark below for the definition of the intersection).

\begin{Rem}
    Let $B$ be a direct factor subcategory of $\rep[\Zl]{G}$ and $e \in \mathcal{Z}_{\Zl}(G)$ be the corresponding idempotent in the centre of $\rep[\Zl]{G}$. We then denote by $B \cap \rep[\Ql]{G}$ the direct factor of $\rep[\Ql]{G}$ cut out by $e \in \mathcal{Z}_{\Zl}(G) \subseteq \mathcal{Z}_{\Ql}(G)$.
\end{Rem}

Now, let us come back to the $\lprime$-block.

\begin{The*}
    Let $\lprime$ be a prime different from $p$. Assume that $G$ is semisimple and simply-connected. Let $R$ be an $\lprime$-block of $\repun{\Zl}{G}$. Then $R$ is characterized by the non-empty intersection $R \cap \repun{\Ql}{G}$.
\end{The*}

Thus, we need to describe the intersection of the $\lprime$-blocks and the unipotent category. To achieve that, we define an equivalence relation on $\TGu$ in the following way. Let $d$ be the order of $q$ modulo $\lprime$. Let $\mathfrak{t}$ and $\mathfrak{t}'$ be two elements of $\TGu$ and $\omega \in \bt$. Then we say that $\mathfrak{t} \simx{\omega} \mathfrak{t}'$ if and only if $\mathfrak{t}=\mathfrak{t'}$ or there exist $(\sigma,\pi)$ and $(\tau,\pi')$ such that $\mathfrak{t}=[\sigma,\pi]$, $\mathfrak{t}'=[\tau,\pi']$, $\omega$ is a face of $\sigma$ and $\tau$, and the Harish-Chandra series in $\quotred{G}{\omega}$ corresponding to the cuspidal pairs $(\quotred{G}{\sigma},\pi)$ and $(\quotred{G}{\tau},\pi')$ are both contained in the same $\dun$-series. Note that by our computation of the $\dun$-series, for $\mathfrak{t}$ and $\omega$ fixed, we know explicitly the set of $\mathfrak{t}'\in\TGu$ such that  $\mathfrak{t} \simx{\omega} \mathfrak{t}'$. Now we define $\siml$, an equivalence relation on $\TGu$ by $\mathfrak{t} \siml \mathfrak{t}'$ if and only if there exist $\omega_1,\cdots,\omega_r \in \bt$ and $\mathfrak{t}_{1}, \cdots, \mathfrak{t}_{r-1} \in \TGu$ such that $\mathfrak{t} \simx{\omega_1} \mathfrak{t}_1\simx{\omega_2} \mathfrak{t}_2 \cdots \simx{\omega_r} \mathfrak{t}'$. We write $\eql{\mathfrak{t}}$ for the equivalence class of $\mathfrak{t}$.

\begin{The*}
    Let $\lprime$ be an odd prime number, different from $p$, such that $\lprime \geq 5$ if a group of exceptional type ($\tDq$, $\dG$, $\Fqu$, $\Esi$, $\dEsi$, $\Ese$) is involved in a reductive quotient and $\lprime \geq 7$ if $\Eh$ is involved in a reductive quotient. To each equivalence class $\eql{\mathfrak{t}} \in \TGu/{\siml}$, we can associate $\rep[\Zl][\eql{\mathfrak{t}}]{G}$ a Serre subcategory of $\repun{\Zl}{G}$, constructed with a consistent system of idempotents such that
    \begin{enumerate}
        \item We have a decomposition
              \[ \repun{\Zl}{G}= \prod_{\eql{\mathfrak{t}} \in \TGu/{\siml}}  \rep[\Zl][\eql{\mathfrak{t}}]{G}.\]
        \item $\rep[\Zl][\eql{\mathfrak{t}}]{G} \cap \repun{\Ql}{G} = \prod_{\mathfrak{u} \in \eql{\mathfrak{t}}} \rep[\Ql][\mathfrak{u}]{G} $

        \item We also have a description of $\rep[\Zl][\eql{\mathfrak{t}}]{G} \cap \rep[\Ql]{G}$.
              Let $(\sigma,\chi) \in \TGl$. Let $t$ be a semi-simple conjugacy class in $\quotred*{G}{\sigma}$ of order a power of $\lprime$, such that $\chi$ is in the Deligne-Lusztig series associated to $t$. Let $\quotred{G}{\sigma}(t)$ be a Levi in $\quotred{G}{\sigma}$ dual to $\cent*{t}{\quotred*{G}{\sigma}}$, the connected centralizer of $t$, $\gpfini{P}$ be a parabolic subgroup with Levi component $\quotred{G}{\sigma}(t)$, $\hat{t}$ be a linear character of $\quotred{G}{\sigma}(t)$ associated to $t$ by duality, and $\chi_t$ be a unipotent character in $\quotred{G}{\sigma}(t)$ such that $ \langle \chi, \inddllite{\quotred{G}{\sigma}(t) \subseteq \gpfini{P}}{\quotred{G}{\sigma}}(\hat{t}\chi_t) \rangle \neq 0$. Let $\pi$ be an irreducible component of $\inddllite{\quotred{G}{\sigma}(t) \subseteq \gpfini{P}}{\quotred{G}{\sigma}}(\chi_t)$. Let $(\quotred{G}{\tau},\lambda)$ be the cuspidal support of $\pi$. Then
              \[ \rep[\Ql][(\sigma,\chi)]{G} \subseteq \rep[\Zl][\eql{(\tau,\lambda)}]{G} \cap \rep[\Ql]{G} \]

        \item  When $G$ is semisimple and simply-connected, the categories $\rep[\Zl][\eql{\mathfrak{t}}]{G}$ are $\lprime$-blocks.
    \end{enumerate}

\end{The*}

We also obtain results for the bad prime $\lprime=2$ in some special cases (which include classical groups).

\begin{The*}
    Let $G$ be a semisimple and simply-connected group such that all the reductive quotients only involve types among $\Aa$, $\Bb$, $\Cc$ and $\Dd$, and $p\neq 2$. Then $\rep[\overline{\mathbb{Z}}_{2}][1]{G}$ is a $2$-block.
\end{The*}

As mentioned before, we can compute explicitly the equivalence relation $\simx{\omega}$, so we can also know $\siml$. We work out a few examples here, where we make $\siml$ explicit, hence also the unipotent $\lprime$-blocks.

\begin{The*}
    Let $G$ be a semisimple and simply-connected group.
    \begin{enumerate}
        \item If $\lprime$ is banal (see Definition \ref{deflbanal}), then the unipotent $\lprime$-blocks are indexed by $\TGu$.
        \item If $\lprime$ divides $q-1$ and satisfies the conditions of the previous theorem, then $\siml$ is the trivial relation and the unipotent $\lprime$-blocks are indexed by $\TGu$. Moreover, the intersection of an $\lprime$-block with $\repun{\Ql}{G}$ is a Bernstein block.
        \item If $G=\sl{n}(\kk)$ then $\repun{\Zl}{\sl{n}(\kk)}$ is an $\lprime$-block.
    \end{enumerate}
\end{The*}

We also work out the case $G=\sp{2n}(\kk)$, but to do that we require a few more notations.

Let $\TGs:=\{(s,s')\in \mathbb{N}^2, s(s+1)+s'(s'+1) \leq n\}$. To $(s,s') \in \TGs$ we can associate $\mathfrak{t}(s,s')=(\sigma(s,s'),\pi(s,s')) \in \TGu$, such that the reductive quotient at $\sigma(s,s')$ is $\gl{1}(\res)^{n-s(s+1)+s'(s'+1)}\times \sp{2s(s+1)}(\res)\times\sp{2s'(s'+1)}(\res)$ and $\pi(s,s')$ is the unique unipotent irreducible cuspidal representation in this group. The map $(s,s')\mapsto \mathfrak{t}(s,s')$ gives a bijection between $\TGs$ and $\TGu$. Also denote by $\Scusp$ the set
\[\Scusp=\{(s,s') \in \TGs, \left\{
    \begin{array}{ll}
        s(s+1)+s'(s'-1) > n-d/2 \\
        s'(s'+1)+s(s-1) > n-d/2
    \end{array}
    \right\}\}.\]

Putting together the previous theorems and making the equivalence relation explicit, we obtain the following description of the unipotent $\ell$-blocks of $\sp{2n}(\kk)$.

\begin{The*}
    Let $\lprime$ be prime not dividing $q$.
    \begin{enumerate}
        \item If $\lprime=2$: $\rep[\overline{\mathbb{Z}}_{2}][1]{\sp{2n}(\kk)}$ is a $2$-block.
        \item If $\lprime \neq 2$. Let $d$ be the order of $q$ modulo $\lprime$.
              \begin{enumerate}
                  \item if $d$ is odd, $\siml$ is the trivial equivalence relation giving the following decomposition into $\lprime$-blocks
                        \[\repun{\Zl}{\sp{2n}(\kk)}= \prod_{\mathfrak{t}\in \TGu} \rep[\Zl][\eql{\mathfrak{t}}]{\sp{2n}(\kk)}.\]
                  \item if $d$ is even, the equivalence classes of $\siml$ are the singletons $\{\mathfrak{t}(s,s')\}$ for $(s,s') \in \Scusp$ and $\{\mathfrak{t}(s,s'), (s,s') \in \TGs \setminus \Scusp\}$ thus giving the $\lprime$-block decomposition
                        \[\repun{\Zl}{\sp{2n}(\kk)}=\rep[\Zl][\eql{\mathfrak{t}(0,0)}]{\sp{2n}(\kk)} \times \prod_{(s,s')\in \Scusp} \rep[\Zl][\eql{\mathfrak{t}(s,s')}]{\sp{2n}(\kk)}.\]

              \end{enumerate}
    \end{enumerate}
\end{The*}

\begin{Rem*}
    \begin{enumerate}
        \item In the case $d$ odd, or $d$ even and $(s,s')\in \Scusp$, we see that the intersection of an $\lprime$-block with $\repun{\Ql}{G}$ is a Bernstein block.
        \item If $\lprime > n $, in the case $d$ even and $(s,s')\in \Scusp$, then $\rep[\Zl][\eql{\mathfrak{t}(s,s')}]{\sp{2n}(\kk)} \cap \rep[\Ql]{G}$ is a Bernstein block.
    \end{enumerate}

\end{Rem*}

\bigskip

Let us now turn to the study of the stable $\lprime$-blocks. Let $G$ be a classical unramified group. In this case we have the local Langlands correspondence (\cite{HarrisTaylor} \cite{henniart} \cite{arthur} \cite{mok} \cite{KMSW}). The block decomposition is not compatible with the local Langlands correspondence, two irreducible representations may have the same Langlands parameter but may not be in the same block. However, we can look for the ``stable'' blocks, which are the smallest direct factors subcategories stable by the local Langlands correspondence. These categories correspond to the primitive idempotents in the stable Bernstein centre, as defined in \cite{haines}. In \cite{lanard2}, the decomposition into stable blocks of the depth zero category is given by
\[\rep[\Ql][0]{\G} = \prod_{(\phi, \sigma) \in \Lpbm{\iner^{\Ql}}} \rep[\Ql][(\phi,\sigma)]{\G}\]
where the set $\Lpbm{{\iner^{\Ql}}}$ is defined in \cite[Def. 4.4.2]{lanard2}. An analogous decomposition is given over $\Zl$ and we prove here that this is the stable $\lprime$-block decomposition.

\begin{The*}
    Let $G$ be an unramified classical group and $p \neq 2$. Then the decomposition  of \cite{lanard2}
    \[ \rep[\Zl][0]{G} = \prod_{(\phi, \sigma) \in \Lpbm{\iner^{\Zl}}} \rep[\Zl][(\phi,\sigma)]{G}.\]
    is the decomposition of $\rep[\Zl][0]{\G}$ into stable $\lprime$-blocks, that is, these categories correspond to primitive integral idempotent in the stable Bernstein centre.
\end{The*}

\paragraph{\textbf{Acknowledgements}}
I would like to thank Jean-François Dat for the comments and remarks that made this article a better one.

\tableofcontents

\section{Notations}

Let $\kk$ be a non-archimedean local field and $\res$ its residue field. Let $q$ be the cardinal of $\res$ and $p$ its characteristic.

\bigskip

In this paper, we will be interested in reductive groups over $\kk$ and over $\res$. In order not to confuse the two settings, we will use the font $\gpalg{G}$ for a connected reductive group over $\kk$ and $\gpfinialg{G}$ for a connected reductive group over $\res$.

\bigskip

Let $\gpalg{G}$ be a connected reductive group over $\kk$. We denote by $G:=\gpalg{\G}[\kk]$ the $\kk$-points of $\gpalg{G}$. If $\Lambda$ is a ring where $p$ is invertible, then we will write $\rep[\Lambda]{G}$ for the abelian category of smooth representations of $G$ with coefficients in $\Lambda$. The full subcategory of representations of depth zero will be denoted by $\rep[\Lambda][0]{G}$ (see Definition \ref{defDepthZero}).

\bigskip

In the same way, if $\gpfinialg{G}$ is a connected reductive group over $\res$, we denote by $\gpfini{G}:=\gpfinialg{G}[\res]$ the group of its $\res$-points. This group can be seen as $\gpfini{G}:=\gpfinialg{G}[\resalg]^{\fr}$, the group of fixed points of a Frobenius automorphism $\fr$. If $\gpfinialg{P}$ is a parabolic  subgroup admitting $\gpfinialg{M}$ a $\fr$-stable Levi subgroup, we will write $\inddl{M}{P}{G}$ for the Deligne-Lusztig induction from $\gpfini{M}$ to $\gpfini{G}$ (defined in \cite{DeligneLusztig}). It is a map between spaces of virtual representations $\inddl{M}{P}{G} : \mathbb{Z}\Irr(\gpfini{M})\to \mathbb{Z}\Irr(\gpfini{G})$. When $\gpfinialg{P}$ is also  $\fr$-stable, since the Deligne-Lusztig induction is the same as the Harish-Chandra induction, we will also use $\indPara{M}{P}{G}$ and $\resPara{M}{P}{G}$ for the Harish-Chandra induction and restriction. Let $\gpfinialg*{G}$ be in duality with $\gpfinialg{G}$, a duality defined over $\res$, with Frobenius $\fr$ on $\gpfinialg*{G}$.

\bigskip

In all this paper, $\lprime$ will be a prime number not dividing $q$. We shall assume that choices have been made, once and for all, of isomorphisms of $\resalg^{*}$ with $(\mathbb{Q}/\mathbb{Z})_{p'}$ and of $\resalg^{*}$ with the group of roots of unity of order prime to $p$ in $\Ql$.

\section{Bernstein blocks}
\label{secBernsblock}

Let $G$ be the $\kk$-points of a connected reductive group. When the field of coefficients is $\Ql$ (or $\mathbb{C}$), the blocks of $G$ are well known thanks to the theory of Bernstein \cite{bern}. In this paper, the $\lprime$-blocks of $G$ will be constructed using consistent systems of idempotents on the Bruhat-Tits building of $G$. The purpose of this section is to explain, in the case where $G$ is semisimple and simply-connected, how we can recover Bernstein blocks using consistent systems of idempotents.

\subsection{Consistent systems of idempotents}

\label{secsystcohe}
In this section, we recall the basic definitions and properties of systems of idempotents.

\bigskip

Let $\bt$ be the semi-simple Bruhat-Tits building associated to $G$. This is a polysimplicial complex and we denote by $\bts$ the set of vertices, that is of polysimplices of dimension 0. We will usualy use Latin letters $x$,$y$,$\cdots$ for vertices and Greek letters $\sigma$,$\tau$, $\cdots$ for polysimplices. We can define an order relation on $\bt$ by $\sigma \leq \tau$ if $\sigma$ is a face of $\tau$. Two vertices $x$ and $y$ are adjacent if there exists a polysimplex $\sigma$ such that $x \leq \sigma$ and $y \leq \sigma$.

\bigskip

Let $\ld$ be a ring where $p$ is invertible. We fix a Haar measure on $G$ and denote by $\hecke{\G}{\ld}$ the Hecke algebra with coefficients in $\ld$, that is the algebra of functions from $G$ to $\ld$ locally constant with compact support.

\begin{Def}[{\cite[Def. 2.1.]{meyer_resolutions_2010}}]
    A system of idempotents $e=(e_{x})_{x\in \bts}$ of $\hecke{\G}{\ld}$ is said to be consistent if the following properties are satisfied:
    \begin{enumerate}
        \item $e_{x}e_{y}=e_{y}e_{x}$ when $x$ and $y$ are adjacent.
        \item $e_{x}e_{z}e_{y}=e_{x}e_{y}$ when $z$ is adjacent to $x$ and in the polysimplicial hull of $x$ and $y$.
        \item $e_{gx}=ge_{x}g^{-1}$ for all $x\in \bts$ and $g\in \G$.
    \end{enumerate}
\end{Def}

If $e=(e_{x})_{x\in \bts}$ is a consistent system of idempotent, then for $\sigma \in \bt$ we can define $e_{\sigma}:=\prod_{x} e_x$, where the product is taken over the vertices $x$ such that $x \leq \sigma$.

\bigskip

Consistent systems of idempotents are very interesting because we have the following theorem due to Meyer and Solleveld.

\begin{The}[\cite{meyer_resolutions_2010}, Thm 3.1]
    \label{thmMeyerSollveld}
    Let $e=(e_{x})_{x\in \bts}$ a consistent system of idempotents, then the full sub-category $\rep[\ld][e]{\G}$ of objects $V$ of $\rep[\ld]{\G}$ such that $V=\sum_{x\in \bts}e_{x}V$ is a Serre sub-category.
\end{The}

It may not be easy to check the conditions of consistency. But, if we are working with the subcategory of depth zero representations, we can find in \cite{lanard} the notion of 0-consistent, which implies consistency, and is easier to check.

\bigskip

Let $\sigma \in \bt$. We denote by $\para{G}{\sigma}$ the parahoric subgroup at $\sigma$ and by $\radpara{G}{\sigma}$ its pro-$p$-radical. The quotient, $\quotred{G}{\sigma}$, is then the group of $\res$-points of a connected reductive group $\quotredalg{G}{\sigma}$ defined over $\res$.

If $\sigma\in \bt$ is a polysimplex, then $\radpara{G}{\sigma}$ defines an idempotent $e_{\sigma}^{+} \in \mathcal{H}_{\mathbb{Z}[1/p]}(G)$ by $e_{\sigma}^{+}=\mu(\radpara{G}{\sigma})^{-1}\chi_{\radpara{G}{\sigma}}$, where $\mu$ is our fixed Haar measure and $\chi_{\radpara{G}{\sigma}}$ is the characteristic function of $\radpara{G}{\sigma}$. The system of idempotents $(e_{x}^{+})_{x \in \bts}$ is consistent and cuts out the category of depth zero.

\begin{Def}
    \label{defDepthZero}
    An object $V$ of $\rep[\ld]{\G}$ has depth zero if $V=\sum_{x\in \bts}e_{x}^{+}V$.
\end{Def}
In other words, with the notations of Theorem \ref{thmMeyerSollveld}, the depth zero category is $\rep[\ld][0]{\G}=\rep[\ld][e^{+}]{\G}$, with $e^+=(e_{x}^{+})_{x \in \bts}$.

\begin{Def}[{\cite[Def. 1.0.5]{lanard}}]
    We say that a system $(e_{\sigma})_{\sigma\in \bt}$ is 0-consistent if
    \begin{enumerate}
        \item $e_{gx}=ge_{x}g^{-1}$ for all $x\in \bt_{0}$ and $g\in G$.
        \item $e_{\sigma}=e_{\sigma}^{+}e_{x}=e_{x}e_{\sigma}^{+}$ for $x \in \bt_{0}$ and $\sigma \in \bt$ such that $x \leq \sigma$.
    \end{enumerate}
\end{Def}

\begin{Pro}[{\cite[Prop. 1.0.6]{lanard}}]
    If $(e_{\sigma})_{\sigma\in \bt}$ is a 0-consistent system of idempotents, then it is consistent.
\end{Pro}

Let us give two examples of systems of idempotents which are 0-consistent. Let $\sigma\in \bt$. Let $\dl{\quotred{G}{\sigma}}{1}$ be the Deligne-Lusztig series associated with the trivial conjugacy class, that is the set of unipotent characters in $\quotred{G}{\sigma}$. Let $e_{1,\quotred{G}{\sigma}}$, be the central idempotent in $\Ql[\quotred{\G}{\sigma}]$ that cuts out $\dl{\quotred{G}{\sigma}}{1}$. Thanks to the isomorphism $ \para{G}{\sigma} / \radpara{G}{\sigma} \tosim \quotred{G}{\sigma}$, we can pull back $e_{1,\quotred{G}{\sigma}}$ to an idempotent $e_{1,\sigma} \in \hecke{\para{G}{\sigma}}{\Ql}$. The system $e_{1}=(e_{1,\sigma})_{\sigma \in \bt}$ is then 0-consistent (see \cite[Prop. 2.3.2]{lanard}). Thus it defines $\repun{\Ql}{G}$ the full-subcategory of $\rep[\Ql]{G}$ of unipotent representations.

In the same way, let $\dll{\quotred{G}{\sigma}}{1}$ be the union of the $\dl{\quotred{G}{\sigma}}{t}$, where $t$ is a semi-simple conjugacy class in the dual of $\quotred{G}{\sigma}$, of order a power of $\lprime$. By \cite{bonnafe_rouquier} Theorem A' and remark 11.3, the idempotent that cuts out this series is in $\Zl[\quotred{\G}{\sigma}]$. We can then pull it back to get $e^{\lprime}_{1,\sigma} \in \hecke{\para{G}{\sigma}}{\Zl}$. This system $e^{\lprime}_{1}=(e^{\lprime}_{1,\sigma})_{\sigma \in \bt}$ is also 0-consistent and defines the $\lprime$-unipotent subcategory $\repun{\Zl}{G}$.

\subsection{Bernstein blocks with system of idempotents}
\label{secidempotenthc}

In this section, we want to reinterpret the Bernstein blocks of depth zero (that is the blocks over $\Ql$ or $\mathbb{C}$), in terms of consistent systems of idempotents. To do that, we will construct a 0-consistent system of idempotents from unrefined depth zero types, hence subcategories of $\rep[\Ql][0]{G}$. When $G$ is semisimple and simply-connected, these categories will be blocks.

\sautintro

We define, as in \cite{latham}, ``unrefined depth zero types'' to be the pairs $(\sigma,\pi)$, where $\sigma \in \bt$ and $\pi$ is an irreducible cuspidal representation of $\quotred{G}{\sigma}$. Let $\TG$ be the set of unrefined depth zero types, up to $G$-conjugacy.

If $\sigma,\tau \in \bt$ are two polysimplices with $\tau \leq \sigma$, we can see $\quotred{G}{\sigma}$ as a Levi subgroup of $\quotred{G}{\tau}$. Let $\mathfrak{t}$ and $\mathfrak{t}'$ be two elements of $\TG$ and $\omega \in \bt$. Then we say that $\mathfrak{t} \sim_{\omega} \mathfrak{t}'$ if and only if $\mathfrak{t}=\mathfrak{t'}$ or there exist $(\sigma,\pi)$ and $(\tau,\pi')$ such that $\mathfrak{t}=[\sigma,\pi]$, $\mathfrak{t}'=[\tau,\pi']$, $\omega$ is a face of $\sigma$ and $\tau$, and the cuspidal pairs $(\quotred{G}{\sigma},\pi)$ and $(\quotred{G}{\tau},\pi')$ are conjugated in $\quotred{G}{\omega}$. Now we define $\sim$, an equivalence relation on $\TG$ by $\mathfrak{t} \sim \mathfrak{t}'$ if and only if there exist $\omega_1,\cdots,\omega_r \in \bt$ and $\mathfrak{t}_{1}, \cdots, \mathfrak{t}_{r-1} \in \TG$ such that $\mathfrak{t} \sim_{\omega_1} \mathfrak{t}_1\sim_{\omega_2} \mathfrak{t}_2 \cdots \sim_{\omega_r} \mathfrak{t}'$. We write $[\mathfrak{t}]$ for the equivalence class of $\mathfrak{t}$.

\bigskip

If $\gpfini{G}$ is a connected reductive group over $\res$, then the theory of Harish-Chandra allows us to partition $\Irr(\gpfini{G})$ according to cuspidal support $[\gpfini{M},\pi]$:
\[\Irr(\gpfini{G})= \bigsqcup \Irr_{(\gpfini{M},\pi)}(\gpfini{G}).\]
Now we construct from $[\mathfrak{t}] \in \TG/{\sim}$ a system of idempotents $e_{[\mathfrak{t}]}$ in the following way. Let $\tau \in \bt$ and define $e_{[\mathfrak{t}]}^{\tau} \in \Ql[\quotred{G}{\tau}]$ the idempotent that cuts out the union of $\Irr_{(\quotred{G}{\sigma},\pi)}(\quotred{G}{\tau})$ for every $[\sigma,\pi] \in [\mathfrak{t}]$ with $\tau \leq \sigma$. We can then pull pack $e_{[\mathfrak{t}]}^{\tau}$ to an idempotent $e_{[\mathfrak{t}],\tau} \in \hecke{\para{G}{\tau}}{\Ql}\subseteq \hecke{G}{\Ql}$, giving us $e_{[\mathfrak{t}]}$ a system of idempotents.

\begin{Lem}
    \label{lempropidemptype}
    Let $x \in \bts$, $\sigma \in \bt$ with $x \leq \sigma$. We have the following properties
    \begin{enumerate}
        \item $e_{\sigma}^{+}=\sum_{[\mathfrak{t}] \in \TG/{\sim}} e_{[\mathfrak{t}],\sigma}$.
        \item For all $\mathfrak{t},\mathfrak{t}' \in \TG$ with $[\mathfrak{t}] \neq [\mathfrak{t}']$, $e_{[\mathfrak{t}],x} e_{[\mathfrak{t}'],\sigma} = 0$.
    \end{enumerate}
\end{Lem}

\begin{proof}
    \begin{enumerate}
        \item The partition $\Irr(\quotred{G}{\sigma})= \bigsqcup \Irr_{(\gpfini{M},\pi)}(\quotred{G}{\sigma})$ and the fact that each $\Irr_{(\gpfini{M},\pi)}(\quotred{G}{\sigma})$ can be written as $\Irr_{(\gpfini{M},\pi)}(\quotred{G}{\sigma}) = \Irr_{(\quotred{G}{\tau},\pi)}(\quotred{G}{\sigma})$ for a polysimplex $\tau \geq \sigma$ show the wanted equality.

        \item The group $\quotred{G}{\sigma}$ is a Levi quotient of a parabolic $\gpfini{P}_{\sigma}$ of $\quotred{G}{x}$, and we denote by $\gpfini{U}_{\sigma}$ the unipotent radical of $\gpfini{P}_{\sigma}$. The idempotent  $e_{[\mathfrak{t}'],\sigma} \in \hecke{\para{G}{\sigma}}{\Ql} \subseteq \hecke{\para{G}{x}}{\Ql}$ gives us in $\Ql[\quotred{G}{x}]$ the idempotent $e_{\gpfini{U}_{\sigma}} e_{[\mathfrak{t}']}^{\sigma}$, where $e_{\gpfini{U}_{\sigma}}$ is the idempotent which averages along the group $\gpfini{U}_{\sigma}$. We have to prove that $e_{[\mathfrak{t}]}^{x} e_{\gpfini{U}_{\sigma}} e_{[\mathfrak{t}']}^{\sigma}=0$ in $\Ql[\quotred{G}{x}]$. But $\Ql[\quotred{G}{x}]e_{\gpfini{U}_{\sigma}} e_{[\mathfrak{t}']}^{\sigma}$ is the parabolic induction from $\quotred{G}{\sigma}$ to $\quotred{G}{x}$ of the module $\Ql[\quotred{G}{\sigma}]e_{[\mathfrak{t}']}^{\sigma}$. Since $[\mathfrak{t}] \neq [\mathfrak{t}']$ no representation in $\Irr_{(\quotred{G}{\tau},\pi)}(\quotred{G}{x})$, with $[\tau,\pi] \in [\mathfrak{t}]$ can be in the induction of a representation in $\Irr_{(\quotred{G}{\tau'},\pi')}(\quotred{G}{\sigma})$ with $[\tau',\pi'] \in [\mathfrak{t}']$. Hence $e_{[\mathfrak{t}]}^{x} e_{\gpfini{U}_{\sigma}} e_{[\mathfrak{t}']}^{\sigma}=0$.
    \end{enumerate}
\end{proof}

\begin{Pro}
    The system of idempotents $e_{[\mathfrak{t}]}$ is 0-consistent.
\end{Pro}

\begin{proof}
    An element $\mathfrak{t} \in \TG$ is defined up to $G$-conjugacy, hence $e_{[\mathfrak{t}]}$ is $G$-equivariant. Let $x \in \bts$ and $\sigma \in \bt$ such that $x \leq \sigma$. We have to prove that $e_{[\mathfrak{t}],\sigma} = e_{\sigma}^{+}e_{[\mathfrak{t}],x}$. By 1. in \ref{lempropidemptype} we have that $e_{\sigma}^{+}=\sum_{[\mathfrak{t}'] \in \TG/{\sim}} e_{[\mathfrak{t}'],\sigma}$. Hence, $e_{[\mathfrak{t}],x}e_{\sigma}^{+}=\sum_{[\mathfrak{t}'] \in \TG/{\sim}} e_{[\mathfrak{t}],x}e_{[\mathfrak{t}'],\sigma}$. Now by 2. in \ref{lempropidemptype}, we have that if $[\mathfrak{t}] \neq [\mathfrak{t'}]$ then $e_{[\mathfrak{t}],x}e_{[\mathfrak{t}'],\sigma}=0$. So $e_{[\mathfrak{t}],x}e_{\sigma}^{+} = e_{[\mathfrak{t}],x}e_{[\mathfrak{t}],\sigma}$. In the same way, $e_{[\mathfrak{t}],x}e_{[\mathfrak{t}],\sigma} = e_{x}^{+}e_{[\mathfrak{t}],\sigma}$. So, $e_{[\mathfrak{t}],x}e_{\sigma}^{+} = e_{[\mathfrak{t}],x}e_{[\mathfrak{t}],\sigma}=e_{x}^{+}e_{[\mathfrak{t}],\sigma}=e_{x}^{+}e_{\sigma}^{+}e_{[\mathfrak{t}],\sigma}=e_{\sigma}^{+}e_{[\mathfrak{t}],\sigma}=e_{[\mathfrak{t}],\sigma}$.
\end{proof}

Let $\mathfrak{t} \in \TG$. We denote by $\rep[\Ql][[\mathfrak{t}]]{G}$ the category associated with $e_{[\mathfrak{t}]}$.

\begin{Pro}
    \label{prodecompotypes}
    We have the decomposition
    \[ \rep[\Ql][0]{G}=\prod_{[\mathfrak{t} ]\in \TG/{\sim}}\rep[\Ql][[\mathfrak{t}]]{G}.\]
\end{Pro}

\begin{proof}
    The proof is similar to the proof of \cite[Prop. 2.3.5]{lanard}. Property 2. in Lemma \ref{lempropidemptype} shows that these categories are pairwise orthogonal and property 1. in Lemma \ref{lempropidemptype} shows that the product if $\rep[\Ql][0]{G}$.

\end{proof}

\begin{The}
    \label{theQlblock}
    If $G$ is semisimple and simply-connected the category $\rep[\Ql][[\mathfrak{t}]]{G}$ is a block.
\end{The}

\begin{proof}
    When $G$ is semisimple and simply-connected, Theorem 4.9 of \cite{morris} shows that we have a bijection between $\TG/{\sim}$ and level zero Bernstein blocks. We then deduce from Proposition \ref{prodecompotypes} that $\rep[\Ql][0]{G}=\prod_{[\mathfrak{t}] \in \TG/{\sim}}\rep[\Ql][[\mathfrak{t}]]{G}$ is the decomposition of $\rep[\Ql][0]{G}$ into Bernstein blocks.
\end{proof}

We would like to do the same thing to construct $\lprime$-blocks. The simplest case is when $\lprime$ is banal.

\begin{Def}
    \label{deflbanal}
    We say that a prime number $\lprime \neq p$ is banal when for every vertex $x\in \bts$, $\lprime$ does not divide the cardinal of $\quotred{G}{x}$.
\end{Def}

Therefore, when $\lprime$ is banal each idempotent $e_{[\mathfrak{t}]}$ is in $\hecke{G}{\Zl}$. Thus we have a decomposition
\[\rep[\Zl][0]{G}=\prod_{[\mathfrak{t}] \in \TG/{\sim}}\rep[\Zl][[\mathfrak{t}]]{G}\]
and the following theorem

\begin{The}
    If $G$ is semisimple and simply-connected, and $\lprime$ is banal, the category $\rep[\Zl][[\mathfrak{t}]]{G}$ is an $\lprime$-block.
\end{The}

In the general case, the idempotents do not have coefficients in $\Zl$. The topic of the followings sections will be to explain how to sum these idempotents to get idempotents with integral coefficients.

\section{\texorpdfstring{$\dun$}{(d,1)}-theory}
\label{secduntheo}

We have seen in section \ref{secidempotenthc} how to construct the Bernstein blocks with consistent systems of idempotents when we have a simply-connected group. To construct $\lprime$-blocks, we need to produce central idempotents for finite reductive groups with coefficients in $\Zl$. In this section, we introduce the notion of a $\dun$-set. This is a subset of $\Irr(\gpfini{G})$ which is a union of Harish-Chandra series and gives a central idempotent with coefficients in $\Zl$. These $\dun$-sets will be used in the next sections to describe the unipotent $\lprime$-blocks for simply-connected $p$-adic groups.

\bigskip

This section will only deal with finite reductive groups. Let us take $(\gpfinialg{G}, \fr)$ a connected reductive group defined over $\res$, and let $\gpfini{G}:=(\gpfinialg{G})^{\fr}$. We recall that $q=\card{\res}$. We will define the $\dun$-set and $\dun$-series, then explain how to compute them, and to finish, we will show that they behave well with respect to Harish-Chandra induction and Deligne-Lusztig induction from particular Levi subgroups.

\subsection{Unipotent \texorpdfstring{$\lprime$}{l}-blocks for finite reductive groups}

\label{secblockunipfini}

We recall in this section the theory of $\lprime$-blocks for a finite connected reductive group. These blocks will be constructed using a modified Harish-Chandra induction called $d$-Harish-Chandra induction, defined using Deligne-Lusztig theory.

\sautintro

For each connected reductive group $(\gpfinialg{G}, \fr)$ over $\res$, there exists a unique polynomial $\ordpol{G} \in \mathbb{Z}[x]$ called the polynomial order of $\gpfinialg{G}$ (see for example \cite{bmm} section 1.A) with the property that there is $a \geq 1$ such that $\card{\gpfinialg{G}^{\fr^m}}=\ordpol{G}(q^m)$ for all $m \geq 1$ such that $m \equiv 1 \pmod{a}$. The prime factors of $\ordpol{G}$ distinct from $x$ are cyclotomic polynomials. Let $d\geq 1$ be an integer and $\cycl{d}$ the corresponding cyclotomic polynomial. We say that $\gpfinialg{T}$ is a $\cycl{d}$-subgroup if $\gpfinialg{T}$ is a $\fr$-stable torus of $\gpfinialg{G}$ whose polynomial order is a power of $\cycl{d}$. A $d$-split Levi subgroups of $\gpfinialg{G}$ is the centralizer in $\gpfinialg{G}$ of some $\cycl{d}$-subgroup of $\gpfinialg{G}$.

Let $\chi \in \Irr(\gpfini{G})$ be an ordinary irreducible character. We say that $\chi$ is $d$-cuspidal if and only if $\rdl \chi = 0$ for every proper $d$-split Levi subgroup $\gpfinialg{L}$ and every parabolic $\gpfinialg{P}$ admitting $\gpfinialg{L}$ as Levi subgroup.

A ``unipotent $d$-pair'' is a pair $(\gpfinialg{L},\lambda)$ where $\gpfinialg{L}$ is a $d$-split Levi and $\lambda$ is a unipotent character of $\gpfinialg{L}$. Such a pair is said to be cuspidal if $\lambda$ is cuspidal. We define an order on unipotent $d$-pairs by $(\gpfinialg{M},\mu) \preceq (\gpfinialg{L},\lambda)$ if $\gpfinialg{M}$ is a Levi subgroup of $\gpfinialg{L}$ and there is a parabolic subgroup $\gpfinialg{P}$ of $\gpfinialg{L}$ admitting $\gpfinialg{M}$ as a Levi such that $\langle \lambda, \inddl{M}{P}{L} (\mu) \rangle \neq 0$. For $(\gpfinialg{L},\lambda)$ a unipotent $d$-cuspidal pair, let us define $\dldcusp{L}{\lambda}$ to be the subset of $\dl{\gpfini{G}}{1}$ of characters $\chi$ such that $(\gpfinialg{L},\lambda) \preceq (\gpfinialg{G},\chi)$. We call $\dldcusp{L}{\lambda}$ a $d$-series.

\begin{The}[\cite{bmm}Theorem 3.2 (1)]
    \label{thmdecompodseries}
    For each $d$, the sets $\dldcusp{L}{\lambda}$ (where $(\gpfinialg{L},\lambda)$ runs over a complete set of representatives of $\gpfini{G}$-conjugacy classes of unipotent $d$-cuspidal pairs) partition $\dl{\gpfini{G}}{1}$.
\end{The}

An $\lprime$-block is a primitive idempotent in the centre $Z(\Zl[\gpfini{G}])$ of the group algebra $\Zl[\gpfini{G}]$. For $b$ an $\lprime$-block, we denote by $\Irr(b)$ the subset of $\Irr(\gpfini{G})$ that is cuts out by the idempotent $b$. This defines a partition $\Irr(\gpfini{G})=\sqcup_{b}\Irr(b)$. The $\lprime$-unipotent series $\dll{\gpfini{G}}{1}$, defined as the union of the $\dl{\gpfini{G}}{t}$ with $t$ of order a power of $\lprime$, defines a central idempotent in $\Zl[\gpfini{G}]$ (\cite{bonnafe_rouquier} Theorem A' and remark 11.3), hence it is a union of $\lprime$-blocks: $\dll{\gpfini{G}}{1}=\sqcup_{b}\Irr(b)$. We will call these blocks the unipotent $\lprime$-blocks.

\bigskip

Let $\lprime$ be a prime number not diving $q$. We will say that $\lprime$ satisfies the condition \eqref{eql} if
\begin{equation}
    \label{eql}
    \lprime \text{ is odd, } \lprime \text{ is good for } \gpfinialg{G} \text{ and } \lprime \neq 3 \text{ if } \tDq \text{ is involved in } (\gpfinialg{G},\fr)
    \tag{$\ast$}
\end{equation}

\medskip

Let us summarize the condition of being good and \eqref{eql} in a table

\begin{tabular}{|c|c|c|c|c|c|}
    \hline
    Types           & $\An$, $\dAn$    & $\Bn$, $\Cn$, $\Dn$, $\dDn$ & $\tDq$           & $\dG$, $\Fqu$, $\Esi$, $\dEsi$, $\Ese$ & $\Eh$            \\
    \hline
    bad $\lprime$'s & $\emptyset$      & $\{2\}$                     & $\{2\}$          & $\{2, 3\}$                             & $\{2, 3, 5\}$    \\
    \hline
    \eqref{eql}     & $\lprime \geq 3$ & $\lprime \geq 3$            & $\lprime \geq 5$ & $\lprime \geq 5$                       & $\lprime \geq 7$ \\
    \hline
\end{tabular}

\begin{The}[{\cite[Thm. 4.4]{CabanesEngueharduni}}]
    \label{thelblockdcuspi}
    We assume that $\lprime$ satisfies \eqref{eql} and let $d$ be the order of $q$ modulo $\lprime$. Then there is a bijection
    \[ (\gpfinialg{L},\lambda) \mapsto b(\gpfinialg{L},\lambda),\]
    between the set of $\gpfini{G}$-conjugacy classes of unipotent $d$-cuspidal pairs of $\gpfini{G}$ and the set of unipotent $\lprime$-blocks.

    Moreover, we have that $\Irr(b(\gpfinialg{L},\lambda)) \cap \dl{\gpfini{G}}{1} = \{\chi, (\gpfinialg{L},\lambda) \preceq (\gpfinialg{G},\chi) \}$.
\end{The}

If $b$ is a unipotent $\lprime$-block, then the knowledge of $\Irr(b) \cap \dl{\gpfini{G}}{1}$ is enough to describe all the characters in $\Irr(b)$. To explain this, we need a few more notations.

Let $t \in \gpfini*{G}$ be a semi-simple element of order a power of $\lprime$. Let $\lprime$ be a good prime for $\gpfinialg{G}$. Then $\cent*{t}{\gpfinialg*{G}}$ is a Levi subgroup (see for example \cite{CabanesEngueharduni} Proposition 2.1). Let $\gpfinialg{G}(t)$ be a Levi subgroup in $\gpfinialg{G}$ in duality with $\cent*{t}{\gpfinialg*{G}}$ over $\res$ and $\gpfinialg{P}$ be a parabolic subgroup with Levi component $\gpfinialg{G}(t)$.

Since $t$ is a central element of $(\cent*{t}{\gpfinialg*{G}})^{\fr}$, by \cite{digneMichel} Proposition 13.30, there exists a linear character $\hat{t} \in \Irr(\gpfini{G}(t))$ such that the tensor product with $\hat{t}$ defines a bijection from $\dl{\gpfini{G}(t)}{1}$ to $\dl{\gpfini{G}(t)}{t}$. Let $\chi \in \dl{\gpfini{G}}{t}$. Then, by the Jordan decomposition in the case of non connected centre (defined in \cite{lusztigdisco}) there exists $\chi_t \in \dl{\gpfini{G}(t)}{1}$ such that $ \langle \chi, \inddllite{\gpfinialg{G}(t) \subseteq \gpfinialg{P}}{\gpfinialg{G}}(\hat{t}\chi_t) \rangle \neq 0$.

\begin{The}[{\cite[Thm. 4.4]{CabanesEngueharduni}}]
    \label{thedescriplblock}
    Let $\lprime$ be a prime good for $\gpfinialg{G}$. Let $\chi \in \dl{\gpfini{G}}{t}$, for $t$ a semi-simple conjugacy class in $\gpfini*{G}$ of order a power of $\lprime$. Let $b$ be the $\lprime$-block such that $\chi \in \Irr(b)$. Let $\gpfinialg{G}(t)$ be a $\fr$-stable Levi in $\gpfinialg{G}$ dual to $\cent*{t}{\gpfinialg*{G}}$, $\gpfinialg{P}$ be a parabolic subgroup with Levi component $\gpfinialg{G}(t)$, and $\chi_t \in \dl{\gpfini{G}(t)}{1}$ such that $ \langle \chi, \inddllite{\gpfinialg{G}(t) \subseteq \gpfinialg{P}}{\gpfinialg{G}}(\hat{t}\chi_t) \rangle \neq 0$. For any such $(\gpfinialg{G}(t), \gpfinialg{P}, \chi_t)$ associated to $\chi$, all the irreducible  components of $\inddllite{\gpfinialg{G}(t) \subseteq \gpfinialg{P}}{\gpfinialg{G}}(\chi_t)$ are in $\Irr(b) \cap \dl{\gpfini{G}}{1}$.
\end{The}

Let $(\gpfinialg{L},\lambda)$ be a unipotent $d$-cuspidal pair. Then we define the $\lprime$-extension of the $d$-series $\dldcusp{L}{\lambda}$ as the subset $\mathcal{E}_{\lprime}(\gpfini{G},(\gpfinialg{L},\lambda)) \subseteq \dll{\gpfini{G}}{1}$ of characters $\chi \in \dll{\gpfini{G}}{1}$ such that, with the notation of Theorem \ref{thedescriplblock}, all the irreducible  components of $\inddllite{\gpfinialg{G}(t) \subseteq \gpfinialg{P}}{\gpfinialg{G}}(\chi_t)$ are in $\dldcusp{L}{\lambda}$. Hence, if $\lprime$ satisfies \eqref{eql}, then $\mathcal{E}_{\lprime}(\gpfini{G},(\gpfinialg{L},\lambda)) = \Irr(b(\gpfinialg{L},\lambda))$.

\subsection{\texorpdfstring{$\dun$}{(d,1)}-series}

\label{secdunseries}

We have seen in Section \ref{secidempotenthc} that in order to construct Bernstein blocks we needed to decompose $\Irr(\gpfini{G})$ as Harish-Chandra series. But to get $\lprime$-blocks we need to decompose it as $d$-series, as seen in Section \ref{secblockunipfini}. In this section, we will introduce $\dun$-series, which will give a partition of $\Irr(\gpfini{G})$ into subsets which are both a union of Harish-Chandra series and a union of $d$-series.

\bigskip

First, let us remark that 1-series are just Harish-Chandra series, so from now on we will speak of 1-split Levi, 1-cuspidal pairs and 1-series when we want to talk about ``normal'' Levi subgroup, cuspidal pairs and Harish-Chandra series.

\begin{Def}
    We define a $\dun$-set to be a subset of $\Irr(\gpfini{G})$ which is a union of $1$-series and a union of $d$-series. A $\dun$-series is then a $\dun$-set with no proper non-empty $\dun$-subset.

    A $\dun$-set, respectively a $\dun$-series, included in $\dl{\gpfini{G}}{1}$ will be called a unipotent $\dun$-set, respectively a unipotent $\dun$-series.
\end{Def}

\begin{Rem}
    \begin{enumerate}
        \item By Theorem \ref{thmdecompodseries} $\dl{\gpfini{G}}{1}$ is a $\dun$-set, so the unipotent $\dun$-series give a partition of $\dl{\gpfini{G}}{1}$.
        \item If $\cycl{d}$ does not divide $\ordpol{G}$, then the only $\cycl{d}$-torus is the trivial one. Hence the $\dun$-series are just the $1$-series.
    \end{enumerate}
\end{Rem}

Let $\mathcal{E}$ be a unipotent $\dun$-series. Since $\mathcal{E}$ can be written as a union of $d$-series $\mathcal{E}=\bigsqcup_{i} \mathcal{E}_i$, we can define the $\lprime$-extension of a $\dun$-series by
\[\mathcal{E}_{\lprime}:=\bigsqcup_{i} \mathcal{E}_{i,\lprime}.\]

We want to compute the unipotent $\dun$-series. The first step is to reduce to the case of simple groups.

\bigskip

To every orbit $\omega$ of $\fr$ on the set of connected components of the Dynkin diagram of $\gpfinialg{G}$ there corresponds a well defined $\fr$-stable subgroup $\gpfinialg{G}'_{\omega}$ of $[\gpfinialg{G},\gpfinialg{G}]$ and a component $\gpfinialg{G}_{\omega} = Z^{\circ}(\gpfinialg{G})\gpfinialg{G}'_{\omega}$ of $\gpfinialg{G}$. The finite group $(\gpfinialg{G}_{\omega}/Z(\gpfinialg{G}_{\omega}))^{\fr}$ is characterized by its simple type $\{ \An, \dAn, \Bn, \Cn, \Dn, \dDn, \tDq, \dG, \Fqu, \Esi, \dEsi, \Ese, \Eh\}$ and an extension field $\Fq{m(\omega)}$ of $\Fq{}$ of degree $m(\omega)$ equal to the length of the orbit of $\omega$. Moreover, when $\gpfinialg{G}=\ad{\gpfinialg{G}}$, where $\ad{\gpfinialg{G}}$ denotes the adjoint group of $\gpfinialg{G}$, then it is a direct product of its components.

Let us begin, by showing how to reduce to $\gpfinialg{G}$ of adjoint type.

\begin{Pro}
    \label{produnserieadjoint}
    Let $\pi:\gpfinialg{G} \to \ad{\gpfinialg{G}}$ be the reduction map modulo $Z(\gpfinialg{G})$. Then $\pi$ induces a bijection between $\dl{\ad{\gpfini{G}}}{1}$ and $\dl{\gpfini{G}}{1}$ which commutes with the Deligne-Lusztig induction and preserves unipotent $\dun$-series.
\end{Pro}

\begin{proof}
    This follows from \cite[Prop. 1.36]{bmm} and \cite[Rem. 1.25]{bmm}.
\end{proof}

Let $a\in \mathbb{N}^{*}$. Denote by $(\rscal{a}{\gpfinialg{G}},\rscal{a}{\fr})$ the restriction of scalars (or Weil restriction) of $(\gpfinialg{G},\fr)$ from $\Fq{a}$ to $\Fq{}$. This is a reductive group defined over $\Fq{}$ characterized by the property: for any $\Fq{}$-algebra $A$ we have $\rscal{a}{\gpfinialg{G}}(A)=\gpfinialg{G}(A \otimes_{\Fq{}} \Fq{a})$. In particular, $(\rscal{a}{\gpfinialg{G}})^{\rscal{a}{\fr}}=\gpfinialg{G}^{\fr^{a}}$ (that is $\rscal{a}{\gpfinialg{G}}(\Fq{})=\gpfinialg{G}(\Fq{a})$). Moreover, the isomorphism $(\rscal{a}{\gpfinialg{G}})^{\rscal{a}{\fr}} \simeq \gpfinialg{G}^{\fr^{a}}$ ``commutes'' with the Deligne-Lusztig induction and map isomorphically $\dl{(\rscal{a}{\gpfinialg{G}})^{\rscal{a}{\fr}}}{1}$ to $\dl{\gpfinialg{G}^{\fr^{a}}}{1}$. Now a group of adjoint type is a direct product of restriction of scalars of simple groups. Let us take a look at the behaviour of $\dun$-series with respect to restriction of scalars.

\begin{Pro}
    \label{produnscalaire}
    Let $a \in \mathbb{N}^{*}$. We have a bijection between the $\dun$-series in $\dl{(\rscal{a}{\gpfinialg{G}})^{\rscal{a}{\fr}}}{1}$ and the $\dung{d/\gcd(d,a)}{1}$-series in $\dl{\gpfinialg{G}^{\fr^{a}}}{1}$.
\end{Pro}

\begin{proof}
    If $p$ is a prime number, then
    \[
        \cycl{n}(x^p) = \left\{
        \begin{array}{ll}
            \cycl{pn}(x)            & \mbox{if } p|n    \\
            \cycl{pn}(x)\cycl{n}(x) & \mbox{otherwise.}
        \end{array}
        \right.\]
    From that we can deduce what is $\cycl{n}(x^a)$. We write $a=	a_n a_n'$, with $a_n'$ relatively prime with $n$ and all the prime numbers dividing $a_n$ also divide $n$. Then we have
    \[ \cycl{n}(x^a)= \prod_{k | a_n'} \cycl{k a_n n}(x).\]

    Let us prove that $\cycl{d}(x)$ divides $\cycl{n}(x^a)$ if and only if $n=d/\gcd(d,a)$.

    First assume that $n=d/\gcd(d,a)$. Hence, we want to prove that there exists $k | a_n'$ such that $ka_n=\gcd(d,a)$. If $p^e | a_n$, then $p| n=d/\gcd(d,a)$. So, $\nu_{p}(d) \geq \nu_{p}(a)$, where $\nu_p$ is the $p$-adic valuation. Hence, $\nu_p(\gcd(d,a))=\nu_p(a)=\nu_p(a_n)$. Thus $a_n |\gcd(d,a)$. Let $k=\gcd(d,a)/a_n$. It remains to prove that $k | a_n'$. We have that $k | a$ and if $p | k$, then $p \nmid a_n$ since it would imply that $\nu_p(\gcd(d,a))=\nu_p(a)=\nu_p(a_n)$ and a contradiction. Hence $k | a_n'$.

    Now, let us assume that there exist $n$ and $k$, such that $k| a_n'$ and $ka_n n =d$. We want to prove that $n=d/\gcd(d,a)$. It is enough to prove that $ka_n=\gcd(d,a)$. First $k|a_n'$ and since $a_n$ and $a_n'$ are relatively prime, $k|a$. We also have that $ka_n | d$, thus $ka_n | \gcd(d,a)$. Now, if $p^e | \gcd(d,a)$, then $p^e |a=a_n a_n'$. If $p^e | a_n$, then $p^e | ka_n$. If not, $p^e | a_n'$. Thus $p \nmid n$. But since $p^e | d=ka_n n$, we have that $p^e | k$ and $p^e |ka_n$. We conclude that $ka_n=\gcd(d,a)$.

    We have just proved that $\cycl{d}(x) | \cycl{n}(x^a)$ if and only if $n=d/\gcd(d,a)$. As a result, if $\gpfinialg{T}'$ is a torus in $\rscal{a}{\gpfinialg{G}}$, then  $\gpfinialg{T}'$ is a $d$-torus in $\rscal{a}{\gpfinialg{G}}$ if and only if it is the maximal $d$-sub-torus of $\rscal{a}{\gpfinialg{T}}$, for $\gpfinialg{T}$ a $d/\gcd(d,a)$-torus of $\gpfinialg{G}$. Thus the $d$-split Levi subgroup of $\rscal{a}{\gpfinialg{G}}$ are of the form $\rscal{a}{\gpfinialg{L}}$ for $\gpfinialg{L}$ a $d/\gcd(d,a)$-split Levi subgroup of $\gpfinialg{G}$. We then conclude the proof with the following commutative diagram of \cite[Prop. 1.37]{bmm}:
    \[
        \xymatrix{
        \mathbb{Z}\dl{(\rscal{a}{\gpfinialg{G}})^{\rscal{a}{\fr}}}{1}  \ar[r]^-{\sim} & \mathbb{Z}\dl{\gpfinialg{G}^{\fr^{a}}}{1} \\
        \mathbb{Z}\dl{(\rscal{a}{\gpfinialg{L}})^{\rscal{a}{\fr}}}{1} \ar[u]^{\inddllite{\rscal{a}{\gpfinialg{L}}}{\rscal{a}{\gpfinialg{G}}}} \ar[r]^-{\sim} & \ar[u]^{\inddllite{\gpfinialg{L}}{\gpfinialg{G}}} \mathbb{Z}\dl{\gpfinialg{L}^{\fr^{a}}}{1}} \]

\end{proof}

To compute the $\dun$-series of $\dl{\gpfini{G}}{1}$, Proposition \ref{produnserieadjoint} allows us to reduce to the case where $\gpfini{G}$ is adjoint. Now, an adjoint group can be written as a product of restriction of scalars of simple groups. The $\dun$-series of a direct product is the product of the $\dun$-series. Hence by Proposition \ref{produnscalaire}, we can compute the unipotent $\dun$-series of $\gpfini{G}$, if we know them for simple groups. This is what we do in the following sections.

\subsection{Computation of \texorpdfstring{$\dun$}{(d,1)}-series for type \texorpdfstring{$\An$}{An} and \texorpdfstring{$\dAn$}{2An}}
\label{secdunAn}
In this section, we want to compute the unipotent $\dun$-series for groups of type $\An$ and $\dAn$.

\bigskip

Let us start by explaining what the $d$-series are. First, let $\gpfini{G}$ be of type $\An$. The unipotent characters are in bijection with partitions of $n+1$. On partitions, there is the well defined notion of $d$-hook and of $d$-core (see for example \cite{jamesKerber} Chapter 2.7). The proof of Theorem 3.2 in \cite{bmm} then shows the following proposition.

\begin{Pro}
    \label{prodseriesAn}
    The $d$-cuspidal unipotent characters are precisely those where the partition is itself a $d$-core. Moreover, two characters are in the same $d$-series if and only if they have the same $d$-core.
\end{Pro}

In order to get the result for groups of type $\dAn$, we will use an ``Ennola''-duality. We use here the notation of \cite{bmm}. Let $\mathbb{G}=(\Gamma,W\phi)$ be a generic finite reductive group (\cite[§1.A]{bmm}. We can then define $\mathbb{G}^{-}$ by $\mathbb{G}^{-}:=(\Gamma,W(-\phi))$. To $\mathbb{G}$ we can associate a finite set $\Uch(\mathbb{G})$ (\cite[Thm. 1.26]{bmm}) which is in bijection with the set of unipotent characters of $\gpfini{G}=\mathbb{G}(q)$.

\begin{The}[{\cite[Thm. 3.3]{bmm}}]
    \label{theEnnoladual}
    There exists a natural bijective isometry $\sigma^{\mathbb{G}} : \mathbb{Z}\Uch(\mathbb{G}) \to \mathbb{Z}\Uch(\mathbb{G}^{-})$ such that whenever $\mathbb{L}$ is $d$-split for some $d$, the following diagram is commutative
    \[
        \xymatrix{
        \mathbb{Z}\Uch(\mathbb{G})  \ar[r]^-{\sigma^{\mathbb{G}}} &\mathbb{Z}\Uch(\mathbb{G}^{-}) \\
        \mathbb{Z}\Uch(\mathbb{L}) \ar[u]^{\mathcal{R}_{\mathbb{L}}^{\mathbb{G}}} \ar[r]^-{\sigma^{\mathbb{L}}} &  \ar[u]^{\mathcal{R}_{\mathbb{L}^{-}}^{\mathbb{G}^{-}}} \mathbb{Z}\Uch(\mathbb{L}^{-})}\]
\end{The}

Note that if $\mathbb{G}(q)$ is of rational type $(\An,q)$ then $\mathbb{G}^{-}(q)$ is of rational type $(\dAn,q)$. In particular, we see that the unipotent characters for $\dAn$ are still parametrized by partitions of $n+1$. If $\mathbb{T}$ is a generic torus with polynomial order $\cycl{d}(x)$,  $\mathbb{T}^{-}$ has polynomial order $\cycl{d}(-x)$. The map $\mathbb{L} \mapsto \mathbb{L}^{-}$ is a bijection between $\cycl{d}(x)$-subgroup of $\dAn$ and $\cycl{d}(-x)$-subgroup of $\An$. Now, for $d>2$, we have that $\cycl{d}(-x)=\cycl{2d}(x)$ if $d$ is odd, $\cycl{d}(-x)=\cycl{d/2}(x)$ if $d$ is congruent to 2 modulo 4 and $\cycl{d}(-x)=\cycl{d}(x)$ if $d$ is divisible by 4. Let $d'$ be the integer defined by
\[
    d' = \left\{
    \begin{array}{ll}
        2d  & \mbox{if } d \mbox{ is odd}    \\
        d/2 & \mbox{if } d \equiv 2 \pmod{4} \\
        d   & \mbox{if } d \equiv 0 \pmod{4}
    \end{array}
    \right.
\]

By Theorem \ref{theEnnoladual} the $d$-series of $\dAn$ correspond to $d'$-series of $\An$ which are given by Proposition \ref{prodseriesAn}.

\begin{Rem}
    If $d$ is the order of $q$ modulo $\lprime$ then $d'$ is the order of $-q$ modulo $\lprime$.
\end{Rem}

\bigskip

In both cases, it is very important to be able to compute hooks and cores of partitions. In order to make the computation easier, and also to match with the following section \ref{dunclassicalgroup}, we will use the notion of a $\beta$-set instead of a partition.

A $\beta$-set is a subset $\lambda \subseteq \mathbb{N}$, and we will write $\lambda=(x_1\ x_2\ \cdots\ x_a)$ with $x_1 < x_2 < \cdots < x_a$. We define the rank of a $\beta$-set by $\rank(\lambda)=\sum_{i=1}^{a}x_i - a(a-1)/2$. We define an equivalence relation on the $\beta$-sets by $(x_1\ x_2\ \cdots\ x_a) \sim (0 \ x_1+1\ x_2+1\ \cdots\ x_a+1)$. The rank is invariant by this equivalence relation hence can be extended to equivalence classes. Now, a partition $a_1 \leq \cdots \leq a_k$ of $n+1$ can be sent to a $\beta$-set of rank $n+1$ defined by $\lambda=(a_1\ a_2+1 \ a_3+2 \ \cdots \ a_k+(k-1))$ and this gives us a bijection between partitions of $n+1$ and equivalence classes of $\beta$-set of rank $n+1$.

Let $\lambda$ and $\lambda'$ be two $\beta$-sets. We say that $\lambda'$ is obtained from $\lambda$ by a $d$-hook if there exists $x \in \lambda$ such that $x-d \notin \lambda$ and $\lambda'=\lambda\setminus\{x\} \cup \{x-d\}$. The $d$-core of $\lambda$ is then the $\beta$-set without $d$-hook obtained from $\lambda$ by repetitively removing $d$-hooks.

\begin{Lem}[\cite{jamesKerber} Lemma 2.7.13]
    Let $\lambda,\lambda'$ be two $\beta$-sets and $\alpha,\alpha'$ be two partitions corresponding respectively to $\lambda,\lambda'$. Then $\alpha'$ is obtained from $\alpha$ by a $d$-hook if and only if $\lambda'$ is obtained from $\lambda$ by a $d$-hook.
\end{Lem}

\bigskip

Now we have everything we need to compute the unipotent $\dun$-series for type $\An$ and $\dAn$.

For a group $\gpfini{G}$ of type $\An$, this is easy because there is no unipotent cuspidal representation. Hence, there is only one unipotent 1-series $\dl{G}{1}$ which is thus a $\dun$-series.

\begin{Pro}
    If $\gpfini{G}$ is of type $\An$, $\dl{G}{1}$ is a $\dun$-series.
\end{Pro}

Now, we assume that $\gpfini{G}$ is of type $\dAn$. We saw previously that two $\beta$-sets are in the same $d$-series if and only if they have the same $d'$-core and that they are in the same 1-series if and only if they have the same 2-core.

\begin{Rem}
    As we will see below, there are two different behaviours of the $\dun$-series depending on the parity of $d'$. When we will apply these results to $\lprime$-modular representations theory, $d$ will be the order of $q$ modulo $\lprime$. The primes $\lprime$ such that $d'$ is even are called \emph{linear} and when $d'$ is odd they are called \emph{unitary}.
\end{Rem}

The first case to consider is when $d'$ is even (linear prime case). We then have the following result.

\begin{Pro}
    If $d'$ is even (linear prime case) then the unipotent $\dun$-series for type $\dAn$ are the unipotent 1-series.
\end{Pro}

\begin{proof}
    If $d'$ is even, removing a $d'$-hook to a $\beta$-set can be obtained by removing $d'/2$ 2-hooks, hence the unipotent $\dun$-series are the unipotent 1-series.
\end{proof}

Now, let us assume that $d'$ is odd (unitary prime case).

\bigskip

Let $\lambda$ be a $\beta$-set with finite cardinal. Let $o$ be the number of odd numbers in $\lambda$ and $e$ be the number of even numbers. We define the defect of $\lambda$ by $\defect(\lambda)=o-e$ if $o\geq e$ and $e-o-1$ if $o<e$. The defect is invariant under the equivalence relation and we extend it to equivalence classes.

The 2-core  of a $\beta$-set is of the form $(1\ 3 \cdots \ 2k+1)$ (possibly $\emptyset$) which all have different defect. Moreover, removing a 2-hook does not change the defect of a $\beta$-set, so the defect of a $\beta$-set determines its 2-core, hence it characterizes the 1-series. Adding a 2-hook increase the rank of a $\beta$-set by $2$. Therefore, we get the following lemma.

\begin{Lem}
    \label{lemexistbsetkm}
    There exists a $\beta$-set of rank $m$ and defect $k$ if and only if $m-k(k+1)/2$ is even and positive.
\end{Lem}

Let $[\lambda]$ be an equivalence class of $\beta$-sets. We define $\max([\lambda])$ to be 0 if $(0) \in [\lambda]$ and $\max([\lambda]):=\max(\lambda')$ where $\lambda'$ is the unique $\beta$-set in $[\lambda]$ such that $0 \notin \lambda'$ if $(0) \notin [\lambda]$. Then $\max([\lambda])$ is the length of the largest hook in $[\lambda]$.

\begin{Lem}
    \label{lemmaxAn}
    Let $k\geq 0$ and $m\geq 1$ such that $m-k(k+1)/2$ is even and positive. We have
    \[ \max\{ \max([\lambda]), \defect(\lambda)=k, \rank(\lambda)=m\}=\begin{cases}
            m-\frac{k^2-3k+2}{2} & \text{ if } k \geq 1 \\
            m                    & \text{ if } k =0
        \end{cases}.\]
\end{Lem}

\begin{proof}
    A $\beta$-set of rank $m$ and defect $k$ is obtained by $1/2(m-k(k+1)/2)$ 2-hooks from $(1\ 3 \ \cdots \ (2k-1))$. Each 2-hook increase the maximum of the coefficients by at most 2, giving us the result.
\end{proof}

\begin{Def}
    Let us define for $\gpfini{G}$ of type $\dAn$,
    \[\kgdc{\gpfini{G}}{d}:=\max\{k\geq 1,(k^2-3k+2)/2\leq n+1-d\}\]
    if it exists and $-1$ otherwise.
\end{Def}

From Lemma \ref{lemmaxAn}, $\kgdc{\gpfini{G}}{d}$ is the greatest integer $k$ such that there exists $\lambda$ of defect $k$ and rank $n+1$ having a hook of length at least $d$. In particular, if $\lambda$ has defect $k > \kgdc{\gpfini{G}}{d}$ and rank $n+1$ then it is a $d$-core.

\begin{Pro}
    Assume that $d'$ is odd (unitary prime case) and $\gpfini{G}$ is of type $\dAn$. Then, the unipotent 1-series with defect strictly greater than $\kgdc{\gpfini{G}}{d'}$ are $\dun$-series, composed uniquely of $d$-cuspidal representations, and the union of the unipotent 1-series with defect lower or equal to $\kgdc{\gpfini{G}}{d'}$ is a $\dun$-series.
\end{Pro}

\begin{proof}
    The $\beta$-sets of rank $n+1$ and defect strictly greater than $\kgdc{\gpfini{G}}{d'}$ are all $d'$-core. Therefore the corresponding unipotent characters are $d$-cuspidal. Thus the unipotent 1-series with defect strictly greater than $\kgdc{\gpfini{G}}{d'}$ are $\dun$-series.

    Since $\dl{\gpfini{G}}{1}$ is a $\dun$-set, we have that the union of the unipotent 1-series with defect lower or equal to $\kgdc{\gpfini{G}}{d'}$ is a $\dun$-set. It remains to prove that it is a $\dun$-series. Let $k \leq \kgdc{\gpfini{G}}{d'}$. Let us assume that $k \geq 3$. Let $\lambda := (1\ 3 \ \cdots \ (2k-3)\ (n+1-(k^2-3k+2)/2) )$ be a $\beta$-set of defect $k$ and rank $n+1$. Let $u$ be an odd number, $1 \leq u \leq 2k-3$ such that $u+d' \neq n+1-(k^2-3k+2)/2 -d'$ (such a $u$ exists since there are more than two odd numbers between $1$ and $2k-3$). Let $\lambda':=(1\ 3 \ \cdots \ u+d'\ \cdots\ (2k-3)\ (n+1-(k^2-3k+2)/2-d') )$ (with a possible permutation of the coefficients so that they are written in the correct order). The $\beta$-set $\lambda'$ is obtain from $\lambda$ by removing a $d'$-hook and then adding a $d'$-hook. Hence $\lambda$ and $\lambda'$ are in the same $d$-series. Since $d'$ is odd, if $k\geq 4$ then $\defect(\lambda')=k-4$ and if $k=3$, $\defect(\lambda')=0$. Hence the unipotent 1-series with defect $k\geq 4$ are in the same $\dun$-series as the unipotent 1-series with defect $k-4$. We have the same result for defects $0$ and $3$. Thus to prove the result, we are left with the 1-series of defect $1$ and $2$. By Lemma \ref{lemexistbsetkm}, depending on the parity of $n$, we can only have simultaneously $\beta$-sets of rank $n+1$ and defects $0,3$ or defects $1,2$. Therefore, we need to prove that, if they exist, the $\beta$-sets with defects $1,2$ are in the same $\dun$-series.

    If there are $\beta$-sets of rank $n+1$ with defects $1,2$. We start by assuming that $n \neq 4$. Either $n \neq 2d'$ or $n \neq 4+2d'$. If $n \neq 2d'$ then we take $\lambda=(1\ n+1)$ and $\lambda'=(1+d'\ n+1-d')$, with $\defect(\lambda)=2$ and $\defect(\lambda')=1$. If $n \neq 4+2d'$ then we take $\lambda=(3\ n-1)$ and $\lambda'=(3+d'\ n-1-d')$, with $\defect(\lambda)=2$ and $\defect(\lambda')=1$ (we can note here that we can well assume that $d' \leq n-1$ because if not then $d'\geq (n+1)$ and we can use the previous case since $n \neq 2d'$). So we are left with $n=4$. We can then have $d'=1,3$ or $5$. If $d'=1$, every $\beta$-set has the same 1-core, so the result follows. If $d'=3$, we take $\lambda=(1\ 3\ 4)$ and $\lambda'=(1 \ 2\ 3\ 5)$. Finally, if $d'=5$ we take $\lambda=(5)$ and $\lambda'=(1\ 5)$.
\end{proof}

In the case $d'$ odd, We will write $\Eun{G}{d}$ for the union of the unipotent 1-series of defect lower or equal to $\kgdc{\gpfini{G}}{d'}$. Thus, if it is not empty, $\Eun{G}{d}$ is the unipotent $\dun$-series containing the trivial representation.

\subsection{Computation of \texorpdfstring{$\dun$}{(d,1)}-series for classical groups}

\label{dunclassicalgroup}
In this section we compute the unipotent $\dun$-series for groups of type $\Bn$, $\Cn$, $\Dn$ and $\dDn$.

\bigskip
Just as before, let us start by studying $d$-series. When $\gpfini{G}$ is a classical group we have a classification of unipotent characters with the notion of symbols that we recall here. Furthermore, with these symbols, we can describe the decomposition into $d$-series of Theorem \ref{thmdecompodseries}.

\sautintro

A symbol is an unordered set $\{S,T\}$ of two subsets $S,T \subseteq \mathbb{N}$. We write such a symbol in the following way
\[ \Sigma = \begin{pmatrix}
        x_1 & \cdots & x_a \\
        y_1 & \cdots & y_b
    \end{pmatrix}\]
with $x_1 < \cdots < x_a$, $y_1 < \cdots < y_b$ and $S=\{x_1,\cdots,x_a\}$, $T=\{y_1,\cdots,y_b\}$. Two symbols are said to be equivalent if they can be transformed into each other by a sequence of steps
\[\begin{pmatrix}
        x_1 & \cdots & x_a \\
        y_1 & \cdots & y_b
    \end{pmatrix}
    \sim
    \begin{pmatrix}
        0 & x_1+1 & \cdots & x_a+1 \\
        0 & y_1+1 & \cdots & y_b+1
    \end{pmatrix} \]
or by interchanging the rows.

We define the defect of $\Sigma$ by $\defect(\Sigma) = |a-b|$ and its rank by
\[ \rank(\Sigma)=\sum_{i=1}^{a}x_i + \sum_{i=1}^{b}y_i - \left[ \left( \frac{a+b-1}{2} \right)^{2}\right].\]
These two notions can be defined on the equivalence classes of symbols.

\bigskip

If $\gpfini{G}$ is a group of type $\Bn$, $\Cn$, $\Dn$ or $\dDn$, Lusztig has shown that the unipotent characters may be parametrized by these symbols (see \cite{lusztigirr}). The unipotent characters of groups of type $\Bn$ or $\Cn$ are in bijection with the equivalence classes of symbols of rank $n$ and odd defect. For the groups of type $\Dn$, the unipotent characters are parametrized by classes of symbols of rank $n$ and defect divisible by 4 (except that if the two rows are identical, two characters correspond to the same symbol). And the unipotent characters of groups of type $\dDn$ are in bijection with symbols of rank $n$ and defect congruent $2 \pmod{4}$.

\bigskip

Let $\{S,T\}$ be a symbol and $d\geq 1$ an integer. If there exists $x\in S$ such that $x+d \notin S$, or $y \in T$ with $y+d \notin T$, then the symbol $\{S\setminus\{x\}\cup\{x+d\},T\}$ or $\{S,T\setminus\{y\}\cup\{y+d\}\}$, is said to be obtained from $\{S,T\}$ by adding a $d$-hook. We define the $d$-core of $\{S,T\}$ as the symbol $\{U,V\}$ without $d$-hook obtained from $\{S,T\}$ by removing a sequence of $d$-hooks.

In the same way, if there exists $x\in S$ such that $x+d \notin T$, or $y \in T$ with $y+d \notin S$, then the symbol $\{S\setminus\{x\},T\cup\{x+d\}\}$ or $\{S\cup\{y+d\},T\setminus\{y\}\}$, is said to be obtained from $\{S,T\}$ by adding a $d$-cohook. And we define like previously the $d$-cocore of $\{S,T\}$.

\begin{Pro}
    \label{prodseriescore}
    \begin{enumerate}
        \item If $d$ is odd: Then the $d$-cuspidal unipotent characters are precisely those where $\Sigma$ is itself a $d$-core. Moreover, two characters are in the same $d$-series if and only if they have the same $d$-core.
        \item If $d$ is even: Then the $d$-cuspidal unipotent characters are precisely those where $\Sigma$ is itself a $d/2$-cocore. Moreover, two characters are in the same $d$-series if and only if they have the same $d/2$-cocore.
    \end{enumerate}
\end{Pro}

\begin{proof}
    This is proved in the proof of Theorem 3.2 in \cite{bmm}.
\end{proof}

Now let us compute the unipotent $\dun$-series. The first case is when $d$ is odd (the linear prime case). To obtain the $d$-series we need to take the $d$-core of the symbols by the proposition \ref{prodseriescore}. We also obtain the 1-series by taking the 1-core. But two symbols which have the same $d$-core have the same 1-core, so each unipotent 1-series is a $\dun$-series.

\begin{Pro}
    \label{prodimpaire}
    If $d$ is odd (linear prime case), the unipotent $\dun$-series are the unipotent 1-series.
\end{Pro}

Now, assume that $d$ is even (unitary prime case). This case is a little bit more complicated because we need to take the $d/2$-cocore for the $d$-series and the 1-core for the 1-series. We will do a proof similar to the case of $\dAn$ done in Section \ref{secdunAn}.

Let $\Sigma$ be a symbol. We define $\max(\Sigma)$ to be $\max(\Sigma):=0$ if $\Sigma\sim\{\emptyset,\emptyset\} $, and otherwise $\max(\Sigma):=\max(S\cup T)$ where $\{S,T\}$ is the unique symbol equivalent to $\Sigma$ with $0 \notin S \cap T$. Note that $\max(\Sigma)$ is the longest length of a hook or cohook in $\Sigma$.

\begin{Lem}
    \label{lemmaxdefectk}
    Let $k\geq 0$ and $n\geq 1$ such that $n \geq \frac{k^2-1}{4}$. We have

    \[ \max\{ \max(\Sigma), \defect(\Sigma)=k, \rank(\Sigma)=n\}=
        \begin{cases}
            n-\frac{k^2-4k+3}{4} & \text{ if } k \text{ is odd}            \\
            n-\frac{k^2-4k+4}{4} & \text{ if } k \text{ is even}, k \neq 0 \\
            n                    & \text{ if } k = 0
        \end{cases}
    \]

\end{Lem}

\begin{proof}
    Every symbol of defect $k$ is obtained from $\Sigma_k=\begin{pmatrix}
            0 & \cdots & k-1 \\
              &        &
        \end{pmatrix}$, for $k \geq 1$, and $\Sigma_0=\{\emptyset,\emptyset\}$, for $k=0$, by adding 1-hooks. Each 1-hook increases the rank of 1. So in order to get a symbol of rank $n$, we need to do $m:= n-\rank(\Sigma_k)$ 1-hooks. Note that, for $k \geq 1$,

    \[\rank(\Sigma_k)=\frac{(k-1)k}{2}-\left[\left(\frac{k-1}{2}\right)^2\right]=
        \begin{cases}
            \frac{k^2-1}{4} & \text{ if } k \text{ is odd}  \\
            \frac{k^2}{4}   & \text{ if } k \text{ is even}
        \end{cases}
        ,\]
    and $\rank(\Sigma_0)=0$.
    Remark also, that the hypothesis $n \geq (k^2-1)/4$ is equivalent to $m\geq0$. Each 1-hook increases the maximum of the coefficients by at most one, so $\max\{ \max(\Sigma), \defect(\Sigma)=k\}=k-1+m$, for $k\geq1$, and $m$ for $k=0$ (we have equality by adding the 1-hooks on the last coefficient on the top row).
\end{proof}

Let us define an integer $\kgd$ in the following way.

\begin{Def}
    If $\gpfini{G}$ is of type $\Bn$ or $\Cn$ we define
    \[\kgd=\max\{k\geq 1, k \text{ odd},(k^2-4k+3)/4\leq n-d/2\}\]
    if it exists and $\kgd=-1$ otherwise.

    If $\gpfini{G}$ is of type $\Dn$ or $\dDn$ then in the same way
    \[\kgd=\max\{k\geq 2, k \text{ even},(k^2-4k+4)/4\leq n-d/2\}\]
    if it exists and $\kgd=-1$ otherwise.
\end{Def}

As in the case of $\dAn$, by Lemma \ref{lemmaxdefectk} we see that $\kgd$ is the greatest integer $k$ such that there exists $\Sigma$ of defect $k$ and rank $n$ having a cohook of length at least $d$. In particular, if $\defect(\Sigma)> \kgdc{\gpfini{G}}{d}$ then $\Sigma$ is a $d$-cocore.

\begin{Rem}
    Two symbols are in the same 1-series if and only if they have the same 1-core by Proposition \ref{prodseriescore}. But removing a 1-hook does not change the defect of a symbol. Hence, every symbol in a 1-series has the same defect. Moreover, the 1-core of a symbol is of the form $\begin{pmatrix}
            0 & \cdots & k-1 \\
              &        &
        \end{pmatrix}$ where $k$ is the defect of the symbol (or $\{\emptyset,\emptyset\}$ when the defect is 0). Hence, two symbols are in the same 1-series if and only if they have the same defect. And the defect associated with a 1-series is the defect of the cuspidal representation associated to this 1-series.
\end{Rem}

We have the following partition of $\dl{\gpfini{G}}{1}$ into $\dun$-series.
\begin{Pro}
    \label{prod1series}
    If $d$ is even (unitary prime case), the unipotent 1-series with defect strictly greater than $\kgd$ are $\dun$-series, composed uniquely of $d$-cuspidal representations, and the union of the unipotent 1-series with defect lower or equal to $\kgd$ is a $\dun$-series.
\end{Pro}

\begin{proof}
    Let $k>\kgd$ and a unipotent 1-series with defect $k$. Then by definition of $\kgd$ and with Lemma \ref{lemmaxdefectk}, $d/2$ is strictly greater than every coefficient in every symbol in the 1-series chosen. Hence, this 1-series is composed of $d$-cuspidal representations, so is a $\dun$-series.

    We also deduce from that, that the union of the unipotent 1-series with defect lower or equal to $\kgd$ is a $\dun$-set. It remains to prove that this is a $\dun$-series. Let $3\leq k \leq \kgd$ such that there is a unipotent 1-series with defect $k$. We want to prove that the unipotent 1-series with defect $k$ is in the same $\dun$-series as a unipotent 1-series with defect strictly less than $k$, which will finish the proof. Let $\Sigma_k=\begin{pmatrix}
            0 & \cdots & k-1 \\
              &        &
        \end{pmatrix}$ and $m= n-\rank(\Sigma_k)$ as in the proof of Lemma \ref{lemmaxdefectk}. Then the symbol
    \[\Sigma=\begin{pmatrix}
            0 & \cdots & k-2 & k-1+m \\
              &        &     &
        \end{pmatrix}\]
    has defect $k$ and rank $n$ so is in the 1-series chosen. Now by definition of $\kgd$, $d/2 \leq k-1+m$, we can then remove a $d/2$-cohook from $\Sigma$ to get
    \[\Sigma'=\begin{pmatrix}
            0         & \cdots & k-2 \\
            k-1+m-d/2 &        &
        \end{pmatrix}.\]
    Let $v \in \{0,\cdots,k-2\}$ such that $v+d/2 \neq k-1+m-d/2$. Then we can add a $d/2$-cohook to $\Sigma'$ to obtain
    \[\Sigma''=\begin{pmatrix}
            0         & \cdots & v-1 & v+1 & \cdots & k-2 \\
            k-1+m-d/2 & v+d/2
        \end{pmatrix}\]
    (we possibly have to swap the numbers in the lower row so that they are written in the good order). The symbol $\Sigma''$ is a symbol of defect $k-4$ if $k>3$ and $k-2$ if $k=3$, which has the same $d/2$-cocore as $\Sigma$. Hence, $\Sigma$ and $\Sigma$ are in the same $\dun$-series, and $\defect(\Sigma')< \defect(\Sigma)$.
\end{proof}

As before, when $d$ is even, we write $\Eun{G}{d}$ for the union of the 1-series of defect lower or equal to $\kgd$, which is, if not empty, the $\dun$-series containing the trivial representation.

\subsection{Computation of \texorpdfstring{$\dun$}{(d,1)}-series for exceptional groups}
We have computed the unipotent $\dun$-series for groups of type $\textbf{A}$ and for classical groups. We are left with groups of exceptional type, that is of type $\tDq$, $\dG$, $\Fqu$, $\Esi$, $\dEsi$, $\Ese$ and $\Eh$.

\bigskip

Unfortunately, we do not have a nice classification with partitions or symbols like for groups of types $\textbf{A}$, $\textbf{B}$, $\textbf{C}$ and $\textbf{D}$. However, since we are working with groups with bounded rank, we can do a case by case analysis. We will summarize the result in Table \ref{tabdunexp}. We need to explain the notations used. To keep the notation as simple as possible, we are writing the unipotent $\dun$-series in terms of 1-series. We will write a 1-series by the corresponding 1-cuspidal representation of the 1-split Levi defining this series. The notations for the cuspidal representations are the notations of \cite{carter} section 13.9. So for example for $\Fqu$, we have a (2,1)-series $\{1, \B{2}, \Fqu[-1], \Fqu[i], \Fqu''[1]\}$. This set is composed of the principal series (denoted by $1$), the characters coming from the unipotent cuspidal character of $\B{2}$ (denoted by $\B{2}$) and 3 cuspidal representations of $\Fqu$ : $\Fqu[-1]$, $\Fqu[i]$ and $\Fqu''[1]$. Thus $\{1, \B{2}, \Fqu[-1], \Fqu[i], \Fqu''[1]\}$ denotes a set composed of 33 unipotent characters.

If a $d$ does not appear in Table \ref{tabdunexp}, it means that the unipotent $\dun$-series are the unipotent $1$-series.

\begin{Pro}
    The unipotent $\dun$-series for groups of exceptional types are written in the Table \ref{tabdunexp}.
\end{Pro}

\begin{proof}
    In \cite{carter} section 13.9 we can find tables for the unipotent characters of groups of exceptional types and the partitions into 1-series. So to compute the unipotent $\dun$-series, we need to know about the $d$-series. In \cite{bmm}, we find in Tables 1 and 2 a list of the $d$-series $\dldcusp{L}{\lambda}$, where $(\gpfinialg{L},\lambda)$ is a unipotent $d$-cuspidal pair and $\gpfinialg{L}$ is not a torus. So we are missing the cases of $\gpfinialg{L}$ a torus (hence $\lambda$ is trivial). However, in the case $\gpfinialg{L}=\gpfinialg{T}$ of a torus, the Deligne-Lusztig induction $\inddllite{\gpfini{T}}{\gpfini{G}}$ is known by the work of Lusztig. Hence combining all the computations, we prove the results of Table \ref{tabdunexp}.
\end{proof}

\begin{longtable}{|c|l|l|}

    \hline
    Group   & $d$       & unipotent $\dun$-series                                                                                                  \\
    \hline
    $\dG$   & 2         & $\{1, \dG[1], \dG[-1]\}$, $\{\dG[\theta]\}$, $\{\dG[\theta^2]\}$                                                         \\
            & 3         & $\{1, \dG[1], \dG[\theta], \dG[\theta^2]\}$, $\{\dG[-1]\}$                                                               \\
            & 6         & $\{1, \dG[-1], \dG[\theta], \dG[\theta^2]\}$, $\{\dG[1]\}$                                                               \\

    \hline

    $\tDq$  & 2, 6      & $\{1, \tDq[1], \tDq[-1]\}$                                                                                               \\
            & 3         & $\{1, \tDq[1] \}$, $\{\tDq[-1]\}$                                                                                        \\
            & 12        & $\{1, \tDq[-1] \}$, $\{\tDq[1]\}$                                                                                        \\
    \hline

    $\Fqu$  & 2         & $\{1, \B{2}, \Fqu[-1], \Fqu[i], \Fqu''[1]\}$, $\{\Fqu[-i]\}$, $\{\Fqu[\theta]\}$, $\{\Fqu[\theta^2]\}$,                  \\
            &           & $\{\Fqu'[1]\}$                                                                                                           \\
            & 3         & $\{1, \Fqu[\theta], \Fqu[\theta^2],\Fqu'[1]\}$, $\{\B{2}\}$, $\{\Fqu[-i]\}$, $\{\Fqu[-1]\}$, $\{\Fqu[i]\}$,              \\
            &           & $\{\Fqu''[1]\}$                                                                                                          \\
            & 4         & $\{1, \B{2}, \Fqu[-i], \Fqu[i], \Fqu'[1], \Fqu''[1]\}$, $\{\Fqu[-1]\}$, $\{\Fqu[\theta]\}$,                              \\
            &           & $\{\Fqu[\theta^2]\}$                                                                                                     \\
            & 6         & $\{1, \B{2}, \Fqu[-1], \Fqu[\theta], \Fqu[\theta^2], \Fqu'[1] \}$, $\{\Fqu[-i]\}$, $\{\Fqu[i]\}$,                        \\
            &           & $\{\Fqu''[1]\}$                                                                                                          \\
            & 8         & $\{1, \Fqu[-1], \Fqu[-i], \Fqu[i] \}$, $\{\B{2}\}$, $\{\Fqu'[1]\}$, $\{\Fqu[\theta]\}$, $\{\Fqu[\theta^2]\}$,            \\
            &           & $\{\Fqu''[1]\}$                                                                                                          \\
            & 12        & $\{1, \B{2}, \Fqu[-i], \Fqu[i], \Fqu[\theta], \Fqu[\theta^2]\}$, $\{\Fqu[-1]\}$, $\{\Fqu'[1]\}$,                         \\
            &           & $\{\Fqu''[1]\}$                                                                                                          \\

    \hline

    $\Esi$  & 2, 4, 8   & $\{1, \D{4}\}$, $\{\Esi[\theta]\}$, $\{\Esi[\theta^2]\}$                                                                 \\
            & 3, 9      & $\{1, \Esi[\theta], \Esi[\theta^2]\}$, $\{\D{4}\}$                                                                       \\
            & 5         & $\{1\}$, $\{\D{4}\}$, $\{\Esi[\theta]\}$, $\{\Esi[\theta^2]\}$                                                           \\
            & 6, 12     & $\{1, \D{4}, \Esi[\theta], \Esi[\theta^2]\}$                                                                             \\

    \hline

    $\Ese$  & 2, 10, 14 & $\{1, \D{4}, \Ese[\xi], \Ese[-\xi]\}$, $\{\Esi[\theta]\}$, $\{\Esi[\theta^2]\}$                                          \\
            & 3, 9      & $\{1, \Esi[\theta], \Esi[\theta^2]\}$, $\{ \D{4}\}$, $\{\Ese[\xi]\}$, $\{ \Ese[-\xi]\}$                                  \\
            & 4, 8      & $\{1, \D{4}\}$, $\{\Esi[\theta]\}$, $\{\Esi[\theta^2]\}$, $\{\Ese[\xi]\}$, $\{ \Ese[-\xi]\}$                             \\
            & 5, 7      & $\{1\}$, $\{\D{4}\}$, $\{\Esi[\theta]\}$, $\{\Esi[\theta^2]\}$, $\{\Ese[\xi]\}$, $\{ \Ese[-\xi]\}$                       \\
            & 6, 18     & $\{1, \D{4}, \Esi[\theta], \Esi[\theta^2], \Ese[\xi], \Ese[-\xi]\}$                                                      \\
            & 12        & $\{1, \D{4}, \Esi[\theta], \Esi[\theta^2]\}$, $\{\Ese[\xi]\}$, $\{ \Ese[-\xi]\}$                                         \\

    \hline

    $\Eh$   & 2         & $\{1, \D{4}, \Ese[\xi], \Ese[-\xi], \Eh[-1], \Eh'[1], \Eh''[1]\}$,                                                       \\
            &           & $\{\Esi[\theta], \Eh[-\theta], \Eh[\theta]\}$, $\{\Esi[\theta^2], \Eh[\theta^2], \Eh[-\theta^2]\}$, $\{\Eh[-i] \}$,      \\
            &           & $\{\Eh[\zeta^4]\}$, $\{\Eh[\zeta^3] \}$, $\{\Eh[\zeta^2] \}$, $\{\Eh[\zeta] \}$, $\{\Eh[i] \}$                           \\
            & 3         & $\{1, \Esi[\theta], \Esi[\theta^2], \Eh[\theta^2], \Eh[\theta], \Eh'[1]\}$,                                              \\
            &           & $\{\D{4}, \Eh[-1], \Eh[-\theta^2], \Eh[-\theta]\}$, $\{\Ese[-\xi]\}$, $\{ \Ese[\xi]\}$, $\{\Eh''[1] \}$,                 \\
            &           & $\{\Eh[-i] \}$, $\{\Eh[\zeta^4]\}$, $\{\Eh[\zeta^3] \}$, $\{\Eh[\zeta^2] \}$, $\{\Eh[\zeta] \}$, $\{\Eh[i] \}$           \\
            & 4         & $\{1, \D{4}, \Eh[-i], \Eh[i], \Eh'[1], \Eh''[1]\}$, $\{\Ese[-\xi] \}$, $\{\Ese[\xi] \}$,                                 \\
            &           & $\{\Esi[\theta]\}$, $\{\Esi[\theta^2] \}$, $\{\Eh[\zeta^4]\}$, $\{\Eh[\zeta^3] \}$, $\{\Eh[\zeta^2] \}$,                 \\
            &           & $\{\Eh[\zeta] \}$,  $\{\Eh[-1]\}$, $\{ \Eh[-\theta] \}$, $\{ \Eh[\theta]\}$, $\{\Eh[\theta^2] \}$, $\{\Eh[-\theta^2] \}$ \\
            & 5         & $\{1, \Eh[\zeta^4], \Eh[\zeta^3], \Eh[\zeta^2], \Eh[\zeta], \Eh'[1]\}$, $\{\Ese[-\xi] \}$, $\{\Ese[\xi] \}$,             \\
            &           & $\{ \D{4}\}$, $\{\Esi[\theta]\}$, $\{\Esi[\theta^2] \}$, $\{ \Eh[-i]\}$, $\{\Eh[i] \}$, $\{\Eh''[1] \}$,                 \\
            &           & $\{\Eh[-1]\}$, $\{ \Eh[-\theta] \}$, $\{ \Eh[\theta]\}$, $\{\Eh[\theta^2] \}$, $\{\Eh[-\theta^2] \}$                     \\
            & 6         & $\{1, \Ese[-\xi], \Ese[\xi], \D{4}, \Esi[\theta], \Esi[\theta^2], \Eh[-1], \Eh[-\theta^2], \Eh[-\theta],$                \\
            &           & $ \Eh[\theta^2], \Eh[\theta], \Eh'[1], \Eh''[1]\}$, $\{\Eh[-i] \}$, $\{\Eh[\zeta^4]\}$, $\{\Eh[\zeta^3] \}$,             \\
            &           & $\{\Eh[\zeta^2] \}$, $\{\Eh[\zeta] \}$, $\{\Eh[i] \}$                                                                    \\
            & 7         & $\{1\}$, $\{\D{4}\}$, $\{\Ese[-\xi] \}$, $\{\Ese[\xi] \}$, $\{\Esi[\theta]\}$, $\{\Esi[\theta^2] \}$, $\{ \Eh[-i]\}$,    \\
            &           & $\{\Eh[i]\}$, $\{\Eh'[1]\}$,  $\{\Eh''[1]\}$, $\{\Eh[\zeta^4]\}$, $\{\Eh[\zeta^3] \}$, $\{\Eh[\zeta^2] \}$,              \\
            &           & $\{\Eh[\zeta] \}$, $\{\Eh[-1]\}$, $\{ \Eh[-\theta] \}$, $\{ \Eh[\theta]\}$, $\{\Eh[\theta^2] \}$, $\{\Eh[-\theta^2] \}$  \\
            & 8         & $\{1, \D{4}, \Eh[-1], \Eh[-i], \Eh[i]\}$, $\{\Ese[-\xi] \}$, $\{\Ese[\xi] \}$, $\{\Esi[\theta]\}$,                       \\
            &           & $\{\Esi[\theta^2] \}$, $\{\Eh[\zeta^4]\}$, $\{\Eh[\zeta^3] \}$, $\{\Eh[\zeta^2] \}$, $\{\Eh[\zeta] \}$,  $\{\Eh'[1]\}$,  \\
            &           & $\{ \Eh[-\theta] \}$, $\{ \Eh[\theta]\}$, $\{\Eh[\theta^2] \}$, $\{\Eh[-\theta^2] \}$, $\{ \Eh''[1]\}$                   \\
            & 9         & $\{1, \Esi[\theta], \Esi[\theta^2]\}$, $\{\D{4}\}$, $\{\Ese[-\xi] \}$, $\{\Ese[\xi] \}$, $\{ \Eh[-i]\}$,                 \\
            &           & $\{\Eh[i]\}$, $\{\Eh'[1]\}$,  $\{\Eh''[1]\}$, $\{\Eh[\zeta^4]\}$, $\{\Eh[\zeta^3] \}$, $\{\Eh[\zeta^2] \}$,              \\
            &           & $\{\Eh[\zeta] \}$, $\{\Eh[-1]\}$, $\{ \Eh[-\theta] \}$, $\{ \Eh[\theta]\}$, $\{\Eh[\theta^2] \}$, $\{\Eh[-\theta^2] \}$  \\
            & 10        & $\{1, \Ese[-\xi], \Ese[\xi], \D{4}, \Eh[-1], \Eh[\zeta^4], \Eh[\zeta^3], \Eh[\zeta^2], \Eh[\zeta],$                      \\
            &           & $\Eh''[1]\}$, $\{\Esi[\theta]\}$, $\{\Esi[\theta^2] \}$, $\{\Eh[-i]\}$, $\{ \Eh[i]\}$, $\{\Eh'[1] \}$,                   \\
            &           & $\{ \Eh[-\theta] \}$, $\{ \Eh[\theta]\}$, $\{\Eh[\theta^2] \}$, $\{\Eh[-\theta^2] \}$                                    \\
            & 12        & $\{1, \D{4}, \Esi[\theta], \Esi[\theta^2], \Eh[-1], \Eh[-\theta^2], \Eh[-\theta], \Eh[-i], \Eh[\theta^2], $              \\
            &           & $\Eh[\theta], \Eh[i], \Eh''[1]\}$, $\{\Ese[-\xi] \}$, $\{\Ese[\xi] \}$, $\{\Eh'[1]\}$, $\{\Eh[\zeta^4]\}$,               \\
            &           & $\{\Eh[\zeta^3] \}$, $\{\Eh[\zeta^2] \}$, $\{\Eh[\zeta] \}$                                                              \\
            & 14        & $\{1, \Ese[-\xi], \Ese[\xi], \D{4} \}$, $\{\Esi[\theta]\}$, $\{\Esi[\theta^2] \}$, $\{ \Eh[-i]\}$, $\{\Eh[i]\}$,         \\
            &           & $\{\Eh'[1]\}$, $\{\Eh''[1]\}$, $\{\Eh[\zeta^4]\}$, $\{\Eh[\zeta^3] \}$, $\{\Eh[\zeta^2] \}$, $\{\Eh[\zeta] \}$,          \\
            &           & $\{\Eh[-1]\}$, $\{ \Eh[-\theta] \}$, $\{ \Eh[\theta]\}$, $\{\Eh[\theta^2] \}$, $\{\Eh[-\theta^2] \}$                     \\
            & 15        & $\{1, \Esi[\theta], \Esi[\theta^2], \Eh[\theta^2],\Eh[\theta], \Eh[\zeta^4], \Eh[\zeta^3], \Eh[\zeta^2], \Eh[\zeta] \}$, \\
            &           & $\{\D{4}\}$, $\{\Ese[-\xi] \}$, $\{\Ese[\xi] \}$, $\{ \Eh[-i]\}$, $\{\Eh[i]\}$, $\{\Eh'[1]\}$,                           \\
            &           & $\{\Eh''[1]\}$, $\{\Eh[-1]\}$, $\{ \Eh[-\theta] \}$, $\{\Eh[-\theta^2] \}$                                               \\
            & 18        & $\{1, \Ese[-\xi], \Ese[\xi], \D{4}, \Esi[\theta], \Esi[\theta^2], \Eh[-\theta^2], \Eh[-\theta], \Eh[\theta^2],$          \\
            &           & $\Eh[\theta] \}$, $\{ \Eh[-i]\}$, $\{\Eh[i]\}$, $\{\Eh'[1]\}$,  $\{\Eh''[1]\}$, $\{\Eh[\zeta^4]\}$,                      \\
            &           & $\{\Eh[\zeta^3] \}$, $\{\Eh[\zeta^2] \}$, $\{\Eh[\zeta] \}$,  $\{\Eh[-1]\}$                                              \\
            & 20        & $\{1, \D{4}, \Eh[-i], \Eh[\zeta^4], \Eh[\zeta^3], \Eh[\zeta^2], \Eh[\zeta], \Eh[i]\}$, $\{\Ese[-\xi] \}$,                \\
            &           & $\{\Ese[\xi] \}$, $\{\Esi[\theta]\}$, $\{\Esi[\theta^2] \}$, $\{\Eh'[1]\}$, $\{\Eh''[1]\}$, $\{\Eh[-1]\}$,               \\
            &           & $\{ \Eh[-\theta] \}$, $\{ \Eh[\theta]\}$, $\{\Eh[\theta^2] \}$, $\{\Eh[-\theta^2] \}$                                    \\
            & 24        & $\{1, \D{4}, \Esi[\theta], \Esi[\theta^2], \Eh[-\theta^2], \Eh[-\theta], \Eh[-i], \Eh[i]\}$,                             \\
            &           & $\{\Ese[-\xi] \}$, $\{\Ese[\xi] \}$, $\{\Eh'[1]\}$, $\{\Eh''[1]\}$, $\{\Eh[\zeta^4]\}$, $\{\Eh[\zeta^3] \}$,             \\
            &           & $\{\Eh[\zeta^2] \}$, $\{\Eh[\zeta] \}$,  $\{\Eh[-1]\}$, $\{ \Eh[\theta]\}$, $\{\Eh[\theta^2] \}$                         \\
            & 30        & $\{1, \Ese[-\xi], \Ese[\xi], \D{4}, \Esi[\theta], \Esi[\theta^2], \Eh[-\theta^2], \Eh[-\theta], \Eh[\zeta^4],$           \\
            &           & $ \Eh[\zeta^3], \Eh[\zeta^2], \Eh[\zeta]\}$, $\{ \Eh[-i]\}$, $\{\Eh[i]\}$, $\{\Eh'[1]\}$, $\{\Eh''[1]\}$,                \\
            &           & $\{\Eh[-1]\}$, $\{ \Eh[\theta]\}$, $\{\Eh[\theta^2] \}$                                                                  \\

    \hline

    $\dEsi$ & 2         & $\{1, \dA{5}, \dEsi[1]\}$, $\{\dEsi[\theta]\}$, $\{\dEsi[\theta^2]\}$                                                    \\
            & 3         & $\{1, \dEsi[1], \dEsi[\theta], \dEsi[\theta^2]\}$, $\{ \dA{5} \}$                                                        \\
            & 4         & $\{1, \dEsi[1]\}$, $\{ \dA{5} \}$, $\{\dEsi[\theta]\}$, $\{\dEsi[\theta^2]\}$                                            \\
            & 6         & $\{1, \dA{5}, \dEsi[1] , \dEsi[\theta], \dEsi[\theta^2]\}$                                                               \\
            & 8         & $\{1\}$, $\{ \dA{5} \}$, $\{ \dEsi[1]\}$, $\{\dEsi[\theta]\}$, $\{\dEsi[\theta^2]\}$                                     \\
            & 10        & $\{1, \dA{5}\}$, $\{ \dEsi[1]\}$, $\{\dEsi[\theta]\}$, $\{\dEsi[\theta^2]\}$                                             \\
            & 12        & $\{1, \dEsi[\theta], \dEsi[\theta^2]\}$, $\{ \dA{5} \}$, $\{ \dEsi[1] \}$                                                \\
            & 18        & $\{1, \dA{5}, \dEsi[\theta], \dEsi[\theta^2]\}$, $\{ \dEsi[1] \}$                                                        \\

    \hline

    \caption{unipotent $\dun$-series for groups of exceptional types\label{tabdunexp}}
\end{longtable}

\subsection{Summary for unipotent \texorpdfstring{$\dun$}{(d,1)}-series}

In this section, we summarize all the computations of the unipotent $\dun$-series.

\bigskip

First let us recall some definition. For an integer $d$ we define $d'$ by
\[
    d' = \left\{
    \begin{array}{ll}
        2d  & \mbox{if } d \mbox{ is odd}    \\
        d/2 & \mbox{if } d \equiv 2 \pmod{4} \\
        d   & \mbox{if } d \equiv 0 \pmod{4}
    \end{array}
    \right.
\]

We also have defined $\kgd$ by
\[
    \kgd = \left\{
    \begin{array}{ll}
        \max\{k\geq 1,(k^2-3k+2)/2\leq n+1-d\}                 & \mbox{for type } \dAn       \\
        \max\{k\geq 1, k \text{ odd},(k^2-4k+3)/4\leq n-d/2\}  & \mbox{for types } \Bn, \Cn  \\
        \max\{k\geq 2, k \text{ even},(k^2-4k+4)/4\leq n-d/2\} & \mbox{for types } \Dn, \dDn
    \end{array}
    \right.
\]
if it exists and $-1$ otherwise.

\begin{The}
    \label{theresumdun}
    The unipotent $\dun$-series are given by the following cases
    \begin{enumerate}
        \item Type $\An$ : $\dl{\gpfini{G}}{1}$ is a $\dun$-series
        \item Type $\dAn$:
              \begin{enumerate}
                  \item $d'$ even (linear prime case): the unipotent $\dun$-series are the unipotent 1-series.
                  \item $d'$ odd (unitary prime case): the $\dun$-series are,
                        \begin{itemize}
                            \item the unipotent 1-series with defect strictly greater than $\kgdc{G}{d'}$ (composed uniquely of $d$-cuspidal representations);
                            \item $\Eun{G}{d}$, the union of the unipotent 1-series with defect lower or equal to $\kgdc{G}{d'}$.
                        \end{itemize}
              \end{enumerate}
        \item Type $\Bn$, $\Cn$, $\Dn$ and $\dDn$:
              \begin{enumerate}
                  \item $d$ odd (linear prime case): the unipotent $\dun$-series are the unipotent 1-series
                  \item $d$ even (unitary prime case): the $\dun$-series are,
                        \begin{itemize}
                            \item the unipotent 1-series with defect strictly greater than $\kgd$ (composed uniquely of $d$-cuspidal representations);
                            \item $\Eun{G}{d}$, the union of the unipotent 1-series with defect lower or equal to $\kgd$.
                        \end{itemize}
              \end{enumerate}
        \item Type $\tDq$, $\dG$, $\Fqu$, $\Esi$, $\dEsi$, $\Ese$ and $\Eh$ : the unipotent $\dun$-series are given by Table \ref{tabdunexp}
    \end{enumerate}
\end{The}

\subsection{Induction and restriction of \texorpdfstring{$\dun$}{(d,1)}-series}

Now that we know how to compute the unipotent $\dun$-series, we want to prove that they are compatible with Harish-Chandra induction and restriction. In particular, it will be fundamental in order to construct unipotent $\lprime$-blocks of $p$-adic groups, to prove that Harish-Chandra restriction  commutes with taking the $\lprime$-extension of unipotent $\dun$-series.

\bigskip

Let $\gpfinialg{M}$ be a $\fr$-stable Levi of $\gpfinialg{G}$ and $\mathcal{E}$ a subset of $\Irr(\gpfini{M})$. We denote by $\inddllite{\gpfinialg{M}}{\gpfinialg{G}}(\mathcal{E})$ the set of irreducible characters $\pi$ of $\gpfini{G}$ such that there exists $\sigma \in \mathcal{E}$ satisfying $\langle \pi, \inddl{M}{P}{G}(\sigma)\rangle \neq 0$, for $\gpfinialg{P}$ a parabolic subgroup admitting $\gpfinialg{M}$ as a Levi subgroup. When $\gpfinialg{M}$ is a 1-split Levi of $\gpfinialg{G}$, we will simply use the notation $\indParalite{\gpfini{M}}{\gpfini{G}}(\mathcal{E})$. In the same way, for any 1-split Levi $\gpfinialg{M}$ of $\gpfinialg{G}$ and $\mathcal{E}'$ a subset of $\Irr(\gpfini{G})$, $\resParalite{\gpfini{M}}{\gpfini{G}}(\mathcal{E}')$ denotes the set of characters $\sigma$ such that there exists $\pi \in \mathcal{E}'$ satisfying $\langle \sigma, \resPara{M}{P}{G}(\pi)\rangle \neq 0$.

\bigskip

The $\dun$-series are a union of $1$-series and of $d$-series. We know that the Harish-Chandra induction of a $1$-series is included in a $1$-series. But there is no nice result for the Harish-Chandra induction of a $d$-series. The following results have for goal to prove that the $\dun$-series behave well regarding Harish-Chandra induction.

\begin{Lem}
    \label{leminducdun}
    Let $\gpfinialg{M}$ be a 1-split Levi of $\gpfinialg{G}$ and $\mathcal{E} \subseteq \dl{\gpfini{M}}{1}$ a $\dun$-series. Then $\indParalite{\gpfini{M}}{\gpfini{G}}(\mathcal{E})$ is included in a $\dun$-series.
\end{Lem}

\begin{proof}
    By Propositions \ref{produnserieadjoint} and \ref{produnscalaire}, we can assume that $\gpfinialg{G}$ is simple.

    \begin{enumerate}

        \item If $\gpfinialg{G}$ is of type $\An$, then $\dl{\gpfini{G}}{1}$ is a $\dun$-series so we have the result.
        \item If $\gpfinialg{G}$ is of type $\Bn$, $\Cn$, $\Dn$ or $\dDn$. The Levi $\gpfinialg{M}$ has type $\gl{n_1}\times \cdots \times \gl{n_r} \times \gpfinialg{H}$ where $\gpfinialg{H}$ as the same type as $\gpfinialg{G}$. We deduce that $\mathcal{E}\simeq \dl{\gl{n_1}(q)}{1} \times \cdots \times \dl{\gl{n_r}(q)}{1} \times \mathcal{E}_{H} $, where $\mathcal{E}_{H}$ is a $\dun$-series of $\gpfini{H}$. We need to differentiate the case $d$ odd and $d$ even.

              If $d$ is odd, then $\mathcal{E}_{H}$ is a 1-series by Proposition \ref{prodimpaire}  and so is $\mathcal{E}$. The set $\indParalite{\gpfini{M}}{\gpfini{G}}(\mathcal{E})$ is thus included in a 1-series and so in a $\dun$-series.

              If $d$ is even, then by Proposition \ref{prod1series}, $\mathcal{E}_{H}$ is either a $1$-series or $\mathcal{E}_{H}=\Eun{H}{d}$, where $\Eun{H}{d}$ is the union of the 1-series with defect lower or equal to $\kgdc{\gpfini{H}}{d}$. If it is a $1$-series, we have the result like previously. And if $\mathcal{E}_{H}=\Eun{H}{d}$, since $\kgdc{\gpfini{H}}{d} \leq \kgd$ and the fact that the induction preserves the defect, $\indParalite{\gpfini{M}}{\gpfini{G}}(\mathcal{E}) \subseteq \Eun{G}{d}$ which is a $\dun$-series.

        \item If $\gpfinialg{G}$ is of type $\dAn$. Using ``Ennola''-duality, $\gpfinialg{M}$ corresponds to a 2-split Levi of $\gl{n}$, and the unipotent $\dun$-series correspond to $\dung{d'}{2}$-series. The proof is then the same as in $(2)$ regarding that $d'$ is odd or even.

        \item If $\gpfinialg{G}$ is of exceptional type. The proof mainly consists of checking case by case the result using Table \ref{tabdunexp}. We explain here the arguments to do so.

              The first case to remark is when all the unipotent $\dun$-series not containing the trivial representation are composed uniquely of 1-cuspidal representations. In this case, we have directly the result. This happens for approximately half the cases by looking at Table \ref{tabdunexp} and deals completely with $\dG$ and $\tDq$. Now, when $d$ is odd, and the Levi $\gpfinialg{M}$ has only component of types $\An$, $\Bn$, $\Cn$, $\Dn$ or $\dDn$, we know that the unipotent $\dun$-series are $1$-series. Since the induction of a 1-series from a 1-split Levi is included in a 1-series, we get the result. This is enough to deal with $\Esi$. We also get the odd $d$ for $\Ese$, respectively $\Eh$, by checking the compatibility from $\Esi$, respectively $\Ese$, thanks to Table \ref{tabdunexp}. The same argument works for $\dEsi$ but when $d'$ is even (we recall that $d'$ is defined in section \ref{secdunAn}). To finish $\Ese$ and $\Eh$, we need to look when $d$ is even. In all these cases, the 1-series corresponding to the unipotent cuspidal representation of $\D{4}$ is inside the $\dun$-series containing the trivial representation. So we just have to check with Table \ref{tabdunexp} the compatibility with $\Esi$ and $\Ese$. We are left with the last case of $\Fqu$ and $d=8$. But in this case $\cycl{8}$ does not appear in any of the polynomial orders of the 1-split-Levi, which concludes the proof.
    \end{enumerate}
\end{proof}

\begin{Lem}
    Let $\gpfinialg{M}$ be a 1-split Levi of $\gpfinialg{G}$ and $\mathcal{E} \subseteq \dl{\gpfini{G}}{1}$ a $\dun$-series. Then $\resParalite{\gpfini{M}}{\gpfini{G}}(\mathcal{E})$ is a $\dun$-set.
\end{Lem}

\begin{proof}
    Let $\sigma \in \resParalite{\gpfini{M}}{\gpfini{G}}(\mathcal{E})$. There exists $\mathcal{E}'$ a unipotent $\dun$-series in $\gpfini{M}$ such that $\sigma \in \mathcal{E}'$. We need to prove that $\mathcal{E}' \subseteq \resParalite{\gpfini{M}}{\gpfini{G}}(\mathcal{E})$.

    Since $\sigma \in \resParalite{\gpfini{M}}{\gpfini{G}}(\mathcal{E})$, there exists $\pi \in \mathcal{E}$ such that $\langle \sigma , \resParalite{\gpfini{M}}{\gpfini{G}}(\pi) \rangle \neq 0$. By Frobenius reciprocity, $\langle \indParalite{\gpfini{M}}{\gpfini{G}}(\sigma) , \pi \rangle \neq 0$, thus $\pi \in \indParalite{\gpfini{M}}{\gpfini{G}}(\mathcal{E}')$. By Lemma \ref{leminducdun}, $\indParalite{\gpfini{M}}{\gpfini{G}}(\mathcal{E}')$ is included in a $\dun$-series, hence $\indParalite{\gpfini{M}}{\gpfini{G}}(\mathcal{E}') \subseteq \mathcal{E}$. Again, by Frobenius reciprocity, we have that $\mathcal{E}' \subseteq \resParalite{\gpfini{M}}{\gpfini{G}}(\mathcal{E})$ and the result follows.

\end{proof}

We have proved that the unipotent $\dun$-series behave well with $1$-induction. One may wonder if they also behave well with $d$-induction? This is what we are going to prove next. Actually, we will go further. We are going to check the compatibility with induction but from a $d\lprime^a$-split Levi, for certain $\lprime$.

\begin{Lem}
    \label{lemdinducdun}
    Assume that $\lprime$ satisfies \eqref{eql} and let $\gpfinialg{M}$ be a $d\lprime^a$-split Levi of $\gpfinialg{G}$ for some $a\geq 0$. Let $\mathcal{E} \subseteq \dl{\gpfini{M}}{1}$ be a $\dun$-series. Then $\inddllite{\gpfinialg{M}}{\gpfinialg{G}}(\mathcal{E})$ is included in a $\dun$-series.
\end{Lem}

\begin{proof}
    By Propositions \ref{produnserieadjoint} and \ref{produnscalaire}, we can assume that $\gpfinialg{G}$ is simple (notice that if $b$ is an integer then $dl^a/\gcd(dl^a,b)=(d/\gcd(d,b))l^{a'}$, for some $a'$ with $0 \leq a' \leq a$).

    \begin{enumerate}

        \item If $\gpfinialg{G}$ is of type $\An$, then $\dl{\gpfini{G}}{1}$ is a $\dun$-series so we have the result.
        \item If $\gpfinialg{G}$ is of type $\Bn$, $\Cn$, $\Dn$ or $\dDn$. Then, as stated is the proof of \cite[Thm. 3.2]{bmm}, the Levi $\gpfinialg{M}$ has type $\gl{n_1}^{(dl^a)}\times \cdots \times \gl{n_r}^{(dl^a)} \times \gpfinialg{H}$ where $\gpfinialg{H}$ as the same type as $\gpfinialg{G}$. Since $\mathcal{E}$ is a $\dun$-series of $\gpfini{M}$ we know that $\mathcal{E}=\dl{\gl{n_1}(q^{dl^a})}{1} \times \cdots \times \dl{\gl{n_r}(q^{dl^a})}{1} \times \mathcal{E}_{H}$, where $\mathcal{E}_{H}$ is a $\dun$-series of $\gpfini{H}$.

              Let us first assume that $d$ is odd. Thus $\mathcal{E}_{H}$ is a $1$-series in $\gpfini{H}$ by Proposition \ref{prodimpaire}. Let $\pi=\pi_1 \otimes \cdots \otimes \pi_r \otimes \pi_{H} \in \mathcal{E}$. The representation $\pi_{H}$ corresponds to a symbol of $\gpfini{H}$. Now, $\inddllite{\gpfinialg{M}}{\gpfinialg{G}}$ is the functor of $d\lprime^a$-induction. Since $d\lprime^a$ is odd, the proof of Theorem 3.2 of \cite{bmm} shows that the symbols in $\inddllite{\gpfinialg{M}}{\gpfinialg{G}}(\pi)$ are the symbols obtained from the symbol of $\pi_{H}$ by adding $d\lprime^a$-hooks. Thus all these symbols have the same 1-core which is the same as the 1-core of $\pi_{H}$. But $\mathcal{E}_{\gpfini{H}}$ is a $1$-series, so all the representations have the same 1-core, hence this is also true for the representations in $\inddllite{\gpfinialg{M}}{\gpfinialg{G}}(\mathcal{E})$. We have proved that $\inddllite{\gpfinialg{M}}{\gpfinialg{G}}(\mathcal{E})$ is included in a 1-series and thus in a $\dun$-series.

              Now, let us prove the case where $d$ is even. If $\gpfinialg{M}=\gpfinialg{G}$ there is nothing to do, so we can assume that $\gpfinialg{M}$ is proper in $\gpfinialg{G}$. The group $\gpfinialg{M}$ being a proper $d\lprime^a$-split Levi of $\gpfinialg{G}$, none of the representations in $\inddllite{\gpfinialg{M}}{\gpfinialg{G}}(\mathcal{E})$  are $d\lprime^a$-cuspidal. Since $d\lprime^a$ is even, by Proposition \ref{prod1series} we have that $\inddllite{\gpfinialg{M}}{\gpfinialg{G}}(\mathcal{E}) \subseteq \Eun{G}{d\lprime^a}$, the $d\lprime^a$-1-series of $\gpfini{G}$ containing the trivial representation. But $\kgdc{\gpfini{G}}{d\lprime^a} \leq \kgdc{\gpfini{G}}{d}$. Hence, $\Eun{G}{d\lprime^a} \subseteq \Eun{G}{d}$ and we have the result.

        \item If $\gpfinialg{G}$ is of type $\dAn$, the proof is similar as in $(2)$ using ``Ennola''-duality and the parity of $d'$ instead of $d$ (notice that $(d\lprime^a)'=(d')\lprime^a$ since $\lprime$ is odd).

        \item If $\gpfinialg{G}$ is of exceptional type, we will again use Table \ref{tabdunexp}.

              Let us start with the case $a=0$. So we are inducing $\dun$-series from $d$-split Levis. If all the unipotent $\dun$-series not containing the trivial representation are composed uniquely of $d$-cuspidal representations then we have the result. Table \ref{tabdunexp} is written in terms of 1-series. However, we can look at tables 1 and 2 from \cite{bmm} to deduce the $d$-cuspidality. In these tables, the case where the $d$-split Levi is a torus is not written, but all the induced representations from a torus of the trivial representation will be in the $\dun$-series containing the trivial. Hence, for a unipotent $\dun$-series not containing the trivial representation, it is composed uniquely of $d$-cuspidal representations if none of the representations appears in Table 2 of \cite{bmm}. This case deals with almost everyone except for $(\Esi,d=3)$, $(\Ese,d=2)$, $(\Ese,d=3)$, $(\Eh,d=2)$ and $(\Eh,d=3)$. We can then check by hand the remaining case with Table \ref{tabdunexp}. Now, we need to do $a>0$. There are only 8 cases which satisfies the hypotheses on $\lprime$, and such that $\cycl{d}$ and $\cycl{d\lprime^a}$ divide the order of $\gpfini{G}$. In all these cases, all the unipotent $\dun$-series not containing the trivial representation are composed uniquely of $d\lprime^a$-cuspidal representations, and so we have the result.

    \end{enumerate}

\end{proof}

\begin{Rem} We need the hypothesis on $\lprime$. For example, if $\lprime=3$, the group is $\tDq$, $d=1$ and $a=1$, then $\tDq[1]$ is in the induction of the trivial representation from the maximal $3$-torus but is not in the same $\dun$-series as the trivial.
\end{Rem}

We define, as in \cite{CabanesEnguehardBlocks} $E_{q,\lprime}:=\{e, \lprime | \phi_{e}(q)\}=\{d,d\lprime,d\lprime^2,\cdots,d\lprime^a,\cdots\}$, where $d$ is the order of $q$ modulo $\lprime$. A $E_{q,\lprime}$-torus is a $\fr$-stable torus of $\gpfinialg{G}$ such that its polynomial order is a product of cyclotomic polynomials in $\{\phi_{e}, e \in E_{q,\lprime}\}$. A $E_{q,\lprime}$-split Levi is then the centralizer of a $E_{q,\lprime}$-torus.

\bigskip

For a $\fr$-stable Levi subgroup of $\gpfinialg{G}$, let us denote by $Z^{\circ}(\gpfinialg{M})$ the connected centre of $\gpfinialg{M}$ and by $Z^{\circ}(\gpfinialg{M})_{\lprime}^{\fr}$ the subgroup of $Z^{\circ}(\gpfinialg{M})^{\fr}$ of $\lprime$-elements.

\begin{Lem}
    \label{lemindEqldun}
    Assume that $\lprime$ satisfies \eqref{eql}. Let $\gpfinialg{M}$ be a $E_{q,\lprime}$-split Levi of $\gpfinialg{G}$ such that $\gpfinialg{M}=\cent*{(Z^{\circ}(\gpfinialg{M})_{\lprime}^{\fr}}{\gpfinialg{G}}$. Let $\mathcal{E}$ be a unipotent $\dun$-series in $\gpfini{M}$. Then $\inddllite{\gpfinialg{M}}{\gpfinialg{G}}(\mathcal{E})$ is included in a $\dun$-series.
\end{Lem}

\begin{proof}
    We will prove the result by induction on the semi-simple rank of $\gpfinialg{G}$.

    If $\gpfinialg{M}=\gpfinialg{G}$ nothing has to be done. Now, if $\gpfinialg{M}$ is a proper Levi in $\gpfinialg{G}$, then $Z^{\circ}(\gpfinialg{M})^{\fr}_{\lprime} \not\subseteq Z(\gpfinialg{G})$. Thus there exist some $a\geq 0$ such that $Z^{\circ}(\gpfinialg{M})_{\phi_{d\lprime^a}} \not\subseteq Z(\gpfinialg{G})$, where $Z^{\circ}(\gpfinialg{M})_{\phi_{d\lprime^a}}$ is the maximal $\cycl{d\lprime^a}$-subgroup of $Z^{\circ}(\gpfinialg{M})$. Let us denote by $\gpfinialg{L}:=\cent{Z^{\circ}(\gpfinialg{M})_{\phi_{d\lprime^a}}}{\gpfinialg{G}}$ which is then a proper $d\lprime^a$-split Levi of $\gpfinialg{G}$ such that $\gpfinialg{M} \subseteq \gpfinialg{L}$. By Lemma \ref{lemdinducdun}, we know that $\inddllite{\gpfinialg{L}}{\gpfinialg{G}}$ preserves the unipotent $\dun$-series. By the induction hypothesis, $\inddllite{\gpfinialg{M}}{\gpfinialg{L}}$ preserves the unipotent $\dun$-series. Hence $\inddllite{\gpfinialg{M}}{\gpfinialg{G}} =\inddllite{\gpfinialg{L}}{\gpfinialg{G}} \circ \inddllite{\gpfinialg{M}}{\gpfinialg{L}}$ preserves the unipotent $\dun$-series.
\end{proof}

\begin{Rem} Let $\gpfinialg{M}$ be a $E_{q,\lprime}$-split Levi of $\gpfinialg{G}$. If $\ell$ is good for $\gpfinialg{G}$ and $(Z(\gpfinialg{G})/Z^{\circ}(\gpfinialg{G}))^{\fr}$ is of order prime to $\ell$, then $\gpfinialg{M}=\cent*{(Z^{\circ}(\gpfinialg{M})_{\lprime}^{\fr}}{\gpfinialg{G}}$ by \cite[Prop. 3.2]{CabanesEnguehardBlocks}.
\end{Rem}

Let $\mathcal{E}$ be a subset of $\dll{\gpfini{G}}{1}$. We denote by $\overline{\mathcal{E}}$ the smallest $\dun$-set containing $\mathcal{E}$. Thus Lemma \ref{leminducdun} and \ref{lemindEqldun} can be restated by $\overline{\inddllite{\gpfinialg{M}}{\gpfinialg{G}}(\mathcal{E})}$ is a $\dun$-series if $\gpfinialg{M}$ is a 1-split Levi or a $E_{q,\lprime}$-split Levi (satisfying the conditions of Lemma \ref{lemindEqldun}) and $\mathcal{E}$ is a unipotent $\dun$-series of $\gpfini{M}$.

\begin{Lem}
    \label{lemcommutinducdunseries}
    Let $\gpfinialg{M},\gpfinialg{K},\gpfinialg{L},\gpfinialg{G}$ be groups such that $\gpfinialg{M}$ is a 1-split Levi of $\gpfinialg{K}$, $\gpfinialg{L}$ is a 1-split Levi of $\gpfinialg{G}$, $\gpfinialg{M}$ is a $E_{q,\lprime}$-split Levi of $\gpfinialg{L}$ and $\gpfinialg{K}$ is a $E_{q,\lprime}$-split Levi of $\gpfinialg{G}$. We also assume that $\lprime$ satisfies \eqref{eql} and  that the groups $\gpfinialg{M}$ and $\gpfinialg{K}$ satisfy the condition of Lemma \ref{lemindEqldun}. If $\mathcal{E}$ is a $\dun$-series of $\gpfini{M}$ then $\overline{\inddllite{\gpfinialg{K}}{\gpfinialg{G}} ( \inddllite{\gpfinialg{M}}{\gpfinialg{K}}(\mathcal{E}))} = \overline{\inddllite{\gpfinialg{L}}{\gpfinialg{G}} ( \inddllite{\gpfinialg{M}}{\gpfinialg{L}}(\mathcal{E}))}$.
\end{Lem}

\begin{proof}
    Be Lemma \ref{lemindEqldun} and Lemma \ref{leminducdun}, we know that $\inddllite{\gpfinialg{K}}{\gpfinialg{G}} ( \inddllite{\gpfinialg{M}}{\gpfinialg{K}}(\mathcal{E}))$ is included in a $\dun$-series and $\inddllite{\gpfinialg{L}}{\gpfinialg{G}} ( \inddllite{\gpfinialg{M}}{\gpfinialg{L}}(\mathcal{E}))$ is included in a $\dun$-series. Now, since  $\inddllite{\gpfinialg{M}}{\gpfinialg{G}} = \inddllite{\gpfinialg{K}}{\gpfinialg{G}} \circ \inddllite{\gpfinialg{M}}{\gpfinialg{K}} = \inddllite{\gpfinialg{L}}{\gpfinialg{G}} \circ \inddllite{\gpfinialg{M}}{\gpfinialg{L}}$, these two $\dun$-series both contain $\inddllite{\gpfinialg{M}}{\gpfinialg{G}}(\mathcal{E})$, hence they are equal.
\end{proof}

\begin{Rem}
    Note that this lemma does not follow directly from the transitivity of the Deligne-Lusztig induction. Indeed, for a set $\mathcal{E}$, the set $\inddllite{\gpfinialg{L}}{\gpfinialg{G}} ( \inddllite{\gpfinialg{M}}{\gpfinialg{L}}(\mathcal{E}))$ might be larger that $ \inddllite{\gpfinialg{M}}{\gpfinialg{G}}(\mathcal{E})$.
\end{Rem}

\begin{Lem}
    \label{lemdualEql}
    Let $\lprime$ be a good prime. Let $\gpfinialg{L}$ be a $E_{q,\lprime}$-split Levi of $\gpfinialg{G}$ such that $\gpfinialg{L}=\cent*{(Z^{\circ}(\gpfinialg{L})_{\lprime}^{\fr}}{\gpfinialg{G}}$. Let $\gpfinialg*{L}$ be a Levi in $\gpfinialg*{G}$ in duality with $\gpfinialg{L}$. Then $\gpfinialg*{L}$ is a $E_{q,\lprime}$-split Levi of $\gpfinialg*{G}$ such that $\gpfinialg*{L}=\cent*{(Z^{\circ}(\gpfinialg*{L})_{\lprime}^{\fr}}{\gpfinialg*{G}}$.
\end{Lem}

\begin{proof}
    We adapt the proof of \cite{CabanesEngueharduni} Proposition 1.4. Let $\gpfinialg*{L}$ be a Levi in $\gpfinialg*{G}$ in duality with $\gpfinialg{L}$. Let $\gpfinialg*{M}:=\cent*{(Z^{\circ}(\gpfinialg*{L})_{\lprime}^{\fr}}{\gpfinialg*{G}}$. We have that $\gpfinialg*{L} \subseteq \gpfinialg*{M}$, and since $\lprime$ is good, $\gpfinialg*{M}$ is a Levi subgroup by \cite{CabanesEngueharduni} Proposition 2.1 (ii). We have that $Z^{\circ}(\gpfinialg*{L})_{\lprime}^{\fr}=Z^{\circ}(\gpfinialg*{M})_{\lprime}^{\fr}$.

    Let $\gpfinialg{M}$ be a dual Levi such that $\gpfinialg{L} \subseteq \gpfinialg{M} \subseteq \gpfinialg{G}$. We have that $Z^{\circ}(\gpfinialg{M})_{\lprime}^{\fr} \subseteq Z^{\circ}(\gpfinialg{L})_{\lprime}^{\fr}$. But, by \cite{carter} Proposition 4.4.5, $\card{Z^{\circ}(\gpfinialg{M})^{\fr}}=\card{Z^{\circ}(\gpfinialg*{M})^{\fr}}$ and $\card{Z^{\circ}(\gpfinialg{L})^{\fr}}=\card{Z^{\circ}(\gpfinialg*{L})^{\fr}}$, thus $Z^{\circ}(\gpfinialg{M})_{\lprime}^{\fr} = Z^{\circ}(\gpfinialg{L})_{\lprime}^{\fr}$. So, $\gpfinialg{M} \subseteq \cent*{(Z^{\circ}(\gpfinialg{L})_{\lprime}^{\fr}}{\gpfinialg{G}}=\gpfinialg{L}$ and $\gpfinialg{M}=\gpfinialg{L}$.
\end{proof}

Let $\lprime$ be a good prime for $\gpfinialg{G}$. Let $t \in \gpfini*{G}$ a semi-simple element of order a power of $\lprime$. Then $\cent*{t}{\gpfinialg*{G}}$ is a Levi subgroup, and denote by $\gpfinialg{G}(t)$ a Levi in $\gpfinialg{G}$ dual to $\cent*{t}{\gpfinialg*{G}}$.

Since $t$ is a central element of $(\cent*{t}{\gpfinialg*{G}})^{\fr}$, by \cite{digneMichel} Proposition 13.30, there exist a linear character $\hat{t} \in \Irr(\gpfini{G}(t))$ such that the tensor product with $\hat{t}$ defines a bijection from $\dl{\gpfini{G}(t)}{1}$ to $\dl{\gpfini{G}(t)}{t}$.

Let $\pi \in \dl{\gpfini{G}}{t}$. By the Jordan decomposition in the case of non connected centre (defined in \cite{lusztigdisco}) there exists $\pi_t \in \dl{\gpfini{G}(t)}{1}$ such that
\[\varepsilon_{\gpfini{G}}\varepsilon_{\gpfinialg{G}(t)} \inddllite{\gpfinialg{G}(t)}{\gpfini{G}}(\hat{t}\pi_t)=\sum_{\pi' \in C\cdot \pi} \pi',\]
where $\pi'$ runs over the orbit of $\pi$ under the action of $C:=\cent{t}{\gpfinialg*{G}}^{\fr} / (\cent*{t}{\gpfinialg*{G}})^{\fr}$, and $\varepsilon_{\gpfini{G}}$ and $\varepsilon_{\gpfini{G}(t)}$ are signs defined in \cite{lusztigdisco} Proposition 5.1.

\begin{Rem}
    \label{lemunsplitlevi}
    Let $\gpfinialg{M}$ be a 1-split Levi of $\gpfinialg{G}$ and $t \in \gpfinialg*{M}$. Then there exists $\gpfinialg{M}(t)$ a Levi in $\gpfinialg{M}$ dual to $\cent*{t}{\gpfinialg*{M}}$ which is a 1-split Levi of $\gpfinialg{G}(t)$ (the intersection of a 1-split Levi subgroup with a maximal rank subgroup is a 1-split Levi of the subgroup by \cite[Prop. 2.2]{digneMichel}).
\end{Rem}

\begin{Lem}
    \label{lemcuspsupportjordan}
    Let $\gpfinialg{M}$ be be a 1-split Levi subgroup of $\gpfinialg{G}$. Let $t$ be a semi-simple element of $\gpfini*{M}$, of order a power of $\lprime$, $\sigma \in \dl{\gpfini{M}}{t}$, and $\pi \in \dl{\gpfini{G}}{t}$ such that $\langle \pi, \inddllite{\gpfinialg{M}}{\gpfinialg{G}}(\sigma) \rangle \neq 0$.

    Let $\sigma_t \in \dl{\gpfini{M}(t)}{1}$ corresponding to $\sigma$ by the Jordan decomposition. Then, there exists $\pi_t \in  \dl{\gpfini{G}(t)}{1}$, such that $\pi_t$ corresponds to $\pi$ by the Jordan decomposition and $\langle \pi_t, \inddllite{\gpfinialg{M}(t)}{\gpfinialg{G}(t)}(\sigma_t) \rangle \neq 0$.
\end{Lem}

\begin{proof}

    Let us write $\inddllite{\gpfinialg{M}(t)}{\gpfinialg{G}(t)}(\hat{t}\sigma_t)$ as a sum of irreducible characters $\inddllite{\gpfinialg{M}(t)}{\gpfinialg{G}(t)}(\hat{t}\sigma_t) = \sum_{i} n_i \pi_i$, with $n_i \in \mathbb{N}$ and $\pi_i$ irreducible (note that since $\gpfinialg{M}(t)$ is a 1-split Levi subgroup of $\gpfinialg{G}(t)$, $\inddllite{\gpfinialg{M}(t)}{\gpfinialg{G}(t)}$ is the usual Harish-Chandra induction, thus all the $n_i$ are positive). Then we have that $ \inddllite{\gpfinialg{G}(t)}{\gpfinialg{G}}(\inddllite{\gpfinialg{M}(t)}{\gpfinialg{G}(t)}(\hat{t}\sigma_t)) = \sum_{i} n_i  \inddllite{\gpfinialg{G}(t)}{\gpfinialg{G}}(\pi_i)$. By the ``Jordan decomposition'', each $\inddllite{\gpfinialg{G}(t)}{\gpfinialg{G}}(\pi_i)$ is, up to a sign independent of $\pi_i$, a sum of irreducible characters of an orbit in $\dl{\gpfini{G}}{t}$ under the action of $\cent{t}{\gpfinialg*{G}}^{\fr} / (\cent*{t}{\gpfinialg*{G}})^{\fr}$. Hence, up to a sign, $ \inddllite{\gpfinialg{G}(t)}{\gpfinialg{G}}(\inddllite{\gpfinialg{M}(t)}{\gpfinialg{G}(t)}(\hat{t}\sigma_t))$ is a sum with positive coefficients of irreducible characters of $\gpfini{G}$.

    Now, we have that $\inddllite{\gpfinialg{G}(t)}{\gpfinialg{G}}(\inddllite{\gpfinialg{M}(t)}{\gpfinialg{G}(t)}(\hat{t}\sigma_t)) = \inddllite{\gpfinialg{M}(t)}{\gpfinialg{G}}(\hat{t}\sigma_t)= \inddllite{\gpfinialg{M}}{\gpfinialg{G}}(\inddllite{\gpfinialg{M}(t)}{\gpfinialg{M}}(\hat{t}\sigma_t))$.

    We have that $\varepsilon_{\gpfini{M}}\varepsilon_{\gpfini{M}(t)} \inddllite{\gpfinialg{M}(t)}{\gpfinialg{M}}(\hat{t}\sigma_t)=\sum_{\sigma' \in C\cdot \sigma} \sigma'$. Thus $\varepsilon_{\gpfini{M}}\varepsilon_{\gpfini{M}(t)}\inddllite{\gpfinialg{M}}{\gpfinialg{G}}(\inddllite{\gpfinialg{M}(t)}{\gpfinialg{M}}(\hat{t}\sigma_t)) = \sum_{\sigma' \in C\cdot \sigma} \inddllite{\gpfinialg{M}}{\gpfinialg{G}}(\sigma')$. Like before, $\inddllite{\gpfinialg{M}}{\gpfinialg{G}}$ is the usual Harish-Chandra induction, so it is a positive sum of characters. By hypothesis, $\langle \pi , \inddllite{\gpfinialg{M}}{\gpfinialg{G}}(\sigma) \rangle \neq 0$, thus $\langle \pi , \inddllite{\gpfinialg{M}(t)}{\gpfinialg{G}}(\hat{t}\sigma_t) \rangle \neq 0$.

    Hence, there exists $i_0$ such that $n_{i_0} \neq 0$ and $\langle \pi, \inddllite{\gpfinialg{G}(t)}{\gpfinialg{G}}(\pi_{i_0})\rangle \neq 0$. Take $\pi_t$, such that $\hat{t}\pi_t=\pi_{i_0}$. This $\pi_t$ satisfies the conditions of the lemma.
\end{proof}

We remind the reader that for $\mathcal{E}$ a subset of $\dll{\gpfini{G}}{1}$, the set $\overline{\mathcal{E}}$ denote the smallest $\dun$-set containing $\mathcal{E}$.

\begin{Pro}
    \label{proindldunseries}
    We assume that $\lprime$ satisfies \eqref{eql}. Let $\gpfinialg{M}$ be a 1-split Levi of $\gpfinialg{G}$ and $\mathcal{E} \subseteq \dl{\gpfini{M}}{1}$ a $\dun$-series. Then $\indParalite{\gpfini{M}}{\gpfini{G}}(\mathcal{E}_{\lprime}) \subseteq \overline{\indParalite{\gpfini{M}}{\gpfini{G}}(\mathcal{E})}_{\lprime}$.
\end{Pro}

\begin{proof}

    Let $\pi \in \indParalite{\gpfini{M}}{\gpfini{G}}(\mathcal{E}_{\lprime})$. By definition, there exists $\sigma \in \mathcal{E}_{\lprime}$ such that $\langle \pi , \indParalite{\gpfini{M}}{\gpfini{G}}(\sigma) \rangle \neq 0$. Let $t \in \gpfini*{M}$ be a semi-simple element of order a power of $\lprime$, such that $\sigma \in \dl{\gpfini{M}}{t}$. We also have, that $\pi \in \dl{\gpfini{G}}{t}$.

    By Lemma \ref{lemcuspsupportjordan}, we can take $\sigma_t \in \dl{\gpfini{M}(t)}{1}$ and $\pi_t \in  \dl{\gpfini{G}(t)}{1}$, such that $\sigma_t$ corresponds to $\sigma$ by the Jordan decomposition, $\pi_t$ corresponds to $\pi$ by the Jordan decomposition and $\langle \pi_t, \inddllite{\gpfinialg{M}(t)}{\gpfinialg{G}(t)}(\sigma_t) \rangle \neq 0$. Let $\sigma'$ and $\pi'$ be two irreducible characters in $\dl{\gpfini{M}}{1}$ and $\dl{\gpfini{G}}{1}$ respectively, such that $\langle \sigma', \inddllite{\gpfinialg{M}(t)}{\gpfinialg{M}}(\sigma_t)\rangle \neq 0$ and $\langle \pi', \inddllite{\gpfinialg{G}(t)}{\gpfinialg{G}}(\pi_t)\rangle \neq 0$.

    By Theorem \ref{thedescriplblock}, $\sigma'$ and $\sigma$ are in the same $\lprime$-block, and $\pi'$ and $\pi$ are also in the same $\lprime$-block. Since, $\sigma'$ and $\sigma$ are in the same $\lprime$-block and $\sigma \in \mathcal{E}_{\lprime}$, we have that $\sigma' \in \mathcal{E}$. In the same way, since $\pi'$ and $\pi$ are in the same $\lprime$-block, to prove that $\pi \in \overline{\indParalite{\gpfini{M}}{\gpfini{G}}(\mathcal{E})}_{\lprime}$ it is enough to prove that $\pi' \in \overline{\indParalite{\gpfini{M}}{\gpfini{G}}(\mathcal{E})}$.

    \bigskip

    Let $\mathcal{E}_t$ be the $\dun$-series of $\gpfini{M}(t)$ containing $\sigma_t$. The Levi $\gpfinialg{G}(t)$ is the dual of $\cent*{t}{\gpfinialg*{G}}$, hence by Lemma \ref{lemdualEql}, it is a $E_{q,\lprime}$-split Levi of $\gpfinialg{G}$ such that $\gpfinialg{G}(t)=\cent*{(Z^{\circ}(\gpfinialg{G}(t))_{\lprime}^{\fr}}{\gpfinialg{G}}$. We have the same result for the Levi $\gpfinialg{M}(t)$ of $\gpfinialg{M}$. Now $\gpfinialg{M}(t)$ is a 1-split Levi of $\gpfinialg{G}(t)$ by Remark \ref{lemunsplitlevi} and $\gpfinialg{M}$ is a 1-split Levi of $\gpfinialg{G}$.

    Let us summarize all the information about the Levis and the representations in a diagram.

    
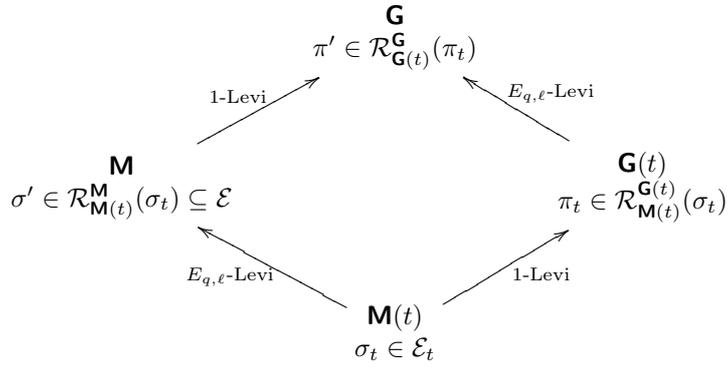
\[
        \xymatrix{
        & {\begin{array}{@{}c@{}}\gpfinialg{G}\\ \pi' \in \inddllite{\gpfinialg{G}(t)}{\gpfinialg{G}}(\pi_t)\end{array}} & \\
        {\begin{array}{@{}c@{}}\gpfinialg{M}\\ \sigma' \in \inddllite{\gpfinialg{M}(t)}{\gpfinialg{M}}(\sigma_t) \subseteq \mathcal{E}\end{array}} \ar[ru]^{1\text{-Levi}} & & {\begin{array}{@{}c@{}}\gpfinialg{G}(t)\\ \pi_t \in \inddllite{\gpfinialg{M}(t)}{\gpfinialg{G}(t)}(\sigma_t)\end{array}} \ar[lu]_{E_{q,\lprime}\text{-Levi}}\\
        & {\begin{array}{@{}c@{}}\gpfinialg{M}(t)\\ \sigma_t \in \mathcal{E}_t\end{array}} \ar[lu]^{E_{q,\lprime}\text{-Levi}} \ar[ru]_{1\text{-Levi}} & }
    \]

    We can apply Lemma \ref{lemcommutinducdunseries} which says that
    \[\overline{\inddllite{\gpfinialg{G}(t)}{\gpfinialg{G}} ( \inddllite{\gpfinialg{M}(t)}{\gpfinialg{G}(t)}(\mathcal{E}_t))} = \overline{\inddllite{\gpfinialg{M}}{\gpfinialg{G}} ( \inddllite{\gpfinialg{M}(t)}{\gpfinialg{M}}(\mathcal{E}_t))}.\]

    Now $\langle \pi_t, \inddllite{\gpfinialg{M}(t)}{\gpfinialg{G}(t)}(\sigma_t) \rangle \neq 0$, so $\pi_t \in \inddllite{\gpfinialg{M}(t)}{\gpfinialg{G}(t)}(\mathcal{E}_t)$, and $\langle \pi', \inddllite{\gpfinialg{G}(t)}{\gpfinialg{G}}(\pi_t)\rangle \neq 0$, so $\pi' \in \overline{\inddllite{\gpfinialg{G}(t)}{\gpfinialg{G}} ( \inddllite{\gpfinialg{M}(t)}{\gpfinialg{G}(t)}(\mathcal{E}_t))}$. Therefore, $\pi' \in \overline{\inddllite{\gpfinialg{M}}{\gpfinialg{G}} ( \inddllite{\gpfinialg{M}(t)}{\gpfinialg{M}}(\mathcal{E}_t))}$ (note that the use of Lemma \ref{lemcommutinducdunseries} is crucial here, as $\pi'$ may not lie in $\inddllite{\gpfinialg{M}(t)}{\gpfinialg{G}}(\mathcal{E}_t)$). Since, $\langle \sigma', \inddllite{\gpfinialg{M}(t)}{\gpfinialg{M}}(\sigma_t)\rangle \neq 0$, $\sigma' \in \inddllite{\gpfinialg{M}(t)}{\gpfinialg{M}}(\mathcal{E}_t)$, and thus $\overline{\inddllite{\gpfinialg{M}(t)}{\gpfinialg{M}}(\mathcal{E}_t)}=\mathcal{E}$. Hence, $\pi' \in \overline{\indParalite{\gpfini{M}}{\gpfini{G}}(\mathcal{E})}$, and we have the result.
\end{proof}

\begin{Pro}
    \label{proresldunseries}
    Let $\gpfinialg{M}$ be a 1-split Levi of $\gpfinialg{G}$ and $\mathcal{E} \subseteq \dl{\gpfini{G}}{1}$ a $\dun$-set. Then if $\lprime$ satisfies \eqref{eql}, we have $\resParalite{\gpfini{M}}{\gpfini{G}}(\mathcal{E}_{\lprime})=\resParalite{\gpfini{M}}{\gpfini{G}}(\mathcal{E})_{\lprime}$.
\end{Pro}

\begin{proof}
    Let $\sigma \in \resParalite{\gpfini{M}}{\gpfini{G}}(\mathcal{E}_{\lprime})$. There exists $\pi \in \mathcal{E}_{\lprime}$ such that $\langle \sigma , \resParalite{\gpfini{M}}{\gpfini{G}}(\pi) \rangle \neq 0$. Now, let $\mathcal{E}'$ be a $\dun$-series such that $\sigma \in \mathcal{E}'_{\lprime}$. By Frobenius reciprocity, $\pi \in \indParalite{\gpfini{M}}{\gpfini{G}}(\mathcal{E}'_{\lprime})$. By Proposition \ref{proindldunseries}, $\indParalite{\gpfini{M}}{\gpfini{G}}(\mathcal{E}'_{\lprime})\subseteq \overline{\indParalite{\gpfini{M}}{\gpfini{G}}(\mathcal{E}')}_{\lprime}$. Now, by Lemma \ref{leminducdun}, $\overline{\indParalite{\gpfini{M}}{\gpfini{G}}(\mathcal{E}')}$ is a $\dun$-series, so $\overline{\indParalite{\gpfini{M}}{\gpfini{G}}(\mathcal{E}')} = \mathcal{E}$. Thus $\mathcal{E}' \subseteq \resParalite{\gpfini{M}}{\gpfini{G}}(\mathcal{E})$ and  $\mathcal{E}'_{\lprime} \subseteq \resParalite{\gpfini{M}}{\gpfini{G}}(\mathcal{E})_{\lprime}$. We have that $\resParalite{\gpfini{M}}{\gpfini{G}}(\mathcal{E}_{\lprime}) \subseteq \resParalite{\gpfini{M}}{\gpfini{G}}(\mathcal{E})_{\lprime}$.

    Let us prove now the other inclusion. Let $\sigma \in \resParalite{\gpfini{M}}{\gpfini{G}}(\mathcal{E})_{\lprime}$. There exists $\mathcal{E}'$ a $\dun$-series  such that $\mathcal{E}' \subseteq \resParalite{\gpfini{M}}{\gpfini{G}}(\mathcal{E})$ and $\sigma \in \mathcal{E}'_{\lprime}$. Now, $\indParalite{\gpfini{M}}{\gpfini{G}}(\mathcal{E}') \subseteq \mathcal{E}$, so $\overline{\indParalite{\gpfini{M}}{\gpfini{G}}(\mathcal{E}') }= \mathcal{E}$. By Proposition \ref{proindldunseries}, $\indParalite{\gpfini{M}}{\gpfini{G}}(\mathcal{E}'_{\lprime}) \subseteq \overline{\indParalite{\gpfini{M}}{\gpfini{G}}(\mathcal{E}')}_{\lprime}=\mathcal{E}_{\lprime}$. Hence, $\mathcal{E}'_{\lprime} \subseteq \resParalite{\gpfini{M}}{\gpfini{G}}(\mathcal{E}_{\lprime})$, and we have the result.
\end{proof}

\section{Blocks over \texorpdfstring{$\Zl$}{Zl}}

\label{seclblocks}

Now that we have introduced and studied the $\dun$-series for finite reductive groups, we can come back to the study of $G$ a reductive group over $\kk$. The purpose of this section is to explain how to find the unipotent $\lprime$-blocks of $G$. To do that, we will combine the results of sections \ref{secBernsblock} and
\ref{secduntheo}. We will sum the 0-consistent systems of idempotents of section \ref{secBernsblock}, following what we have learnt from the $\dun$-theory, so that the idempotents that we obtain have integer coefficients. This process will end up with $\lprime$-blocks in the case of a semisimple and simply-connected group.

\subsection{Unipotent \texorpdfstring{$\lprime$}{l}-blocks}

\label{secunipolblocks}

In section \ref{secidempotenthc}, we explain how to get Bernstein blocks from 0-consistent systems constructed with unrefined depth zero types. In this section, we explain how to group those in order to get unipotent $\lprime$-blocks.

\sautintro
Let $\TGu$ be the subset of $\TG$ of pairs $(\sigma,\pi)$ with $\pi$ unipotent and $\TGl$ the subset of $\TG$ of pairs $(\sigma,\pi)$ with $\pi \in \dll{\quotred{G}{\sigma}}{1}$. We thus have that

\[\repun{\Ql}{G} = \prod_{[\mathfrak{t}] \in \TGu/{\sim}} \rep[\Ql][[\mathfrak{t}]]{G}\]
and
\[\repun{\Zl}{G} \cap \rep[\Ql]{G} = \prod_{[\mathfrak{t} ]\in \TGl/{\sim}} \rep[\Ql][[\mathfrak{t}]]{G}.\]

\begin{Rem}
    \label{remSimUnTriv}
    Let $\gpfini{G}$ be a reductive group over a finite field and $\gpfini{P}$ be a parabolic subgroup of $\gpfini{G}$ with Levi component $\gpfini{L}$. Then if $\gpfini{L}$ admits a unipotent cuspidal representation, then the association class of $\gpfini{P}$ is equal to its conjugation class (see for instance \cite[(8.2.1)]{lusztig}). Hence, the equivalence relation $\sim$ is trivial on $\TGu$. In particular, $\repun{\Ql}{G} = \prod_{\mathfrak{t} \in \TGu} \rep[\Ql][\mathfrak{t}]{G}$.
\end{Rem}

Let $T$ be a subset of $\TGl$ which is $\sim$-stable. We can associate to $T$ a system of idempotents $e_{T}$ by $e_T:=\sum_{[\mathfrak{t}] \in T/{\sim}} e_{[\mathfrak{t}]}$. We say that $T$ is $\lprime$-integral if for all $\sigma \in \bt$, $e_{T,\sigma}=\sum_{[\mathfrak{t}] \in T/{\sim}} e_{[\mathfrak{t}],\sigma}$ is in $\Zl[\quotred{G}{\sigma}]$. Thus, if $T$ is $\lprime$-integral we can form a category $\rep[\Zl][T]{G}$.

\bigskip

If $[\mathfrak{t}] \in \TG/{\sim}$, we denote by $e^{[\mathfrak{t}]}$ the idempotent in the centre of $\rep[\Ql]{G}$ associated to the category $\rep[\Ql][[\mathfrak{t}]]{G}$. We define also $e^T$ by $e^T=\sum_{[\mathfrak{t}]\in T/{\sim}} e^{[\mathfrak{t}]}$.

\begin{Lem}
    \label{lemlintegralidem}
    The idempotent $e^T$ is $\lprime$-integral if and only if $T$ is $\lprime$-integral.
\end{Lem}

\begin{proof}
    It is clear that if $T$ is $\lprime$-integral then $e^T$ is $\lprime$-integral. Let us assume that $e^T$ is $\lprime$-integral. Every $\lprime$-integral element in the centre acts on smooth functions on $G$ valued in $\Zl$ with compact support . In particular, for every $x \in \bts$, the function $e^T * e_x^+$ must be $\lprime$-integral. Let us prove that that for $[\mathfrak{t}] \in \TG/{\sim}$ we have $e^{[\mathfrak{t}]} * e_x^+=e_{[\mathfrak{t}],x}$ which will end the proof.

    Consider $V=\mathcal{C}^{\infty}_{c}(G,\Ql) e_x^+$. Since $e_x^+=\sum_{[\mathfrak{t}'] \in \TG/{\sim}} e_{[\mathfrak{t}'],x}$ by Lemma \ref{lempropidemptype}, we have a decomposition $V=\oplus_{[\mathfrak{t}'] \in \TG/{\sim}} V_{[\mathfrak{t}']}$ where $V_{[\mathfrak{t}'}] = Ve_{[\mathfrak{t}'],x}$. Now, $V_{[\mathfrak{t}]}$ is an object in $\rep[\Ql][[\mathfrak{t}]]{G}$ so $e^{[\mathfrak{t}]}$ acts as the identity on it, and if $[\mathfrak{t}'] \neq [\mathfrak{t}]$, $V_{[\mathfrak{t}']}$ is an object in $\rep[\Ql][[\mathfrak{t}']]{G}$ so is cancelled by $e^{[\mathfrak{t}]}$ which finish the proof.
\end{proof}

\begin{Pro}
    \label{prolblockminimal}
    If $G$ is semisimple and simply-connected the partition of $\TGl$ into minimal $\sim$-stable $\lprime$-integral subsets gives us the decomposition of $\repun{\Zl}{G}$ into $\lprime$-blocks.
\end{Pro}

\begin{proof}
    Since $G$ is semisimple and simply-connected, Theorem \ref{theQlblock} tells us that the idempotents $e^{[\mathfrak{t}]}$ are primitive idempotents in the centre on $\Ql$. Thus, each $\lprime$-block of $\repun{\Zl}{G}$ is associated to a $\sim$-stable subset $T \subseteq \TGl$ such that $e^T$ is $\lprime$-integral. Lemma \ref{lemlintegralidem} tells us that $T$ is $\lprime$-integral. So the $\lprime$-block decomposition of $\repun{\Zl}{G}$ gives us a partition of $\TGl$ into $\sim$-stable $\lprime$-integral subsets. But if $T$ is $\sim$-stable $\lprime$-integral, we can construct a category from $T$, so these subsets must be minimal.
\end{proof}

\begin{Def}
    Let $\lprime$ be a prime number not dividing $q$. We will say that $\lprime$ satisfies the condition \eqref{eqlpadic} if
    \begin{equation}
        \label{eqlpadic}
        \text{ For all } \sigma \in \bt,\ \lprime \text{ satisfies \eqref{eql} for } \quotred{G}{\sigma}
        \tag{$\ast \ast$}
    \end{equation}
    In other words, $\lprime$ satisfies \eqref{eqlpadic} if $\lprime$ is an odd prime number not dividing $q$, such that $\lprime \geq 5$ if a group of exceptional type ($\tDq$, $\dG$, $\Fqu$, $\Esi$, $\dEsi$, $\Ese$) is involved in a reductive quotient and $\lprime \geq 7$ if $\Eh$ is involved in a reductive quotient.
\end{Def}

Let $\lprime$ be a prime number not dividing $q$, and $d$ be the order of $q$ mod $\lprime$. Let $\mathfrak{t}$ and $\mathfrak{t}'$ be two unrefined unipotent depth zero types.

Let $\omega \in \bt$. We define $\simx{\omega}$, an equivalence relation on $\TGu$ by $\mathfrak{t} \simx{\omega} \mathfrak{t}'$ if and only if $\mathfrak{t}=\mathfrak{t'}$ or there exist $(\sigma,\pi)$ and $(\tau,\pi')$ such that $\mathfrak{t}=[\sigma,\pi]$, $\mathfrak{t}'=[\tau,\pi']$, $\omega \leq \sigma$, $\omega \leq \tau$, and $\Irr_{(\quotred{G}{\sigma},\pi)}(\quotred{G}{\omega}) \cup \Irr_{(\quotred{G}{\tau},\pi')}(\quotred{G}{\omega})$ is contained in a $\dun$-series.

\begin{Rem}
    \label{remsommetsimx}
    \begin{enumerate}
        \item If $x\leq \omega$ and $\mathfrak{t}_1 \simx{\omega} \mathfrak{t}_2$, then $\mathfrak{t}_1 \simx{x} \mathfrak{t}_2$ by Proposition \ref{leminducdun}.
        \item If $\lprime$ does not divides $\card{\quotred{G}{\omega}}$, then the $\dun$-series in $\quotred{G}{\omega}$ are just the 1-series, so $\mathfrak{t} \simx{\omega} \mathfrak{t}'$ if and only if $\mathfrak{t}=\mathfrak{t'}$.
        \item For $\mathfrak{t} \in \TGu$ and $\omega \in \bt$ fixed, the study of the $\dun$-series summarized in Theorem \ref{theresumdun} tells us exactly the set of $\mathfrak{t}'$ such that $\mathfrak{t} \simx{\omega} \mathfrak{t}'$.
    \end{enumerate}
\end{Rem}

\begin{Pro}
    \label{prod1linkedblock}
    Assume that $\lprime$ satisfies \eqref{eqlpadic}. Let $\mathfrak{t}, \mathfrak{t}' \in \TGu$ and $\omega \in \bt$ such that $\mathfrak{t} \simx{\omega} \mathfrak{t}'$. Then $\mathfrak{t}$ and $\mathfrak{t}'$ are contained in the same minimal $\sim$-stable $\lprime$-integral subset of $\TGl$.
\end{Pro}

\begin{proof}Let $T$ be the minimal $\sim$-stable $\lprime$-integral subset of $\TGl$ containing $\mathfrak{t}$. We want to show that $\mathfrak{t}' \in T$. Since $T$ is $\lprime$-integral, $e_{T,\omega} \in \Zl[\quotred{G}{\omega}]$ and can be written as a sum of primitive central $\lprime$-integral idempotents. Since $\lprime$ satisfies \eqref{eql} for $\quotred{G}{\omega}$, we have a description of them by Theorem \ref{thelblockdcuspi}. In particular, if we denote by $\mathcal{E}$ the subset of $\Irr(\quotred{G}{\omega})$ cut out by $e_{T,\omega}$, we have that $\mathcal{E} \cap \dl{\quotred{G}{\omega}}{1}$ is a $d$-set. By construction of $e_{T,\omega}$, $\mathcal{E} \cap \dl{\quotred{G}{\omega}}{1}$ is also a $1$-set so it is a $\dun$-set. Let  $(\sigma,\pi)$ and $(\tau,\pi')$ such that $\mathfrak{t}=[\sigma,\pi]$, $\mathfrak{t}'=[\tau,\pi']$ and satisfying the conditions of $\mathfrak{t} \simx{\omega} \mathfrak{t}'$. Since, $\mathfrak{t} \in T$, $\Irr_{(\quotred{G}{\sigma},\pi)}(\quotred{G}{\omega}) \subseteq \mathcal{E} \cap \dl{\quotred{G}{\omega}}{1}$. But $\Irr_{(\quotred{G}{\sigma},\pi)}(\quotred{G}{\omega}) \cup \Irr_{(\quotred{G}{\tau},\pi')}(\quotred{G}{\omega})$ is contained in a $\dun$-series so $\Irr_{(\quotred{G}{\tau},\pi')}(\quotred{G}{\omega}) \subseteq \mathcal{E} \cap \dl{\quotred{G}{\omega}}{1}$, and $\mathfrak{t}' \in T$.
\end{proof}

For $\gpfini{G}$ a finite reductive group, we denote by $\dl{\gpfini{G}}{\lprime'}$ the union of the Deligne–Lusztig series $\dl{\gpfini{G}}{s}$ with $s$ of order prime to $\lprime$. Let $\TGlp$ be the subset of $\TG$ of pairs $(\pi,\sigma)$, such that $\sigma \in \dl{\quotred{G}{\sigma}}{\lprime'}$.

\begin{Pro}
    \label{prointerlprime}
    If $T\subseteq \TG$ is $\sim$-stable $\lprime$-integral then $T \cap \TGlp \neq \emptyset$.
\end{Pro}

\begin{proof}
    Let $\sigma \in \bt$ such that $e_{T,\sigma} \neq 0$. Since $T$ is $\lprime$-integral, $e_{T,\sigma} \in \Zl[\quotred{G}{\sigma}]$. So $e_{T,\sigma}$ is a sum of primitive central idempotents in $ \Zl[\quotred{G}{\sigma}]$. Let $b$ be one of these primitive central idempotents. By \cite[Thm. 9.12]{cabanes_enguehard} there exists $\pi \in \dl{\quotred{G}{\sigma}}{\lprime'}$ such that $b\pi \neq 0$. In particular, $e_{T,\sigma} \pi \neq 0$. There exist a Levi $\gpfini{M}$ of $\quotred{G}{\sigma}$ and a cuspidal representation $\pi'$ such that $\pi \in \Irr_{(M,\pi')}(\quotred{G}{\sigma})$ and $\pi' \in \dl{\gpfini{M}}{\lprime'}$. Thus there exists $\mathfrak{t} \in \TGlp$ such that $e_{[\mathfrak{t}],\sigma} \pi \neq 0$. Moreover, $e_{[\mathfrak{t}],\sigma}$ acts as the identity on $\pi$ so $e_{T,\sigma} e_{[\mathfrak{t}],\sigma} \neq 0$. Now $e_{T,\sigma} = \sum_{[\mathfrak{t}'] \in T/{\sim}} e_{[\mathfrak{t}'],\sigma}$, so $e_{T,\sigma} e_{[\mathfrak{t]},\sigma} = \sum_{[\mathfrak{t}'] \in T/{\sim}} e_{[\mathfrak{t}'],\sigma}  e_{[\mathfrak{t}],\sigma}$. Lemma \ref{lempropidemptype} told us that if $[\mathfrak{t}] \neq [\mathfrak{t}']$ then $e_{[\mathfrak{t}'],\sigma}  e_{[\mathfrak{t]},\sigma}=0$, thus $\mathfrak{t} \in T$.

\end{proof}

Since we are interested in the unipotent blocks, we get the following corollary.

\begin{Cor}
    \label{corinterunipol}

    If $T\subseteq \TGl$ is $\sim$-stable $\lprime$-integral then $T \cap \TGu \neq \emptyset$.
\end{Cor}

\begin{proof}
    This is an immediate consequence of Proposition \ref{prointerlprime}, since $\TGlp \cap \TGl = \TGu$.
\end{proof}

Expressed in terms of $\lprime$-blocks of $\repun{\Zl}{G}$ this gives:

\begin{Cor}
    Assume that $G$ is semisimple and simply-connected. Let $R$ be an $\lprime$-block of $\repun{\Zl}{G}$. Then $R$ is characterized by the non-empty intersection $R \cap \repun{\Ql}{G}$.
\end{Cor}

\begin{proof}
    Since $G$ is semisimple and simply-connected, by Proposition \ref{prolblockminimal} $R$ is defined by $T$ a minimal $\sim$-stable $\lprime$-integral subset of $\TGl$. Now, the minimal $\sim$-stable $\lprime$-integral subsets form a partition of $\TGl$, so $T$ is uniquely determined by any of its elements. Corollary \ref{corinterunipol} tells us that $T \cap \TGu \neq \emptyset$, so $T$ is characterized by $T \cap \TGu$.
\end{proof}

\subsection{Decomposition of \texorpdfstring{$\repun{\Zl}{G}$}{Rep[Zl][1](G)} }

In this section, using the $\dun$-theory for the reductive quotient in the Bruhat-Tits building, we will define an equivalence relation on $\TGu$. When $G$ is semisimple and simply-connected, an equivalence class will exactly correspond to $T \cap \TGu$, for $T$ a minimal $\sim$-stable $\lprime$-integral set, and thus will give us a unipotent $\lprime$-block of $G$.

\bigskip

Let $\lprime$ be a prime number which satisfies \eqref{eqlpadic}, and $d$ be the order of $q$ modulo $\lprime$.

\bigskip

We define $\siml$, an equivalence relation on $\TGu$ by $\mathfrak{t} \siml \mathfrak{t}'$ if and only if there exist $\omega_1,\cdots,\omega_r \in \bt$ and $\mathfrak{t}_{1}, \cdots, \mathfrak{t}_{r-1} \in \TGu$ such that $\mathfrak{t} \simx{\omega_1} \mathfrak{t}_1\simx{\omega_2} \mathfrak{t}_2 \cdots \simx{\omega_r} \mathfrak{t}'$. We write $\eql{\mathfrak{t}}$ for the equivalence class of $\mathfrak{t}$.

\begin{Rem}
    By Remark \ref{remsommetsimx} (1), we can take in the definition $\omega_i \in \bts$.
\end{Rem}

Let $\mathfrak{t} \in \TGu$ and $\omega \in \bt$. We define $\mathcal{E}_{\eql{\mathfrak{t}},\omega}$ to be the subset of $\dl{\quotred{G}{\omega}}{1}$ cut out by $\sum _{\mathfrak{u} \in \eql{\mathfrak{t}}} e_{\mathfrak{u},\omega}$.

\begin{Lem}
    \label{lemEtld1set}
    The set $\mathcal{E}_{\eql{\mathfrak{t}},\omega}$ is a $\dun$-set in $\quotred{G}{\omega}$.
\end{Lem}

\begin{proof}
    By definition $\mathcal{E}_{\eql{\mathfrak{t}},\omega}$ is a 1-set.

    Let $(\sigma,\lambda) \in \TGu$ such that $\omega \leq \sigma$ and $\Irr_{(\quotred{G}{\sigma},\lambda)}(\quotred{G}{\omega})  \subseteq \mathcal{E}_{\eql{\mathfrak{t}},\omega}$. By construction of $\mathcal{E}_{\eql{\mathfrak{t}},\omega}$, we have that $(\sigma,\lambda) \in \eql{\mathfrak{t}}$.

    Let $\mathcal{E}_{\sigma,\lambda}$ be the $\dun$-series containing $\Irr_{(\quotred{G}{\sigma},\lambda)}(\quotred{G}{\omega})$. Let us prove that $\mathcal{E}_{\sigma,\lambda} \subseteq  \mathcal{E}_{\eql{\mathfrak{t}},\omega}$. Let $(\sigma',\lambda') \in \TGu$ such that $\omega \leq \sigma'$ and $\Irr_{(\quotred{G}{\sigma'},\lambda')}(\quotred{G}{\omega}) \subseteq  \mathcal{E}_{\sigma,\lambda}$. Then by definition, $(\sigma,\lambda) \simx{\omega} (\sigma',\lambda')$. Thus, $(\sigma,\lambda) \siml (\sigma',\lambda')$ and $(\sigma',\lambda') \in \eql{\mathfrak{t}}$. Therefore $\Irr_{(\quotred{G}{\sigma'},\lambda')}(\quotred{G}{\omega}) \subseteq  \mathcal{E}_{\eql{\mathfrak{t}},\omega}$ and $\mathcal{E}_{\sigma,\lambda} \subseteq  \mathcal{E}_{\eql{\mathfrak{t}},\omega}$.

    Since, this is true for every $(\sigma,\lambda) \in \TGu$ such that $\omega \leq \sigma$ and $\Irr_{(\quotred{G}{\sigma},\lambda)}(\quotred{G}{\omega}) \subseteq \mathcal{E}_{\eql{\mathfrak{t}},\omega}$, we get that $\mathcal{E}_{\eql{\mathfrak{t}},\omega}$ is a $\dun$-set.
\end{proof}

By Lemma \ref{lemEtld1set}, $\mathcal{E}_{\eql{\mathfrak{t}},\omega}$ is a $\dun$-set, so we can form $\mathcal{E}_{\eql{\mathfrak{t}},\omega,\lprime}$, the $\lprime$-extension of $\mathcal{E}_{\eql{\mathfrak{t}},\omega}$ as in section \ref{secdunseries}. Let $e_{\eql{\mathfrak{t}},\omega}$ be the idempotent in $\quotred{G}{\omega}$ that cuts out $\mathcal{E}_{\eql{\mathfrak{t}},\omega,\lprime}$. Since $\lprime$ satisfies \eqref{eql} for $\quotred{G}{\omega}$, Theorem \ref{thelblockdcuspi} tells us that  $e_{\eql{\mathfrak{t}},\omega}$ is $\lprime$-integral. Thus we just have defined $e_{\eql{\mathfrak{t}}} = (e_{\eql{\mathfrak{t}},\omega})_{\omega \in \bt}$ an $\lprime$-integral system of idempotents.

\begin{Pro}
    The $\lprime$-integral system of idempotent $e_{\eql{\mathfrak{t}}}$ is 0-consistent, thus defines $\rep[\Zl][\eql{\mathfrak{t}}]{G}$ a subcategory of $\repun{\Zl}{G}$.
\end{Pro}

\begin{proof}
    Since the $\mathfrak{t} \in \TGu$ are $G$-conjugacy classes, $e_{\eql{\mathfrak{t}}}$ is $G$-equivariant.

    Let $\tau, \omega \in \bt$ such that $\omega \leq \tau$. It remains to prove that $e_{\tau}^{+}e_{\eql{\mathfrak{t}},\omega}=e_{\eql{\mathfrak{t}},\tau}$.

    The idempotent $e_{\eql{\mathfrak{t}},\omega}$ is the idempotent that cuts out $\mathcal{E}_{\eql{\mathfrak{t}},\omega,\lprime}$ and  $e_{\tau}^{+}e_{\eql{\mathfrak{t}},\omega}$ is the idempotents that cuts out $\resParalite{\quotred{G}{\tau}}{\quotred{G}{\omega}}(\mathcal{E}_{\eql{\mathfrak{t}},\omega,\lprime})$.

    By Proposition \ref{proresldunseries},  $\resParalite{\quotred{G}{\tau}}{\quotred{G}{\omega}}(\mathcal{E}_{\eql{\mathfrak{t}},\omega,\lprime})= \resParalite{\quotred{G}{\tau}}{\quotred{G}{\omega}}(\mathcal{E}_{\eql{\mathfrak{t}},\omega})_{\lprime}$. But we know by definition of $\mathcal{E}_{\eql{\mathfrak{t}},\omega}$ that $\resParalite{\quotred{G}{\tau}}{\quotred{G}{\omega}}(\mathcal{E}_{\eql{\mathfrak{t}},\omega})=\mathcal{E}_{\eql{\mathfrak{t}},\tau}$. Hence, $\resParalite{\quotred{G}{\tau}}{\quotred{G}{\omega}}(\mathcal{E}_{\eql{\mathfrak{t}},\omega,\lprime})=\mathcal{E}_{\eql{\mathfrak{t}},\tau,\lprime}$.

\end{proof}

\begin{Rem}
    \label{remdecriptcat}
    By Propositions \ref{proindldunseries} and \ref{proresldunseries}, $\mathcal{E}_{\eql{\mathfrak{t}},\omega,\lprime}$ is a union of Harish-Chandra series. Hence there exists a $\sim$-stable subset $T \subseteq \TGl$ such that $e_{\eql{\mathfrak{t}}}=e_{T}$. Then Theorem \ref{thedescriplblock} gives us a description of $T$ in the following way. Let $(\sigma,\chi) \in \TGl$. Let $t$ be a semi-simple conjugacy class in $\quotred*{G}{\sigma}$ of order a power of $\lprime$, such that $\chi \in \dl{\quotred{G}{\sigma}}{t}$. Let $\quotred{G}{\sigma}(t)$ a Levi in $\quotred{G}{\sigma}$ dual to $\cent*{t}{\quotred*{G}{\sigma}}$, with $\gpfini{P}$ as a parabolic subgroup, and $\chi_t \in \dl{\quotred{G}{\sigma}(t)}{1}$ such that $ \langle \chi, \inddllite{\quotred{G}{\sigma}(t) \subseteq \gpfini{P}}{\quotred{G}{\sigma}}(\hat{t}\chi_t) \rangle \neq 0$. Let $\pi$ be an irreducible component of $\inddllite{\quotred{G}{\sigma}(t) \subseteq \gpfini{P}}{\quotred{G}{\sigma}}(\chi_t)$. Let $(\quotred{G}{\tau},\lambda)$ be the cuspidal support of $\pi$. Then $(\sigma,\chi)$ is in the subset $T$ associated with $\eql{(\tau,\lambda)}$.

\end{Rem}

\begin{The}
    Let $\lprime$ be a prime number which satisfies \eqref{eqlpadic}. Then we have a decomposition
    \[ \repun{\Zl}{G}= \prod_{\eql{\mathfrak{t}} \in \TGu/{\siml}}  \rep[\Zl][\eql{\mathfrak{t}}]{G}.\]
\end{The}

\begin{proof}
    Let $e^{\lprime}_{1}=(e^{\lprime}_{1,\sigma})_{\sigma \in \bt}$ be the 0-consistent system of idempotent that cuts out $\repun{\Zl}{G}$ (we have recalled the definition of $e^{\lprime}_{1}$ at the end of section \ref{secsystcohe}). Then the systems of idempotents $e_{\eql{\mathfrak{t}}}$, for $\eql{\mathfrak{t}} \in \TGu/{\siml}$ satisfy the following properties :
    \begin{itemize}
        \item for all $\sigma \in \bt$, $e^{\lprime}_{1,\sigma}=\sum_{\eql{\mathfrak{t}} \in \TGu/{\siml}} e_{\eql{\mathfrak{t}},\sigma}$
        \item if $\eql{\mathfrak{t}}$ and $\eql{\mathfrak{t}'}$ are two elements of $\TGu/{\siml}$ such that $\eql{\mathfrak{t}} \neq \eql{\mathfrak{t}'}$, and if $\sigma \in \bt$, then  $e_{\eql{\mathfrak{t}},\sigma} e_{\eql{\mathfrak{t}'},\sigma} = 0$.
    \end{itemize}

    With these properties, the same proof as in \cite[Prop. 2.3.5]{lanard} shows the desired result.
\end{proof}

\begin{Rem}
    \begin{enumerate}
        \item From the construction of the system of idempotents $e_{\eql{\mathfrak{t}}}$, we see that
              \[\rep[\Zl][\eql{\mathfrak{t}}]{G} \cap \repun{\Ql}{G} = \prod_{\mathfrak{u} \in \eql{\mathfrak{t}}} \rep[\Ql][\mathfrak{u}]{G} \]
        \item We also have a description of $\rep[\Zl][\eql{\mathfrak{t}}]{G} \cap \rep[\Ql]{G}$ by Remark \ref{remdecriptcat}.
    \end{enumerate}

\end{Rem}

\begin{The}
    \label{thelblockss}
    When $G$ is semisimple and simply-connected and $\lprime$ satisfies \eqref{eqlpadic}, the decomposition
    \[ \repun{\Zl}{G}= \prod_{\eql{\mathfrak{t}} \in \TGu/{\siml}}  \rep[\Zl][\eql{\mathfrak{t}}]{G},\]
    is the decomposition of $\repun{\Zl}{G}$ into $\lprime$-blocks.
\end{The}

\begin{proof}
    Let $\mathfrak{t} \in \TGu$, we want to prove that $\rep[\Zl][\eql{\mathfrak{t}}]{G}$ is an $\lprime$-block. Let $T$ be the $\sim$-stable subset of $\TGl$ which defines $\rep[\Zl][\eql{\mathfrak{t}}]{G}$. We need to prove that $T$ is a minimal $\lprime$-integral set by Proposition \ref{prolblockminimal}.

    We know that $T$ is $\lprime$-integral. By Corollary \ref{corinterunipol}, it is enough to prove that $T \cap \TGu$ is contained into a minimal $\lprime$-integral set. By construction, we have that $T \cap \TGu = \{ \mathfrak{u} \in \TGu, \mathfrak{u} \in \eql{\mathfrak{t}}\}$.

    Now, if $\mathfrak{u}$, $\mathfrak{u}'$ are two element of $\TGu$ such that $\mathfrak{u} \simx{\omega} \mathfrak{u}'$, then by Proposition \ref{prod1linkedblock}, $\mathfrak{u}$ and $\mathfrak{u}'$ are contained in the same minimal $\lprime$-integral set. Thus, if $\mathfrak{u} \siml \mathfrak{t}$,  $\mathfrak{u}$ and $\mathfrak{t}$ are contained in the same minimal $\lprime$-integral set and we have the wanted result.
\end{proof}

\subsection{Case \texorpdfstring{$\lprime=2$}{l=2} and groups of types \texorpdfstring{$\Aa$, $\Bb$, $\Cc$, $\Dd$}{A, B, C ,D} }
In this section, we examine a case of a bad prime $\lprime=2$, but when the group is good, that is all the reductive  quotients only involve types among $\Aa$, $\Bb$, $\Cc$ and $\Dd$. We will prove that the unipotent category is a $2$-block.

\begin{The}
    \label{theldeux}
    Let $G$ be a semisimple and simply-connected group such that all the reductive quotients only involve types among $\Aa$, $\Bb$, $\Cc$ and $\Dd$, and $p\neq 2$. Then $\rep[\overline{\mathbb{Z}}_{2}][1]{G}$ is a $2$-block.
\end{The}

\begin{proof}
    By Proposition \ref{prolblockminimal}, we want to prove that $\mathcal{T}_{2}^{1}(G)$ is a minimal $\sim$-stable $2$-integral set. Let $T \subseteq \mathcal{T}_{2}^{1}(G)$ be a minimal $\sim$-stable $2$-integral set. Let us prove that $\mathcal{T}_{2}^{1}(G) \subseteq T$.

    Let $\sigma \in \bt$ such that $e_{T,\sigma} \neq 0$. Since $T$ is $2$-integral, $e_{T,\sigma}$ is a sum of $2$-blocks. By \cite[Thm. 21.14]{cabanes_enguehard}, the only unipotent $2$-block of $\quotred{G}{\sigma}$ is the idempotent cuting out $\mathcal{E}_{2}(\quotred{G}{\sigma},1)$. Hence, $e_{T,\sigma}$ is this idempotent. Therefore, we get from the definition of $e_{T,\sigma}$ that for all $\mathfrak{t}=(\omega,\tau) \in \mathcal{T}_{2}^{1}(G)$, such that $\omega \leq \sigma$, we have that $\mathfrak{t} \in T$. In particular $(C,1) \in T$, where $C$ is a chamber. So, for all $\sigma \in \bt$, $e_{T,\sigma}\neq0$ and $\mathcal{T}_{2}^{1}(G) \subseteq T$.
\end{proof}

\section{Some examples}

Section \ref{seclblocks} describes the $\lprime$-blocks for a semisimple and simply-connected group thanks to the equivalence relation $\siml$ on $\TGu$. In this section, we examine some examples and make $\siml$ explicit.

\subsection{\texorpdfstring{$\lprime$}{l} divides \texorpdfstring{$q-1$}{q-1}}

When $\lprime$ divides $q-1$, hence $d=1$, the $\dun$-series are just the 1-series. In this case, $\siml$ is trivial on $\TGu$. Thus Theorem \ref{thelblockss} gives us :

\begin{Pro}
    When $G$ is semisimple and simply-connected, $\lprime$ satisfies \eqref{eqlpadic} and $\lprime$ divides $q-1$, we have a decomposition into $\lprime$-blocks
    \[ \repun{\Zl}{G}= \prod_{\mathfrak{t} \in \TGu}  \rep[\Zl][\mathfrak{t}]{G},\]
    such that $\rep[\Zl][\mathfrak{t}]{G} \cap \rep[\Ql]{G}=\rep[\Ql][\mathfrak{t}]{G}$ is a single Bernstein block.
\end{Pro}

\subsection{Blocks of \texorpdfstring{$\sl{n}$}{SLn}}

Let us make the $\lprime$-blocks of $\sl{n}$ explicit.

\begin{The}
    Let $\lprime$ be prime not dividing $q$, then $\repun{\Zl}{\sl{n}(\kk)}$ is an $\lprime$-block.
\end{The}

\begin{proof}
    If $\lprime \neq 2$, then we can apply Theorem \ref{thelblockss}. In this case, $\TGu$ is composed of only one element, the conjugacy class of $(C,1)$ where $C$ is a chamber. Hence $\repun{\Zl}{\sl{n}(\kk)}$ is an $\lprime$-block.

    If $\lprime=2$, we can apply Theorem \ref{theldeux} and $\rep[\overline{\mathbb{Z}}_{2}][1]{\sl{n}(\kk)}$ is a $2$-block.
\end{proof}

\subsection{Blocks of  \texorpdfstring{$\sp{2n}$}{Sp2n}}

In this section, we have a look at $G=\sp{2n}$. We assume in all this section that $\lprime$ does not divide $q$.

\bigskip

If $\lprime=2$, Theorem \ref{theldeux} gives us the result. So we can assume that $\lprime \neq 2$. Theorem \ref{thelblockss} tells us that to know the $\lprime$-blocks of $\sp{2n}$ we need to understand $\TGu/{\siml}$. Let us start by describing $\TGu$. The group $\sp{2m}(\res)$ has a unipotent cuspidal representation if and only if $m=s(s+1)$ for some integer $s$, and this representation is unique up to isomorphism. If $\sigma \in \bt$, then $\quotred{G}{\sigma} \simeq \gpfini{H} \times \sp{2i}(\res) \times \sp{2j}(\res)$, where $\gpfini{H}$ is a product of $\gl{m}(\res)$, and $i+j \leq n$. Hence, we have a bijection between $\TGu$ and the set $\TGs:=\{(s,s')\in \mathbb{N}^2, s(s+1)+s'(s'+1) \leq n\}$. For $(s,s') \in \TGs$ we will write $\mathfrak{t}(s,s')=(\sigma(s,s'),\pi(s,s'))$ for the corresponding element of $\TGu$.

\bigskip

Let $d$ be the order of $q$ modulo $\lprime$. The first case is when $d$ is odd. Then Proposition \ref{prodimpaire} tells us that for $\sigma \in \bt$, the unipotent $\dun$-series in $\quotred{G}{\sigma}$ are the unipotent 1-series. Hence, $\siml$ is just the trivial equivalence relation on $\TGu$. Thus we get the decomposition of $\repun{\Zl}{\sp{2n}(\kk)}$ into $\lprime$-blocks
\[\repun{\Zl}{\sp{2n}(\kk)}= \prod_{\mathfrak{t}\in \TGu} \rep[\Zl][\eql{\mathfrak{t}}]{\sp{2n}(\kk)}.\]

\bigskip

Now, we assume that $d$ is even. We want to make the equivalence relation $\siml$ on $\TGu$ explicit.

\bigskip

Let us start by finding the $\mathfrak{t} \in \TGu$ such that $\eql{\mathfrak{t}}=\{\mathfrak{t}\}$. Let $\Scusp$ be the subset of $\TGs$ of couples $(s,s')$ such that $\eql{\mathfrak{t}(s,s')}=\{\mathfrak{t}(s,s')\}$. There are $n+1$ non-conjugate vertices in $\bts$ that we denote $x_0, \cdots, x_n$, such that $\quotred{G}{x_i} \simeq \sp{2i}(\res) \times \sp{2(n-i)}(\res)$. Let $(s,s')\in \TGs$. We may assume that all the $x_i$ and $\sigma(s,s')$ are in a same chamber. Then $x_i \leq \sigma(s,s')$ if and only if $s(s+1) \leq i$ and $s'(s'+1) \leq n-i$. Hence
\[\{ x \in \bts, x \leq \sigma(s,s')\}=\{x_i, s(s+1) \leq i \leq n - s'(s'+1)\}.\]

We denote by $\Sigma_s$ the symbol corresponding to the unipotent cuspidal representation of $\sp{2s(s+1)}$. That is
\[ \Sigma_s = \begin{pmatrix}
        0 & 1 & \cdots & 2s \\
          &   &
    \end{pmatrix}.\]

\begin{Lem}
    \label{LemScuspcalcul}
    We have
    \[\Scusp=\{(s,s') \in \TGs, \left\{
        \begin{array}{ll}
            s(s+1)+s'(s'-1) > n-d/2 \\
            s'(s'+1)+s(s-1) > n-d/2
        \end{array}
        \right\}\}.\]
\end{Lem}

\begin{proof}
    By definition of $\siml$, we have that $\Scusp$ is the subset of $\TGs$ of couples $(s,s')$ such that for all $x_i \leq \sigma(s,s')$, either

    \begin{align*}
                & \left\{
        \begin{array}{ll}
            \lprime \nmid \card{\sp{2i}(\res)} \\
            \lprime \nmid \card{\sp{2(n-i)}(\res)}
        \end{array}
        \right. &
        or      &
        \left\{
        \begin{array}{ll}
            \lprime \nmid \card{\sp{2i}(\res)}    \\
            \lprime \mid \card{\sp{2(n-i)}(\res)} \\
            \defect(\Sigma_{s'}) > \kgdc{\sp{2(n-i)}(\res)}{d}
        \end{array}
        \right.           \\
        or      &
        \left\{
        \begin{array}{ll}
            \lprime \mid \card{\sp{2i}(\res)}           \\
            \defect(\Sigma_s) > \kgdc{\sp{2i}(\res)}{d} \\
            \lprime \nmid \card{\sp{2(n-i)}(\res)}
        \end{array}
        \right. &
        or      &
        \left\{
        \begin{array}{ll}
            \lprime \mid \card{\sp{2i}(\res)}           \\
            \defect(\Sigma_s) > \kgdc{\sp{2i}(\res)}{d} \\
            \lprime \mid \card{\sp{2(n-i)}(\res)}       \\
            \defect(\Sigma_{s'}) > \kgdc{\sp{2(n-i)}(\res)}{d}
        \end{array}
        \right.
    \end{align*}

    We know that $\card{\sp{2i}(\res)}=q^{n^2}\prod_{j=1}^{i}(q^{2j}-1)$ and $d$ is the order of $q$ modulo $\lprime$ (with $d$ even), hence $\lprime \mid \card{\sp{2i}(\res)}$ if and only if $d \leq 2i$. In the same way, $\lprime \mid \card{\sp{2(n-i)}(\res)}$ if and only if $d \leq 2(n-i)$.

    By definition, $\kgdc{\sp{2i}(\res)}{d}=\max\{k\geq0, k \text{ odd},(k^2-4k+3)/4\leq i-d/2\}$. So $\defect(\Sigma_s) > \kgdc{\sp{2i}(\res)}{d}$, if and only if $2s+1 > \kgdc{\sp{2i}(\res)}{d}$ if and only if $((2s+1)^2-4(2s+1)+3)/4> i-d/2$. But $((2s+1)^2-4(2s+1)+3)/4 =s(s-1)$. Hence $\defect(\Sigma_s) > \kgdc{\sp{2i}(\res)}{d}$ if and only if $s(s-1) > i-d/2$ and $\defect(\Sigma_{s'}) > \kgdc{\sp{2(n-i)}(\res)}{d}$ if and only if $s'(s'-1) > n-i-d/2$.

    So, $\Scusp$ is the set of $(s,s') \in \TGs$ such that for all $i \in \{s(s+1) ,\cdots, n - s'(s'+1)
        \}$ either
    \begin{align*}
                & \left\{
        \begin{array}{ll}
            d > 2i \\
            d > 2(n-i)
        \end{array}
        \right. &
        or      & \left\{
        \begin{array}{ll}
            d > 2i        \\
            d \leq 2(n-i) \\
            s'(s'-1) > n-i-d/2
        \end{array}
        \right.           \\
        or      & \left\{
        \begin{array}{ll}
            d \leq 2i      \\
            s(s-1) > i-d/2 \\
            d > 2(n-i)
        \end{array}
        \right. &
        or      & \left\{
        \begin{array}{ll}
            d \leq 2i      \\
            s(s-1) > i-d/2 \\
            d \leq 2(n-i)  \\
            s'(s'-1) > n-i-d/2
        \end{array}
        \right. &
    \end{align*}
    To make things clearer, let us rewrite these conditions on conditions on $i$
    \begin{align*}
                & \left\{
        \begin{array}{ll}
            i<d/2 \\
            i > n-d/2
        \end{array}
        \right. &
        or      & \left\{
        \begin{array}{ll}
            i<d/2        \\
            i \leq n-d/2 \\
            i > n-d/2-s'(s'-1)
        \end{array}
        \right.           \\
        or      & \left\{
        \begin{array}{ll}
            i \geq d/2     \\
            i < s(s-1)+d/2 \\
            i > n-d/2
        \end{array}
        \right. &
        or      & \left\{
        \begin{array}{ll}
            i \geq d/2     \\
            i < s(s-1)+d/2 \\
            i \leq n-d/2   \\
            i > n-d/2-s'(s'-1)
        \end{array}
        \right.
    \end{align*}

    Now, since $s'(s'-1)$ is positive,the conditions
    \[
        \left\{
        \begin{array}{ll}
            i<d/2 \\
            i > n-d/2
        \end{array}
        \right.
        \text{ or }
        \left\{
        \begin{array}{ll}
            i<d/2        \\
            i \leq n-d/2 \\
            i > n-d/2-s'(s'-1)
        \end{array}
        \right.\]
    are equivalent to $
        \left\{
        \begin{array}{ll}
            i<d/2 \\
            i > n-d/2-s'(s'-1)
        \end{array}
        \right.$. We also have that the conditions
    \[\left\{
        \begin{array}{ll}
            i \geq d/2     \\
            i < s(s-1)+d/2 \\
            i > n-d/2
        \end{array}
        \right.
        \text{ or }
        \left\{
        \begin{array}{ll}
            i \geq d/2     \\
            i < s(s-1)+d/2 \\
            i \leq n-d/2   \\
            i > n-d/2-s'(s'-1)
        \end{array}
        \right.\]
    are equivalent to $
        \left\{
        \begin{array}{ll}
            i \geq d/2     \\
            i < s(s-1)+d/2 \\
            i > n-d/2-s'(s'-1)
        \end{array}
        \right.
    $.

    But now, since $s(s-1)$ is positive, the conditions
    \[\left\{
        \begin{array}{ll}
            i<d/2 \\
            i > n-d/2-s'(s'-1)
        \end{array}
        \right.
        \text{ or }
        \left\{
        \begin{array}{ll}
            i \geq d/2     \\
            i < s(s-1)+d/2 \\
            i > n-d/2-s'(s'-1)
        \end{array}
        \right.
    \]
    are equivalent to $
        \left\{
        \begin{array}{ll}
            i < s(s-1)+d/2 \\
            i > n-d/2-s'(s'-1)
        \end{array}
        \right.
    $.

    Finally, we have that $\Scusp$ is the set of $(s,s') \in \TGs$ such that for all $i \in \{s(s+1) ,\cdots, n - s'(s'+1)\}$, $i < s(s-1)+d/2$ and $i > n-d/2-s'(s'-1)$, that is, it is the set of $(s,s')$ such that $n - s'(s'+1) < s(s-1)+d/2$ and $s(s+1) > n-d/2-s'(s'-1)$.
\end{proof}

We now want to prove that $\eql{\mathfrak{t}(0,0)}=\{\mathfrak{t}(s,s'), (s,s') \notin \Scusp\}$.

\begin{Pro}
    \label{prolblock00}
    Let $(s,s') \in \TGs \setminus \Scusp$. Then $\mathfrak{t}(s,s') \siml \mathfrak{t}(0,0)$.
\end{Pro}

\begin{proof}

    By definition, since $(s,s') \notin \Scusp$, there exists $i$ such that
    \[
        \left\{
        \begin{array}{ll}
            \lprime \mid \card{\sp{2i}(\res)} \\
            \defect(\Sigma_s) \leq \kgdc{\sp{2i}(\res)}{d}
        \end{array}
        \right.
        \text{ or }
        \left\{
        \begin{array}{ll}
            \lprime \mid \card{\sp{2(n-i)}(\res)} \\
            \defect(\Sigma_{s'}) \leq \kgdc{\sp{2(n-i)}(\res)}{d}
        \end{array}
        \right.
        .\]
    Let us assume for example that $
        \left\{
        \begin{array}{ll}
            \lprime \mid \card{\sp{2(n-i)}(\res)} \\
            \defect(\Sigma_{s'}) \leq \kgdc{\sp{2(n-i)}(\res)}{d}
        \end{array}
        \right.
    $ (the other case is similar). Since $\defect(\Sigma_{s'}) \leq \kgdc{\sp{2(n-i)}(\res)}{d}$ Proposition \ref{prod1series} tells us that $\mathfrak{t}(s,s') \simx{x_i} \mathfrak{t}(s,0)$.

    Let us have a look at $x_n$. First, since $s(s+1) \leq i \leq n$ then $x_n \leq \sigma(s,0)$. Now since $\lprime \mid \card{\sp{2(n-i)}(\res)}$, $i \leq n-d/2$ (like in the proof of Lemma \ref{LemScuspcalcul}). Hence, $d/2 \leq n$ and $s(s-1) \leq s(s+1) \leq i \leq n-d/2$. This can be rewritten (like in the proof of Lemma \ref{LemScuspcalcul}) as $\lprime \mid \card{\sp{2n}(\res)}$ and $\defect(\Sigma_s) \leq \kgdc{\sp{2n}(\res)}{d}$. Again, by Proposition \ref{prod1series}, $\mathfrak{t}(s,0) \simx{x_n} \mathfrak{t}(0,0)$.

    Finally, $\mathfrak{t}(s,s')\simx{x_i} \mathfrak{t}(s,0) \simx{x_n} \mathfrak{t}(0,0) $, so $\mathfrak{t}(s,s') \siml \mathfrak{t}(0,0)$.
\end{proof}

Bringing together everything that has been done so far, we get by Theorems \ref{thelblockss} and \ref{theldeux}.

\begin{The}
    Let $\lprime$ be prime not dividing $q$. Then we have the following decomposition of $\repun{\Zl}{\sp{2n}(\kk)}$ into $\lprime$-blocks :
    \begin{enumerate}
        \item If $\lprime=2$: $\rep[\overline{\mathbb{Z}}_{2}][1]{\sp{2n}(\kk)}$ is a $2$-block.
        \item If $\lprime \neq 2$. Let $d$ the order of $q$ modulo $\lprime$.
              \begin{enumerate}
                  \item if $d$ is odd,
                        \[\repun{\Zl}{\sp{2n}(\kk)}= \prod_{\mathfrak{t}\in \TGu} \rep[\Zl][\eql{\mathfrak{t}}]{\sp{2n}(\kk)}.\]
                  \item if $d$ is even,
                        \[\repun{\Zl}{\sp{2n}(\kk)}=\rep[\Zl][\eql{\mathfrak{t}(0,0)}]{\sp{2n}(\kk)} \times \prod_{(s,s')\in \Scusp} \rep[\Zl][\eql{\mathfrak{t}(s,s')}]{\sp{2n}(\kk)}.\]

              \end{enumerate}
    \end{enumerate}

\end{The}

\begin{Rem*}
    In the case $d$ odd, or $d$ even and $(s,s')\in \Scusp$, we see that the intersection of an $\lprime$-block with $\repun{\Ql}{G}$ is a Bernstein block.
\end{Rem*}

If $\lprime > n $, in the case $d$ even and $(s,s')\in \Scusp$, we can say a bit more.

\begin{Lem}
    If $\lprime > n $, $d$ is even and $(s,s')\in \Scusp$, then $\rep[\Zl][\eql{\mathfrak{t}(s,s')}]{\sp{2n}(\kk)} \cap \rep[\Ql]{G}$ is a Bernstein block.
\end{Lem}

\begin{proof}
    First of all we have that $\eql{\mathfrak{t}(s,s')}=\{\mathfrak{t}(s,s')\}$. Let $x \in \bts$ such that $x \leq \sigma(s,s')$. From the definition of $\Scusp$ and Proposition \ref{prod1series} we get that $\mathcal{E}_{\mathfrak{t}(s,s'),x}$ is composed uniquely of $d$-cuspidal representations. We use Theorem \ref{thedescriplblock} to describe $\mathcal{E}_{\mathfrak{t}(s,s'),x,\lprime}$. Let $t$ be a semi-simple conjugacy class in $\quotred*{G}{x}$ of order a power of $\lprime$. Let $\quotred{G}{x}(t)$ a Levi in $\quotred{G}{x}$ dual to $\cent*{t}{\quotred*{G}{x}}$. The Levi $\quotred{G}{x}(t)$ is then a $E_{q,\lprime}$-split Levi of $\quotred{G}{x}$. But, if $\lprime > n$, then $\lprime$ is large for $\quotred{G}{x}$ in the sense of \cite[Def. 5.1]{bmm} and therefore $E_{q,\lprime}=\{d\}$ by \cite[Prop. 5.2]{bmm}. Thus $\quotred{G}{x}(t)$ is a $d$-split Levi. Hence, if an irreducible constituent of $\inddllite{\gpfini{G}(t) \subseteq \gpfini{P}}{\gpfini{G}}(\chi_t)$, for a unipotent character $\chi_t$, is in $\mathcal{E}_{\mathfrak{t}(s,s'),x}$, then $\quotred{G}{x}(t)=\quotred{G}{x}$. Moreover, $\sp{2i}(\res)$, doesn't have any non trivial character, so Theorem \ref{thedescriplblock} tells us that $\mathcal{E}_{\mathfrak{t}(s,s'),x,\lprime}=\mathcal{E}_{\mathfrak{t}(s,s'),x}$. The system of idempotent $e_{\mathfrak{t}(s,s')}$ is therefore integral, and the proof is done.
\end{proof}

\section{Stable \texorpdfstring{$\lprime$}{l}-blocks for classical groups}

In this section, we want to find the stable depth zero $\lprime$-blocks for classical unramified groups.

\bigskip

When $G$ is a classical unramified group, we have the local Langlands correspondance (\cite{HarrisTaylor} \cite{henniart} \cite{arthur} \cite{mok} \cite{KMSW}).  The block decomposition is not compatible with the local Langlands correspondence, two irreducible representations may have the same Langlands parameter but not be in the same block. However, we can look for the ``stable'' blocks, which are the smallest direct factors subcategories stable by the local Langlands correspondence. These categories correspond to the primitive idempotents in the stable Bernstein centre, as defined in \cite{haines}. In \cite{lanard2}, there is a decomposition of the depth zero category
\[ \rep[\Ql][0]{\G} = \prod_{(\phi, \sigma) \in \Lpbm{\iner^{\Ql}}} \rep[\Ql][(\phi,\sigma)]{\G}\]
indexed by the set $\Lpbm{{\iner^{\Ql}}}$ as defined in \cite[Def. 4.4.2]{lanard2}. This decomposition satisfies the following theorem.

\begin{The}[{\cite[Thm. 4.7.5]{lanard2}}]
    \label{thelblocsstableQl}
    Let $G$ be an unramified classical group, $\ld=\Ql$ and $p \neq 2$. Then the decomposition
    \[ \rep[\Ql][0]{G} = \prod_{(\phi, \sigma) \in \Lpbm{\iner}} \rep[\Ql][(\phi,\sigma)]{G}.\]
    is the decomposition of $\rep[\Ql][0]{\G}$ into stable blocks.
\end{The}

Over $\Zl$, an analogous decomposition is defined in \cite{lanard2} :
\[ \rep[\Zl][0]{\G} = \prod_{(\phi, \sigma) \in \Lpbm{\iner^{\Zl}}} \rep[\Ql][(\phi,\sigma)]{\G}.\]
We would like to prove that for unramified classical groups, this is the decomposition of the depth zero category into stable $\lprime$-blocks, that is that these categories correspond to primitive integral idempotents in the stable Bernstein centre.

\bigskip

Let $(\phi, \sigma) \in \Lpbm{\iner^{\Ql}}$. The category $\rep[\Ql][(\phi,\sigma)]{G}$ is obtained by a consistent system of idempotents $e_{T_{(\phi,\sigma)}}$ associated to $T_{(\phi,\sigma)} \subseteq \TG$. These subsets $T_{(\phi,\sigma)}$ form a partition of $\TG$. A subset $T \subseteq \TG$ is said to be stable, if $T$ is a union of $ T_{(\phi,\sigma)}$ for $(\phi, \sigma) \in \Lpbm{\iner^{\Ql}}$.

\begin{Lem}
    \label{lemminstableset}
    If $G$ is an unramified classical group and $p \neq 2$, the stable $\lprime$-blocks correspond to the minimal $\lprime$-integral stable subsets of $\TG$.
\end{Lem}

\begin{proof}
    By Theorem \ref{thelblocsstableQl}, the primitive idempotents in the stable Bernstein centre correspond to the $T_{(\phi,\sigma)}$, hence every idempotent in the stable Bernstein centre is associated with $T$ a stable subset of $\TG$. Lemma \ref{lemlintegralidem} tells us that if the idempotent is integral then so is $T$.
\end{proof}

Let $(\phi, \sigma) \in \Lpbm{\iner^{\Ql}}$. Then \cite[Prop. 4.4.6]{lanard2} defines a bijection
\[\Gamma:\Lpbm{\iner^{\Ql}} \overset{\sim}{\longrightarrow} \ss{\gpfini*{G}}\]
where $\gpfini*{G}$ is the dual of $G$ over $\res$ and $\ss{\gpfini*{G}}$ is the set of semi-simple rational conjugacy classes in $\gpfini*{G}$.

\begin{Lem}
    \label{leminterstalbe}
    Let $(\phi, \sigma) \in \Lpbm{\iner^{\Ql}}$. Then either $T_{(\phi,\sigma)} \subseteq \TGlp$ (if $\Gamma(\phi,\sigma)$ is of order prime to $\lprime$) or $T_{(\phi,\sigma)} \cap \TGlp = \emptyset$.
\end{Lem}

\begin{proof}
    To $(\phi, \sigma) \in \Lpbm{\iner^{\Ql}}$ is attached a system of conjugacy classes on the Bruhat-Tits building. By \cite{lanard2} section 4.3, if $\Gamma(\phi,\sigma)$ is of order prime to $\lprime$, all of these conjugacy classes are of order prime to $\lprime$, and if $\Gamma(\phi,\sigma)$ is not of order prime to $\lprime$, then none of them are. Thus we get the result.
\end{proof}

\begin{Cor}
    \label{corinterminblock}
    If $T$ is an $\lprime$-integral stable set such that $T \cap \TGlp$ is a minimal stable set then $T$ is a minimal stable $\lprime$-integral set.
\end{Cor}

\begin{proof}
    If $T$ is an $\lprime$-integral stable set, then by Proposition \ref{prointerlprime} $T \cap \TGlp \neq \emptyset$ and by Lemma \ref{leminterstalbe} $T \cap \TGlp$ is a stable set. Hence $T \cap \TGlp$ is an non-empty stable set, and we get the result.
\end{proof}

\begin{The}
    Let $G$ be an unramified classical group and $p \neq 2$. Then the decomposition
    \[ \rep[\Zl][0]{G} = \prod_{(\phi, \sigma) \in \Lpbm{\iner^{\Zl}}} \rep[\Zl][(\phi,\sigma)]{G}.\]
    is the decomposition of $\rep[\Zl][0]{\G}$ into stable $\lprime$-blocks.
\end{The}

\begin{proof}
    Let $(\phi, \sigma) \in \Lpbm{\iner^{\Zl}}$. By construction, the category $\rep[\Zl][(\phi,\sigma)]{G}$ is associated with an $\lprime$-integral subset $T$ of $\TG$. By \cite[Prop. 4.5.1]{lanard2}, $T=\cup_{(\phi',\sigma')} T_{(\phi',\sigma')}$, where the union is taken over the $(\phi',\sigma')$ that are sent to $(\phi,\sigma)$ by the natural map $\Lpbm{\iner^{\Ql}} \to \Lpbm{\iner^{\Zl}}$, described in \cite{lanard2} section 4.5 (obtained by restriction from $\iner^{\Ql}$ to $\iner^{\Zl}$). In particular, the set $T$ is stable. So by Lemma \ref{lemminstableset}, it remains to prove that $T$ is minimal among the stable $\lprime$-integral sets.

    By \cite{lanard2} section 4.5, the inverse image of $(\phi,\sigma)$ by the map $\Lpbm{\iner^{\Ql}} \to \Lpbm{\iner^{\Zl}}$ is all the $(\phi',\sigma')$ such that the $\lprime$-regular part of $\Gamma(\phi',\sigma')$ is given by $\Gamma(\phi,\sigma)$.

    Hence exactly one $(\phi'_0,\sigma'_0)$ is such that $\Gamma(\phi'_0,\sigma'_0)$ is of order prime to $\lprime$. Hence by Lemma \ref{leminterstalbe}, $T \cap \TGlp=T_{(\phi'_0,\sigma'_0)}$. Since $T$ is an $\lprime$-integral stable set such that $T \cap \TGlp$ is a minimal stable set, Corollary \ref{corinterminblock} tells us that $T$ is a minimal stable $\lprime$-integral set, and that completes the proof.

\end{proof}

\bibliographystyle{amsalpha}
\bibliography{biblio}
\end{document}